\documentclass%
{article} 
\usepackage{amsmath,amsfonts,amsthm,amssymb,latexsym}
\usepackage{tikz-cd} 
\usepackage{caption}
\usepackage{subcaption} 
\usepackage{enumerate}
\usepackage{geometry}                
\usepackage{psfrag}
\usepackage{graphicx,color}
\usepackage{epstopdf}
\usepackage{comment}
\usepackage{qtree}
\usepackage{hyperref} 
\usepackage{authblk}
\excludecomment{longversion}\includecomment{shortversion}
\newtheorem{theorem}{Theorem}[section]
\newtheorem{lemma}[theorem]{Lemma}
\newtheorem{corollary}[theorem]{Corollary}
\newtheorem{proposition}[theorem]{Proposition}

\theoremstyle{definition}
\newtheorem{definition}[theorem]{Definition}

\newtheorem{algorithm}[theorem]{Algorithm}

\newtheorem{example}[theorem]{Example}
\theoremstyle{remark}
\newtheorem{remark}[theorem]{Remark}

\numberwithin{figure}{subsection}
\newenvironment{ajd1}{\noindent\color{blue} AJD }{}

\newenvironment{nb1}{\noindent\color{red} NB }{}

\newenvironment{dmr1}{\noindent\color{magenta} DMR }{}

\newenvironment{del1}{\noindent\color{red}}{}

\renewcommand{\a}{\alpha }
\renewcommand{\b}{\beta }
\newcommand{\G}{\Gamma}

\newcommand{\D}{\Delta}

\newcommand{\e}{\varepsilon }

\renewcommand{\l}{\lambda }
\renewcommand{\L}{\Lambda }
\newcommand{\Si}{\Sigma}
\renewcommand{\t}{\tau }
\renewcommand{\phi}{\varphi}
\newcommand{\Om}{\Omega}
\newcommand{\om}{\omega}
\newcommand{\pf}{\mathfrak{p}}
\newcommand{\qf}{\mathfrak{q}}
\newcommand{\FF}{\mathbb{F}}
\newcommand{\NN}{\mathbb{N}}
\newcommand{\ZZ}{\mathbb{Z}}
\newcommand{\cA}{\mathcal{A}}

\newcommand{\cC}{\mathcal{C}}
\newcommand{\cI}{\mathcal{I}}
\newcommand{\cM}{\mathcal{M}}
\newcommand{\cO}{\mathcal{O}}
\newcommand{\cR}{\mathcal{R}}
\newcommand{\cV}{\mathcal{V}}
\newcommand{\cX}{\mathcal{X}}

\newcommand{\cPC}{\mathcal{PC}}
\newcommand{\lp}{\langle}
\newcommand{\rp}{\rangle}
\newcommand{\maps}{\rightarrow}
\newcommand{\cxx}{\mathbf{x}}
\newcommand{\mA}{\lp A\rp}
\newcommand{\ml}{\lp \l\rp}
\newcommand{\xA}{\mathbf{x}\lp A\rp}
\newcommand{\XA}{X\mA}
\newcommand{\bh}{\backslash}

\newcommand{\Ms}{\cM_\psi}
\newcommand{\Mh}{\cM_\varphi}

\newcommand{\Msa}{\cM_{\psi^a}}
\newcommand{\Mhb}{\cM_{\varphi^b}}
\newcommand{\hsa}{\widehat \psi^a}
\newcommand{\hhb}{\widehat \varphi^b}

\newcommand{\be}{\begin{enumerate}}
\newcommand{\ee}{\end{enumerate}}
\newcommand{\bd}{\begin{description}}
\newcommand{\ed}{\end{description}}
\newcommand{\biz}{\begin{itemize}}
\newcommand{\eiz}{\end{itemize}}

\newcommand\RAS[2]{\overset{#1}{\underset{#2}{\Longrightarrow}}}

\newcommand\LAS[2]{\overset{#1}{\underset{#2}{\Longleftarrow}}}

\def\pond*{pond}
\def\OmegaHeading{\texorpdfstring{$\Omega$}{Omega}}
\def\VnrHeading{\texorpdfstring{$V_{n,r}$}{V(n,r)}}
\def\GnrHeading{\texorpdfstring{$G_{n,r}$}{G(n,r)}}
\title{The power conjugacy problem in Higman-Thompson groups
\footnote{This research was partially supported by the EPSRC Grant EP/K016687/1.}
\footnote{The first and third authors were supported by EPSRC doctoral training grants. Parts of this paper appear in the first author's thesis.}
}
\author[1]{Nathan Barker}
\author[2]{Andrew J.~Duncan}
\author[2]{David M.~Robertson}
\affil[1]{CMEP, 
Department for Pure Mathematics and Mathematical Statistics, 
Cambridge University, CB3 OWB, UK}
\affil[2]{School of Mathematics and Statistics, Newcastle University, Newcastle upon Tyne, NE1 7RU, UK}
\begin{document}
\maketitle
\begin{abstract}
An introduction to the universal algebra approach to Higman-Thompson groups (including Thompson's group~$V$) is given, 
following a series of lectures by Graham Higman in 1973.
In these talks, Higman outlined an algorithm for the conjugacy problem; which although essentially correct 
fails in certain cases, as we show here.  
A revised and complete version of the algorithm is written out explicitly. 
From this, we construct an algorithm for the power conjugacy problem in these groups.
Python implementations of these algorithms can be found at~\cite{DMR}. 
\end{abstract}

\section{Introduction}

In 1965, Thompson introduced the group now called  ``Thompson's group $V$'' and its subgroups $F<T$.
In doing so he gave the first examples (namely $V$ and $T$) of finitely presented, infinite simple groups (see \cite{CFP,RT}). 
McKenzie and Thompson \cite{MT} later used $V$ to construct finitely presented groups with unsolvable word problem. 
Subsequently, Galvin and Thompson (unpublished) identified $V$ with the automorphism group of an algebra $V_{2,1}$, studied by J\'onsson and Tarski \cite{JT}.
Higman \cite{Higg} generalised this construction, defining $G_{n,r}$   
as the automorphism group of a generalisation $V_{n,r}$ of $V_{2,1}$, for $n\ge 2$ and $r\ge 1$.
Moreover, Higman showed that the commutator subgroup of~$G_{n,r}$ 
is a finitely generated, infinite, simple group, for all  $n\ge 2$. 
($G_{n,r}$ is perfect when $n$ is even, and its commutator subgroup has index $2$ when $n$ is odd.)

The groups $G_{n,r}$ are the ``Higman-Thompson'' groups
of the title. There are many isomorphic groups in this set: 
in fact the algebras $V_{n,r}$ and $V_{n',r'}$ are isomorphic if and only if $n=n'$ and
$r\equiv r' \mod{n-1}$; so $G_{n,r} \cong G_{n',r'}$ if $n=n'$ and
$r\equiv r' \mod{n-1}$. Higman \cite{Higg} showed that there are infinitely many non-isomorphic groups $G_{n,r}$
and gave necessary conditions for such groups  to be isomorphic. Recently Pardo \cite{P} completed 
the isomorphism classification, showing that Higman's necessary conditions are also sufficient: that is 
$G_{n,r}\cong G_{n',r'}$ if and only if $n=n'$ and $\gcd(n-1,r)=\gcd(n'-1,r')$. 
Higman-Thompson groups have been much studied and further generalised:  
we refer to 
 \cite{CFP,B,Bleaketal,MPN,DMP,BCR} for example.

In this paper we consider the conjugacy and power conjugacy problems in Higman-Thompson groups. 
We use Higman's method, describing the groups $G_{n,r}$ in terms of universal algebra. This allows us to  give
a detailed description of the algorithm for the conjugacy problem; and to uncover a gap in the original algorithm proposed by Higman. 
To be precise, Lemma 9.6 of \cite{Higg} is false, and consequently the ``orbit sharing'' algorithm in \cite{Higg} does not always detect
elements in the same orbit of an automorphism.  
The orbit sharing algorithm is crucial to the algorithm for conjugacy given in \cite{Higg}, which
 may fail to recognise that a pair of elements of $G_{n,r}$ are conjugate. Fortunately it is not difficult to complete the algorithm.   
 We then extend 
 these results to construct an algorithm for the power conjugacy problem: that is, 
given elements $g,h$ in a group  $G$ decide whether or not there exist 
 non-zero integers $a$ and $b$ such that $g^a$ is conjugate to $h^b$. 

The power conjugacy problem though less well known than the conjugacy
problem, already occurs as one of the problems in the hierarchy of decision problems 
studied by Lipschutz and Miller \cite{LM71}. The problem has been shown to be decidable in, for example,  
certain HNN-extensions and free products with cyclic amalgamation \cite{AS74,C77}, in certain one-relator groups \cite{P08}, 
in Artin groups of extra large type \cite{BK08}, in groups with small cancellation  conditions C(3) and T(6) \cite{Bez14} and in free-by-cyclic
groups \cite{BMOV}. 
Cryptographic protocols based on the power conjugacy search problem have been proposed, see for example \cite{KA09}, although these may
be susceptible to attack by quantum computer \cite{Fes14}. 

The third author has implemented the algorithms described in this paper in Python \cite{DMR}. 
In fact it was the process of testing this implementation which uncovered the existence of an orbit unrecognised in \cite{Higg}; 
and it became evident that
the algorithms of \cite{Higg} were incomplete.

Note that other approaches to  algorithmic problems 
in $G_{n,r}$ have been developed. For example \cite{Olga} proposes an algorithm for  the conjugacy problem in $G_{2,1}$  
based on the the revealing tree pairs of Brin \cite{B}. 
In \cite{Bleaketal} the same methods are used to study the centralisers of elements of $G_{n,1}$ for $n\ge 2$. 
Again 
Belk and Matucci \cite{BM} gave a solution to the conjugacy problem in 
$G_{2,1}$ based on strand diagrams. In another direction, Higman's methods were used 
by Brown \cite{Br} to show that all the Higman-Thompson groups are of type $FP_\infty$.
This discussion of finiteness properties has been
extended to generalisations of Higman-Thompson groups,
by Martinez-Perez and Nucinkis \cite{MPN}.
 
In detail the contents of the paper are as follows. In order to make this account 
self-contained,
we begin with an introduction to universal algebra. 
Section \ref{chapterUA} outlines the universal algebra required, following Cohn's account \cite{cohnAlgebra3}. 
In Section \ref{omegaAlgebra} we introduce $\Om$-algebras; that is universal algebras with signature $\Om$. Sections
\ref{sec:cong} and \ref{freealgebrasIntro}  cover quotients of $\Om$-algebras, varieties of $\Om$-algebras 
and free $\Om$-algebras. We use this 
machinery in Section \ref{v21freealgebra} to define the algebras $V_{n,r}$ and establish 
their basic properties, following
the exposition of \cite{Higg}.

 The groups $G_{n,r}$ are defined in Section \ref{HigmanThompsonG21} as the automorphism groups of $V_{n,r}$. 
We represent elements of $G_{n,r}$ as bijections between carefully chosen generating
sets of the algebras  $V_{n,r}$. This is done in two stages, beginning with the semi-normal forms of 
Section \ref{SNF}. There are many ways of representing a given automorphism in \emph{semi-normal} form, but
 in Section \ref{sec:qnf} it is shown that this representation may be refined to a unique \emph{quasi-normal} form.
Furthermore, an algorithm is given which takes an automorphism and produces a quasi-normal form representation.
 
The solution to the conjugacy problem is based on an analysis of certain orbits of automorphisms in 
quasi-normal form, and we give a full account of this analysis in Sections \ref{SNF} and \ref{sec:qnf}. Here we follow 
\cite{Higg} except that, as pointed out above, there exist orbits of types not recognised there, which give   
 automorphisms in quasi-normal form a richer structure, as described here.  
 
Section \ref{VPVRI} contains the algorithm for the conjugacy problem. This involves 
breaking an automorphism down into well-behaved parts. It is shown that every element of $G_{n,r}$ decomposes
into factors which are called \emph{periodic} and \emph{regular infinite} parts. The conjugacy problem 
for periodic and regular infinite components are solved separately and then the results recombined. The decomposition
into these parts is the subject of 
Section \ref{sec:pri} and here we give the main algorithm for the conjugacy problem, 
Algorithm \ref{conjAlgorithm}. This algorithm depends on algorithms for periodic and regular infinite
automorphisms: namely Algorithm \ref{alg:periodic} in Section \ref{subsectionP} and 
Algorithm \ref{alg:reginf} in Section \ref{subsectionRI}.

  In Section \ref{PowerConj} we turn to the 
power conjugacy problem.  In the version considered here the problem is, given $g,h\in G_{n,r}$  to find all pairs
of non-zero integers $(a,b)$ such that $g^a$ is conjugate to $h^b$. Again the problem splits into the periodic and regular infinite parts. The 
periodic part is straightforward, and reduces to the conjugacy problem; see Section \ref{torPower}. 
The algorithm for power conjugacy of regular infinite elements is Algorithm \ref{powerconjAlgorithm}, in Section 
\ref{powerconjAlgorithmSec} and gives the main result of the paper  Theorem \ref{PCG}: that the power conjugacy 
problem is solvable. On input $g,h \in G_{n,r}$ the algorithm returns a (possibly empty) set $S$ consisting of all 
pairs of integers $(a,b)$ such that $g^a$ and $h^b$ are conjugate; as well as a conjugator, for each pair.

In outline, the main steps of the algorithm for the (power-)conjugacy problem are: 
\begin{itemize}
	\item Lemma~\ref{lem:qnf} which computes the quasi-normal basis of a given automorphism;
	\item Lemma~\ref{lem:compcheck},  the `component-sharing test', as in Higman's original algorithm;
        \item Lemma~\ref{9.7H}, the `orbit-sharing test', which recognises and combines components which belong to a single orbit;
	\item Algorithm~\ref{conjAlgorithm} which is Higman's solution to the conjugacy problem; and 
	\item Algorithm~\ref{powerconjAlgorithm} which determines if two automorphisms are power conjugate.
\end{itemize}

The examples given throughout the text are used as examples in \cite{DMR}, from where these and other examples may be run through the third author's implementations of the algorithms.
To find Example $x$.$y$ in \cite{DMR}, follow the instructions in the documentation to install the program; then run

{\ttfamily
	\indent >>> from thompson import * \\
	\indent >>> f = load\_example('example\_x\_y')
}

\noindent in a Python session.
The automorphism will then be available as the Python object \texttt{f}.

\subsection*{Acknowledgements}
The authors thank  Collin Bleak, for suggesting work on conjugacy problems in Thompson's group, and 
Claas R\"over, who suggested the use of Higman's approach, 
and made many improvements to earlier versions of this work. We thank  Sarah Rees for overall guidance and constructive suggestions. We thank 
Steve Pride for pointing us in the direction of 
the power conjugacy problem and Francesco Matucci, Jos\'e Burillo and Matt Brin for helpful conversations.

\section{Universal Algebra}\label{chapterUA}

\subsection{\OmegaHeading-algebras}\label{omegaAlgebra}
In this section we review enough universal algebra to underpin the construction of the
 Higman-Thompson groups in later sections. We follow  \cite{cohnAlgebra3}.

\begin{definition}
An \emph{operator domain} consists of a set $\Omega$ and a mapping $a:\Omega \to \mathbb{N}_0$. The elements of $\Omega$ are called \emph{operators}. If $\omega \in \Omega$, then $a(\omega)$ is called 
the \emph{arity} of $\omega$. We shall write $\Omega(n)=\{\omega\in\Omega \mid a(\omega)=n\}$, and refer to the members of $\Omega(n)$ as $n$\emph{-ary operations}.

An \emph{algebra} with \emph{operator domain} (or \emph{signature}) $\Om$
   consists of a set $S$, called the \emph{carrier} of the algebra, and  
a family of maps $\{\phi_\om\}_{\om\in \Om}$ indexed by $\Omega$, such that for $\omega\in \Omega(n)$, 
$\phi_\om$ is a map from $S^n$ to $S$. 
\end{definition}
Following \cite{cohnAlgebra3} we suppress all mention of  
the maps $\phi_\om$, 
identifying $\phi_\om$ with $\om$, and referring  to any algebra with carrier $S$ and 
operator domain $\Omega$ 
as an $\Om$\emph{-algebra},
which we denote by $(S,\Om)$.
For example, a group $(G,\cdot, ^{-1}, 1)$ is a $\Om$-algebra with operator domain
$\{\cdot, ^{-1}, 1\}$ and carrier $G$, where~$\cdot$ is binary, $^{-1}$ is unary and 
$1$ is a constant.
For this to describe a group, certain laws must hold between these operations, i.e.\ the group axioms.

Given an $\Omega$-algebra $(S,\Omega)$ and $f\in \Omega(n)$, we  write $s_1\cdots s_{n}f$ for
the image of the $n$-tuple $(s_1,\ldots,s_{n})\in S^n$ under $f$. 
We say that a subset $T\subseteq S$ is \emph{closed under the operations of} $\Omega$ 
(or that $T$ is $\Om$\emph{-closed}) if, for all 
$n\ge 0$, for all $f$ 
in $\Omega(n)$ and for all $s_1,\ldots ,s_{n}\in T$ 
the element $s_1\cdots s_{n}f$ is also an element of $T$. Indeed, if $T$ is a subset of $S$ then 
$T$ is $\Om$-closed if and only if  
$(T,\Om)$ is an $\Om$-algebra: which brings us to the next definition.

\begin{definition}\label{subalgebrafree}
Given an $\Omega$-algebra $(S,\Omega)$, an $\Omega$\emph{-subalgebra} is an $\Omega$-algebra $(T,\Omega)$ 
whose carrier $T$ is a subset of $S$.
\end{definition}

The intersection of any family of subalgebras is again a subalgebra. 
Hence, for any subset $X$ of the set $S$ we may define the subalgebra $\langle X\rangle$ \emph{generated} by
$X$ to  be the intersection of all subalgebras containing $X$. 
 The subalgebra  $\langle X\rangle$ may also be defined recursively: that is $\langle X \rangle$ is the
subset of $S$ such that 
(i) $X\subseteq \lp X\rp$, (ii) if $y_1,\ldots ,y_n\in \langle X\rangle$ then $y_1\cdots y_nf\in \langle X\rangle$, 
for all $f\in \Om(n)$ and (iii) if $s$ does not satisfy (i) or (ii) then $s$ 
does not belong to $\langle X\rangle$.  Loosely speaking we might say that 
$\lp X\rp$ is obtained from $X$ by applying a finite sequence of operations
of $\Om$.
 If the subalgebra generated by $X$ is the whole of $S$, then $X$ is called a \emph{generating set} for 
$(S,\Om)$.

A mapping $ g: \mathcal{A}\to \mathcal{B}$ between two $\Omega$-algebras $\mathcal{A}=(S,\Omega),\mathcal{B}=(S',\Omega)$ is said to be \emph{compatible} with $f\in\Omega(n)$ if, for all $s_1,\ldots,s_{n}\in S$,
\[(s_1g)\cdots (s_{n}g)f=(s_1\cdots s_{n}f)g.\]
If $g$ is compatible with each $f\in\Omega$, it is called a \emph{homomorphism} from $\mathcal{A}=(S,\Omega)$ to $\mathcal{B}=(S',\Omega)$. If a homomorphism $g$ from $\mathcal{A}$ to $\mathcal{B}$ has an inverse $g^{-1}$ which is again a homomorphism, $g$ is called an \emph{isomorphism} and then the $\Omega$-algebras $\mathcal{A}=(S,\Omega),\mathcal{B}=(S',\Omega)$ are said to be \emph{isomorphic}. 
An isomorphism of an algebra $\mathcal{A}=(S,\Omega)$ with itself is called an \emph{automorphism} and a homomorphism of an algebra into itself is called an endomorphism. 
A homomorphism is determined once the images of a generating set are fixed.
\begin{proposition}[{\cite[Proposition 1.1]{cohnAlgebra3}}]
Let $g,h:\mathcal{A}\to \mathcal{B}$ be two homomorphisms between $\Omega$-algebras $\mathcal{A}=(S,\Omega),\mathcal{B}=(S',\Omega)$. If $g$ and $h$ agree on a generating set for $\mathcal{A}$, then they are equal. 
\end{proposition}


From a family $\{\mathcal{A}_i\}_{i=1}^{m}$ ($\mathcal{A}_i=(S_i,\Omega)$) of $\Omega$-algebras we can form the \emph{direct product} $P=\prod_{i=1}^{m} \mathcal{A}_i$ of $\Omega$-algebras. Its set is the Cartesian product $S$ of the $S_i$, and the operations are carried out component wise. Thus, if $\pi_i:S\to S_i$ are the projections from the product to the factors then any $f\in \Omega$ of arity $n$ is defined on $S^n$ by the equation 
\[(p_1\cdots p_nf)\pi_i=(p_1\pi_i)\cdots (p_n\pi_i)f,\]
where $p_i\in S$.

 Let $\mathcal{C}$ be a class of $\Omega$-algebras, whose elements we will call $\mathcal{C}$-algebras. By a \emph{free $\mathcal{C}$-algebra} on a set $X$ we mean a $\mathcal{C}$-algebra $F$ with the following universal property.

\begin{quote}
There is a mapping $\mu : X\to F$ such that every mapping $f:X\to \mathcal{A}$ into a $\mathcal{C}$-algebra $\mathcal{A}$ can be factored uniquely by $\mu$ to give a homomorphism from $F$ to $\mathcal{A}$, \emph{i.e.} there exists a unique homomorphism $f':F\to \mathcal{A}$ such that $\mu f'=f$.
\end{quote}

In this case we say that $X$ is a \emph{free generating set} or a \emph{basis} for $F$. If $X$ is a subset of $F$ then we shall always assume that
$\mu$ is the inclusion map.  
Not every class has free algebras, but they do  exist in the  class under consideration here 
(see Proposition \ref{cohnPropFA}).  

A \emph{free product} is defined similarly,  replacing the set $X$ by a collection of $\mathcal{C}$ algebras. Given an indexing set $I$ and for each $i\in I$ an $\Omega$ algebra $A_i$ from $\mathcal{C}$ the free product $\mathcal{A}$ of $\{A_i\}_{i\in I}$, written $\mathcal{A}=\ast_{i\in I}\mathcal{A}_i$, is an $\Omega$-algebra in $\mathcal{C}$ satisfying the following property.
\begin{quote}
There exist homomorphisms $\mu_i : A_i\to \mathcal{A}$, for all $i\in I$, such that for any $\Omega$-algebra $\mathcal{B}$ and homomorphisms $f_i:A_i\to \mathcal{B}$, for all $i\in I$, there exists a unique homomorphism $f':\mathcal{A} \to \mathcal{B}$ such that $\mu_i f'=f_i$, for all $i$.
\end{quote}
Given collections $\{\mathcal{A}_i\}_{i\in I}$ and $\{\mathcal{B}_i\}_{i\in I}$ of $\Omega$-algebras 
such that there exist free products 
$\mathcal{A}=\ast_{i\in I} \mathcal{A}_i$ and $\mathcal{B}=\ast_{i\in I} \mathcal{B}_i$, 
then, by definition, there exist homomorphisms 
$\mu_i:\mathcal{A}_i\to \mathcal{A}$ and  $\mu'_i:\mathcal{B}_i\to \mathcal{B}$,  
for all $i\in I$. In this case, given  homomorphisms $f_i:\mathcal{A}_i\to \mathcal{B}_i$,  
for all $i\in I$, the composition $f_i\mu'_i$ is a homomorphism from $\mathcal{A}_i$ 
to $\mathcal{B}$,  
 so there exists a unique homomorphism $f':\mathcal{A}\to \mathcal{B}$,  
with $\mu_if'=f_i\mu'_i$, for all $i\in I$. We denote $f'$ by $\ast_{i\in I}f_i$.

\subsection{Congruence on an \OmegaHeading-algebra}\label{sec:cong}
A \emph{relation} between  two sets $S$ and $R$ is defined to be a subset of the Cartesian product $S\times R$. 
A \emph{mapping} $f: S\to R$ is a relation $\Gamma_f\subset S\times R$ with the properties that 
for each $s\in S$ there exists $r\in R$ such that $(s,r)\in \Gamma_f$ (everywhere defined) and 
if $(s,r), (s,r')\in \Gamma_f$ then $r=r'$ (single valued).
 A relation $\Gamma\subset S\times R$ has an \emph{inverse} $\G^{-1}$, defined by  
 \[\Gamma^{-1}=\{(r,s)\in R\times S \mid (s,r)\in \Gamma\};\]
 and if  $\Delta \subset R\times T$ is a relation  then the composition  $\Gamma \circ \Delta$ of $\G$ and $\D$ is 
defined by 
 \[\Gamma \circ \Delta =\{(s,t)\in S\times T \mid (s,x)\in \Gamma \ \text{and} \ (x,t)\in\Delta \ \text{for some} \ x\in R\}.\]
 If $\Gamma\subset S\times R$ and $S'\subset S$ we define
 \[S'\Gamma=\{r\in R \mid (s,r)\in \Gamma \ \text{for some $s\in S'$}\}.\] 
 Given a set $S$ the  \emph{identity relation} $1_S=\{(s,s) \mid s\in S\}$ and the \emph{universal relation} $S^2=\{(s,s') \mid s,s'\in S\}$ always exist. 

An \emph{equivalence} on a set $S$ is a subset $\Gamma$ of $S^2$ with the properties
$\Gamma \circ \Gamma\subset \Gamma$ (transitivity): 
$\Gamma^{-1}=\Gamma$ (symmetry) and 
$1_S\subseteq \Gamma$ (reflexivity).
 The \emph{equivalence class} of $s\in S$ is $\{s'\in S \mid (s,s')\in \Gamma\}=\{s\}\Gamma$.  
 Given any subset $U$ of $S\times S$, the \emph{equivalence generated} by $U$ is 
 \[E=\bigcap\{V\subseteq S\times S \mid V \ \text{is an equivalence and $U\subseteq V$}\};\]
that is, the smallest equivalence $E$ on $S$ containing $U$. 
It follows that $E$ is 
\[\{(a,b)\in S\times S\mid \text{there exists $a_0,\ldots,a_n$ such that $a_0=a$, $a_n=b$ and $(a_i,a_{i+1})\in U$}\}.\]

Of particular interest in   the study of $\Omega$-algebras are  relations which are also subalgebras. 
Firstly, if $\mathcal{A}=(S,\Omega)$ and $\mathcal{B}=(R,\Omega)$ are $\Omega$-algebras 
and $\Gamma\subset S\times R$ is a relation which is closed under the operations of $\Omega$, 
as defined in $\mathcal{A}\times \mathcal{B}$, then $(\Gamma,\Omega)$ is a subalgebra of $\mathcal{A}\times \mathcal{B}$. 
In this case we abuse notation and say $\Gamma$ is a subalgebra of $\mathcal{A}\times \mathcal{B}$.

\begin{lemma}[{\cite[Lemma 2.1, Chapter 1]{cohnAlgebra3}}]
Let $\mathcal{A}, \mathcal{B}, \mathcal{C}$ be $\Omega$-algebras and let $\Gamma, \Delta$ be subalgebras of $\mathcal{A}\times \mathcal{B}, \mathcal{B}\times \mathcal{C}$ respectively. Then $\Gamma^{-1}$ is a subalgebra of $\mathcal{B}\times \mathcal{A}$, $\Gamma\circ \Delta$ is a subalgebra of $\mathcal{A}\times \mathcal{C}$ and if $\mathcal{A}'$ is a subalgebra of $\mathcal{A}$, with
carrier $S'\subseteq S$, then $(S'\Gamma,\Om)$ is a subalgebra of $\mathcal{B}$.
\end{lemma} 
 
 Let $S$ and $T$ be sets and $f: S\to T$ a mapping between them. The \emph{image} of $f$ is defined as $S\Gamma_f$,
 and  the \emph{kernel} of $f$ is defined as 
 \[\operatorname{ker} f=\{(x,y)\in S^2 \mid xf=yf\}. \]
 The latter is an equivalence on $S$; the equivalence classes are the inverse images of elements in the image  
(sometimes called the fibres of $f$).
 
 \begin{example}[Groups]
Given a group homomorphism $f:G\to H$, the (group-theoretic) kernel of $f$ is a normal subgroup $N$; 
and the different cosets of $N$ in $G$ are the fibres of $f$.  
 So, the equivalence classes of $\text{ker} f$, in the definition above, are the cosets of $N$ in $G$.
 \end{example}

A \emph{congruence} on an $\Omega$-algebra $\mathcal{A}=(S,\Omega)$ is an equivalence on $S$ which is also a subalgebra of $\mathcal{A}^2$ \emph{i.e.} an equivalence $\Gamma\subset S\times S$ which is $\Omega$-closed. From the above, $1_{\mathcal{A}}$ and $\mathcal{A}^2$ are congruences on $\mathcal{A}$. 
Given any subset $U\subseteq S\times S$ the \emph{congruence generated by $U$} is  
\[C=\bigcap\{V\subseteq S\times S \mid  \text{$V$ is a congruence and $U\subseteq V$}\}.\] 
It follows that $C$ is the smallest congruence on $\mathcal{A}$ containing $U$. 

Let $\mathcal{A}$ be an $\Omega$-algebra. By definition a congruence is an equivalence which admits the operations $\omega$ ($\omega\in \Omega$). Now each $n$-ary operator $\omega$ defines an $n$-ary operation on $\mathcal{A}$:
\begin{equation}\label{law1}
(a_1,\ldots,a_n)\mapsto a_1\cdots a_n\omega \ \text{for $a_1,\ldots ,a_n\in \mathcal{A}$}.
\end{equation}
By giving fixed values in $\mathcal{A}$ to some of the arguments, we obtain $r$-ary operations for $r\leq n$. 
In particular, if we fix all the $a_j$ except one, say the $i$th, 
we obtain, for any $n-1$ fixed elements $a_1,\ldots ,a_{n-1}\in \mathcal{A}$, a unary operation 
\begin{equation}\label{rule1}
x\mapsto a_1\cdots a_{i-1}xa_i\cdots a_{n-1}\omega;
\end{equation}
and this applies for  all $i\in \{1,\ldots,n\}$.
We say that the operation (\ref{rule1}) 
is an \emph{elementary translation} (derived from $\Om$ by specialisation in $\mathcal{A}$). 
 Given a finite sequence $\t_1,\ldots, \t_n$ of elementary transformations the composition
$\t=\t_1\circ \cdots \circ \t_n$ is also a unary operation on $\cA$, which we call a \emph{translation}.
(In particular we allow $n=0$ in this definition, so the identity map on $\cA$ is a translation.)   

\begin{proposition}[{\cite[Proposition 6.1, Chapter6]{Cohn}}]\label{equivTOcong}
An equivalence $\mathfrak{q}$ on an $\Omega$-algebra $\mathcal{A}$ is a congruence if and only if it 
is closed under  all translations. More precisely, a congruence is closed under all translations, 
while any equivalence which is closed under all elementary translations is a congruence.
\end{proposition}

\begin{remark}
If $U\subseteq S\times S$, then the congruence generated by $U$ can be seen to consist of pairs $(a,b)\in S\times S$ such that there 
exist $m\ge 0$,  $a_0,\ldots  ,a_m \in S$,  and a translation $\t$ with    
\begin{itemize}
\item $a_0=a$, $a_m=b$ and 
\item $(a_i,a_{i+1})=(u_i\tau,u_{i+1}\tau)$ 
\end{itemize}
where  either $(u_i,u_{i+1})\in U$,  $(u_{i+1},u_i)\in U$ or $u_i=u_{i+1}$.  
That is,  there exist $s_1,\ldots ,s_{n-1}\in S$,  $u_0,\ldots , u_m\in S$, 
and $\omega \in \Omega(n)$ such that $(u_i,u_{i+1})\in U\cup U^{-1}\cup 1_S$ and setting
\[a_i=(s_1,\ldots,s_{j-1},u_i,s_j,\ldots,s_{n-1})\omega,\]
for $0\le i\le m$, we have $a=a_0$ and $b=a_m$.
\end{remark}

The next two theorems explain the significance of congruences for $\Omega$-algebras and will be used in the following section on free algebras and varieties. 
 \begin{theorem}[{\cite[Theorem 2.2, Chapter 1]{cohnAlgebra3}}]
 Let $g: \mathcal{A}\to \mathcal{B}$ be a homomorphism of $\Omega$-algebras. Then the image of $g$ is a subalgebra of $\mathcal{B}$ 
and the  kernel of $g$ is a congruence on $\mathcal{A}$.
 \end{theorem}
  \begin{theorem}[{\cite[Theorem 2.3, Chapter 1]{cohnAlgebra3}}]
 Let $\mathcal{A}$ be an $\Omega$-algebra and $\mathfrak{q}$ a congruence on $\mathcal{A}$. Then, there exists a unique $\Omega$-algebra, denoted $\mathcal{A}/\mathfrak{q}$, with carrier the set of all $\mathfrak{q}$-classes such that the natural mapping $\nu: \mathcal{A}\to \mathcal{A}/\mathfrak{q}$ is a homomorphism.
 \end{theorem}
The homomorphism $\nu$ in the previous theorem, which maps an element $s$ of the carrier of $\cA$ to its $\qf$-equivalence class,
is called the \emph{natural homomorphism} from $\cA$ to $\cA/\qf$. 
 The algebra $\mathcal{A}/\mathfrak{q}$ is called the \emph{quotient algebra} of $\mathcal{A}$ by $\mathfrak{q}$. 
 
 \begin{example}
 Given a group $G$ and a normal subgroup $N$ of $G$, the natural mapping $G\to G/N$ is a homomorphism.
 \end{example}
 
\subsection{Free algebras and varieties}\label{freealgebrasIntro}

Let $X=\{x_1,x_2,\ldots\}$ be a non-empty, finite or countably enumerable set, called an \emph{alphabet}, and $\Omega$ an operator domain, 
with $\Omega\cap X=\emptyset$. 
We define an $\Omega$-algebra as follows.  An $\Omega$\emph{-row} in $X$ is a finite sequence of elements of $\Omega\cup X$. 
The set of all $\Omega$-rows in $X$ is denoted $W(\Om;X)$.  The \emph{length}  of the  $\Omega$-row  $w=w_1\cdots w_m$ (where $w_i\in \Omega\cup X$) is defined to be $m$ and  is written $|w|$. The carrier of our $\Om$-algebra is $W(\Om;X)$, the set of $\Om$-rows.

We define the action of elements $\Omega$ on $W(\Om;X)$ by concatenation. First observe that if~$u$ and~$v$ are $\Om$-rows then the 
concatenation $uv$ of $u$ with $v$ is also an $\Om$-row, and this may be extended to the concatenation of 
arbitrarily many $\Om$-rows in the obvious way.  For $f\in \Omega(n)$ and $u_1,\ldots ,u_{n}\in W(\Omega;X)$, we define the 
the image of the $n$-tuple $(u_1,\ldots ,u_{n})\in W(\Om,X)^n$ under the operation $f$ 
to be the $\Om$-row $u_1\cdots u_{n}f$.  By abuse of notation we will refer to $W(\Om;X)$ as an $\Omega$-algebra. 

The alphabet $X\subset W(\Omega;X)$ and we call the subalgebra generated by $X$ the $\Omega$-\emph{word algebra} on $X$, 
denoted  $W_{\Omega}(X)$. Its elements are called $\Omega$-\emph{words} in the alphabet $X$. 
There is a clear distinction between $\Omega$-rows that are $\Omega$-words and those that are not. For example, if~$f$ is a binary operation then
\[x_1x_2x_3fx_4ff=(x_1,((x_2,x_3)f,x_4)f)f\]
is a $\Omega$-row which is also an $\Omega$-word, whereas $x_1ffx_2fx_3$ is an $\Omega$-row which is not an $\Omega$-word.

\begin{definition}[{\cite[Chapter 1]{cohnAlgebra3}}] \label{def:valency}
We define the \emph{valency} of an $\Omega$-row $w=w_1\cdots w_m$ ($w_i\in\Omega\cup X$) as $v(w)=\sum_{i=1}^{m}v(w_i)$ where
\[v(w_i)=\begin{cases} 1, & \text{if $w_i\in X$,}
\\
1- \operatorname{arity}(w_i), &\text{if $w_i\in \Omega$.}
\end{cases}\]

\end{definition}

\begin{proposition}[{\cite[Proposition 3.1, Chapter 1]{cohnAlgebra3}}]\label{propC3.1}
An $\Omega$-row $w=w_1\cdots w_m$ in $W(\Omega;X)$ is an $\Omega$-word if and only if every left-hand factor $u_i=w_1\cdots w_i$ of $w$ satisfies
\begin{equation*}
	\text{$v(u_i)>0$ for $i=1,\ldots ,m$ \qquad and \qquad $v(w)=1$.}
\end{equation*}
Moreover, each $\Om$-word can be obtained in precisely one way by applying a finite sequence of operations of $\Om$ to 
elements of $X$.
\end{proposition}

Let $\mathcal{A}$ be an $\Omega$-algebra. If in an element $w$ of $W_{\Omega}(X)$ we replace each element of $X$ by an element of $\mathcal{A}$ we obtain a unique element of $\mathcal{A}$. For $|w|=1$, this is clear, so assume $|w|>1$ and we will use induction on the length of $w$. We have $w=u_1\cdots u_{n}f$ ($f\in\Omega(n)$, $u_i\in W_{\Omega}(X)$), where, by Proposition \ref{propC3.1}, 
the $u_i$ are uniquely determined once $w$ is given. By induction each $u_i$ becomes a unique element $a_i\in \mathcal{A}$, when 
we replace the elements of $X$ by elements of $\mathcal{A}$. Hence $w$ becomes $a_1\cdots a_{n}f$; a uniquely determined 
element of $\mathcal{A}$. 

This establishes the next theorem.

\begin{theorem}[{\cite[Theorem 3.2, Chapter 1]{cohnAlgebra3}}]\label{theoC3.2}
Let $\mathcal{A}$ be an $\Omega$-algebra and let $X$ be a set. Then any injective mapping $\theta: X\to \mathcal{A}$ extends,
 in just one way,  to a homomorphism $\bar{\theta}:W_{\Omega}(X)\to \mathcal{A}$. 
That is,  $W_{\Omega}(X)$ is a free $\Omega$-algebra, freely generated by $X$. 
\end{theorem}

\begin{corollary}[{\cite[Corollary 3.3, Chapter 1]{cohnAlgebra3}}]\label{corC3.2}
Any $\Omega$-algebra $\mathcal{A}$ can be expressed as a homomorphic image of an $\Omega$-word algebra $W_{\Omega}(X)$ for a suitable set $X$. Here $X$ can be taken to be any set mapping  onto a generating set of $\mathcal{A}$.
\end{corollary}

By an \emph{identity} or \emph{law} over $\Omega$ in $X$ we mean a pair 
$(u,v)\in W_{\Omega}(X)\times W_{\Omega}(X)$ or an equation $u=v$ formed from such a pair. 
We say that the law $(u,v)$ \emph{holds} in the $\Omega$-algebra $\mathcal{A}$ 
or that $\mathcal{A}$ \emph{satisfies} the equation $u=v$ 
if every homomorphism $W_{\Omega}(X)\to \mathcal{A}$ maps $u$ and $v$ to the same element of $\mathcal{A}$. 
This correspondence between sets of laws and classes of algebras establishes a pair of maps, with the following definitions. 
\begin{itemize}
\item Given a set $\Sigma$ of laws over $\Om$ in $X$, form $\mathcal{V}_{\Omega}(\Sigma)$, the class of all $\Omega$-algebras satisfying all the laws in $\Sigma$. This class $\mathcal{V}_{\Omega}(\Sigma)$ is called the \emph{variety} generated by $\Sigma$. 
\item  Given a 
class $\cC$ of $\Om$-algebras we can form the set 
$\qf=\qf(\cC)$ of all laws over $\Om$ in $X$ which hold in all algebras of $\cC$.
\end{itemize}
Thus we have a pair of maps $\cV_{\Om}$ and $\qf$; relating 
each variety of $\Omega$-algebras to a relation $\mathfrak{q}$ on $W_{\Omega}(X)$ and vice-versa. We shall see below that
$\qf(\cC)$ is a congruence, but first we make a further definition.

A subalgebra of an $\Omega$-algebra $\mathcal{A}$ is called \emph{fully invariant} if it is mapped into itself by all endomorphisms of $\mathcal{A}$. A congruence $\Gamma$ on $\mathcal{A}$ is said to be \emph{fully invariant}  if $(u,v)\in \G$ implies 
$(u\theta,v\theta)\in \G$, for all endomorphisms $\theta$ of $\cA$. The \emph{fully invariant congruence generated by $\Gamma$} is 
\[\cI=\bigcap\{V \mid \text{$V$ is a fully invariant congruence and $\Gamma\subseteq V$}\}.\]
It follows that $\cI$ is the smallest invariant congruence on $\mathcal{A}$ generated by $\Gamma$.

We claim that if $\cC$ is a class of $\Om$-algebras then $\qf(\cC)$ is a fully invariant congruence on $W_\Om(X)$. To see that $\qf(\cC)$ is 
a congruence, note that in in every class $\mathcal{C}$ of $\Omega$-algebras we have the following:
 $u=u$ for all $u\in W_{\Omega}(X)$; if $u=v$ holds then so does $v=u$; and if $u=v$ and  $v=w$ then also $u=w$. Further, 
if $u_i=v_i$ for $i=1,\ldots ,n$ are laws holding in $\mathcal{A}$ and if $\omega\in\Omega(n)$, then $u_1\cdots u_n\omega=v_1\cdots v_n\omega$ holds 
in $\mathcal{A}$. Hence $\qf(\cC)$ is indeed a congruence. 

To see that $\qf(\cC)$ is a fully invariant congruence, let $(u,v)\in \qf(\cC)$ and let $\theta$ be any endomorphism of  $W_{\Omega}(X)$.  If $\mathcal{A}\in \mathcal{C}$ and $\alpha: W_{\Omega}(X) \to \mathcal{A}$ is any homomorphism, then so is $\theta\alpha$, hence $u\theta\alpha=v\theta\alpha$. Thus the law $u\theta=v\theta$ holds in $\mathcal{A}$, so $(u\theta,v\theta)\in \qf(\cC)$ and thus $\qf(\cC)$ is a fully invariant congruence. 
Cohn shows in addition that the map $\cV_{\Om}$ is a bijection with inverse $\qf$, and deduces the following theorem.

Given sets $S$ and $T$ and a relation $\G$ from $S$ to $T$, we may use  $\Gamma$ to define a system of subsets of $S$, $T$, 
as follows. 
 For any subset $X$ of $S$ we define a subset $X^*$ of $T$ by 
 \[X^*=\{y\in T \mid (x,y)\in \Gamma \ \text{for all $x\in X$}\}=\cap_{x\in X}\{x\}\Gamma,\]
 and similarly, for any subset $Y$ of $T$ we define a subset $Y^*$ of $S$ by
 \[Y^*=\{x\in S \mid (x,y)\in \Gamma \ \text{for all $y\in Y$}\}=\cap_{y\in Y}\{y\}\Gamma^{-1}.\]
 We thus have mappings $X\mapsto X^*$ and $Y\mapsto Y^*$ of the power sets of $S$ and $T$ with the following properties:
 \begin{align}
 \label{power1}
 X_1\subseteq X_2 \Rightarrow X_1^*\supseteq X^*_2 \qquad & \qquad  Y_1\subseteq Y_2 \Rightarrow Y_1^*\supseteq Y^*_2
 \\
 \label{power2}
  X\subseteq X^{**} \qquad &\qquad Y\subseteq Y^{**},
 \\
 \label{power3}
  X^{***}=X^* \qquad &\qquad  Y^{***}=Y^*.
 \end{align}
 A pair of maps $X\mapsto X^*$, from the power set $2^S$ of $S$  to the power set $2^T$ of $T$, and $Y\mapsto Y^*$,
from $2^T$ to $2^S$, satisfying (\ref{power1}--\ref{power3}) is called a \emph{Galois connection}.

\begin{theorem}[{\cite[Theorem 3.5, Chapter 1]{cohnAlgebra3}}]\label{theorem3.5}
Let $W=W_{\Omega}(X)$ be the $\Omega$-word algebra on the alphabet $X$. 
The pair of maps $\Si\mapsto \cV_{\Om}(\Si)$ and $\cC\mapsto \qf(\cC)$ 
forms a Galois connection giving a bijection between 
 varieties of $\Omega$-algebras and  fully invariant congruences $\mathfrak{q}$ on $W_{\Omega}(X)$. 
\end{theorem}

\begin{proposition}[{\cite[Proposition 3.6, Chapter 1]{cohnAlgebra3}}]\label{cohnPropFA}
Let $\mathcal{V}$ be a variety of $\Omega$-algebras and $\mathfrak{q}$ the congruence on $W_{\Omega}(X)$ (the $\Omega$-word algebra generated by $X$) consisting of all the laws on $\mathcal{V}$ \emph{i.e.} the fully invariant congruence $\mathfrak{q}(\mathcal{V})$. Then $W_{\Omega}(X)/\mathfrak{q}$ is the free $\mathcal{V}$-algebra on $X$.
\end{proposition}

Suppose $\Sigma$ is a set of laws over $\Omega$ in $X$ and let ${\cal{V}}={\cal{V}}_{\Omega}(\Sigma)$ and $\mathfrak{q}=\mathfrak{q}(\cal{V})$. Then $\Sigma \subseteq \mathfrak{q}$ and, from Proposition \ref{cohnPropFA}, $\mathfrak{q}$ is a fully invariant 
congruence and $W_\Omega(X)/\mathfrak{q}$ is the free $\cal{V}$-algebra. 

Now let $\mathfrak{p}$ be the fully invariant congruence  generated by 
$\Sigma$. Then, as $\Sigma\subseteq \mathfrak{q}$ and $\mathfrak{q}$ is a 
fully invariant congruence, we have $\mathfrak{p}\subseteq \mathfrak{q}$.
 Let $\cA=W_\Omega(X)/\mathfrak{p}$. Then $\cA$ is an
$\Omega$-algebra, in which every law of $\Sigma$ holds (as $\Sigma \subseteq 
\mathfrak{p}$). Thus $\cA$ is a $\cal{V}$-algebra. 
Then, from Proposition \ref{cohnPropFA}, the natural map $X\to \cA$ extends to
 a homomorphism $W_\Omega(X)/\mathfrak{q}\to \cA$. It follows that 
$\mathfrak{q}\subseteq \mathfrak{p}$. Therefore $\mathfrak{p}=
\mathfrak{q}=\mathfrak{q}(\cal{V})$. 
 We record this as a corollary which we shall  use in  Section \ref{v21freealgebra} to 
construct Higman's algebras $V_{n,r}$. 
\begin{corollary}\label{cor:fic}
Let $\Sigma$ be a set of laws over $\Omega$ in $X$,  
let ${\cal{V}}={\cal{V}}_{\Omega}(\Sigma)$ and $\mathfrak{q}=\mathfrak{q}(\cal{V})$.
Then $\qf$ is the fully invariant congruence generated by $\Si$. 
\end{corollary}

\section{The Higman Algebras \VnrHeading}\label{v21freealgebra}

In this section we define the algebras which Higman called $V_{n,r}$. 
Let $n \geq 2$ be an integer and  let $\cA$ be
 an $\Omega$-algebra, with carrier $S$ and operator 
domain $\Omega=\{\lambda,\alpha_1,\ldots, \alpha_n\}$, such that  $a(\alpha_i)=1$, for $i=1,\ldots, n$ and $a(\lambda)=n$. 
We call the  $n$-ary operation $\lambda : S^n\to S$ 
a \emph{contraction} and the unary operations 
$\alpha_i: S\to S$ 
\emph{descending operations}. 
We define  a map $\a:S\maps S^n$, which we shall call an \emph{expansion}, by 
\[v\alpha=(v\alpha_1,\ldots, v\alpha_n),\]
for all $v\in S$.
 For any subset $Y$ of $S$, a \emph{simple expansion} of $Y$ consists of substituting some element $y$ of $Y$ by 
the $n$ elements of the tuple $y\alpha$. A sequence of $d$ simple expansions of $Y$ is called a 
\emph{$d$-fold expansion of $Y$}.  A set obtained from $Y$ by a $d$-fold expansion, $d\ge 0$,  
is called an \emph{expansion} of $Y$.  For example,   if $x\in S$ then 
$\{x\alpha_1, \ldots, x\alpha_n\}$ is the unique simple expansion of $\{x\}$ and   
the  $2$-fold expansions of $\{x\}$ are  
the sets $\{x\alpha_1, \ldots, x\a_{i-1},x\a_i\a_1,\ldots ,x\a_i\a_n,x\a_{i+1},\ldots , x\alpha_n\}$,
for $1\le i\le n$. Every  $d$-fold expansion of $Y$ has 
$|Y|+(n-1)d$ elements. 
 Similarly, 
a \emph{simple contraction} of $Y$ consists of substituting $n$ distinct elements $\{y_1,\ldots, y_n\}\in Y$ 
by the single element $(y_1,\ldots, y_n)\lambda$. A set obtained from $Y$ by applying a finite number of simple contractions is called a
 \emph{contraction} of $Y$.  

From now on in this paper, $\Om$ is fixed as above. 
 Let $\cxx$ be a non-empty set and 
 recall that  the $\Om$-word algebra $W_{\Omega}(\cxx)$  is the free $\Omega$-algebra on $\cxx$. 
 

\begin{definition}\label{SigmaLaws}
Let $\Sigma_n$ be the set of laws over $\Om$ in $\cxx$:
\begin{enumerate}
\item for all $w\in W_{\Omega}(\cxx)$,
\[w\alpha\lambda=w,\]
(or explicitly $w\alpha_1\cdots w\alpha_n\lambda=w$). 
\item for all  $(w_1,\ldots ,w_n)\in W_{\Omega}(\cxx)^n$ and  $i\in \{1,\ldots, n\}$, 
\[w_1\cdots w_n\lambda\alpha_i=w_i.\]
\end{enumerate}
That is, 
\begin{align*}
\Sigma_n &=
\{(w\alpha_1\cdots w\alpha_n\lambda,w) \mid w\in W_{\Omega}(\cxx)\}\\ &\cup\bigcup_{i=1}^n
\{(w_1\cdots w_n\lambda\alpha_i,w_i) \mid w_i\in W_{\Omega}(\cxx)\}.
\end{align*}
 Let $\mathcal{V}_n=\cV_\Om(\Si_n)$ the variety of $\Omega$-algebras which satisfy $\Si_n$ and let 
$\mathfrak{q}=\qf(\cV_n)$.
\end{definition}

From Proposition \ref{cohnPropFA} and Corollary \ref{cor:fic}, it follows that  $\mathfrak{q}$ is the fully invariant 
congruence on $W_\Om(\cxx)$ generated by $\Si_n$ and $W_\Om(\cxx)/\mathfrak{q}$ is the 
free $\cV_n$-algebra on $\cxx$. 
\begin{definition}\label{def:vnr}
Let $\cxx$ be a non-empty, finite or countably enumerable set of cardinality $r$ and $n\ge 2$ an integer. 
Then $V_{n,r}(\cxx)$ is the free $\cV_n$-algebra $W_\Om(\cxx)/\mathfrak{q}$, where $\qf=\qf(\cV_n)$ and
$\cV_n=\cV_\Om(\Si_n)$. 
\end{definition}
When no ambiguity arises we refer to $V_{n,r}(\cxx)$ as $V_{n,r}$. 

\begin{remark} \label{rem:Higman_std_form}
In \cite[Section 2]{Higg} Higman defines a \emph{standard form over} $\cxx$ to be one of the 
finite sequences of elements of $\cxx\cup \{\alpha_1,\ldots, \alpha_n, \l\}$ specified by the following rules.
\be[(i)]
\item\label{it:sf1} $x\alpha_{i_1}\cdots \alpha_{i_k}$ is a standard form whenever $k\ge 0$, $x \in \cxx$ and $1\leq i_j \leq n$ for $j=1,\dots,k$.
\item\label{it:sf2} If $w_1,\ldots, w_n$ are standard forms then so is $w_1\cdots w_n\lambda$, unless there is a standard 
form $u$ such that $w_i=u\alpha_i$ for $i=1,\ldots, n$.
\item\label{it:sf3} No sequence is a standard form unless this follows from \eqref{it:sf1} and \eqref{it:sf2}.
\ee
We define the descending operations $\alpha_1,\ldots, \alpha_n$ by the rules
\[(x\alpha_{i_1}\cdots \alpha_{i_k})\alpha_i=x\alpha_{i_1}\cdots \alpha_{i_k}\alpha_i,\]
\[(w_1\cdots w_n\lambda)\alpha_i=w_i\]
for $i\in \{1,\ldots, n\}$.
The contraction operation~$\lambda$ is defined by
\[(w_1,\ldots, w_n)\lambda=w_1\cdots w_n\lambda,\]
unless there is a standard form $u$ such that $w_i=u\alpha_i$ for $i=1,\ldots ,n$ in which case 
\[(w_1,\ldots, w_n)\lambda=(u\alpha_1,\ldots, u\alpha_n)\lambda=u.\]
These operations turn the set of standard forms into an $\Omega$-algebra. 
Higman then goes on to prove that this is a free  $\mathcal{V}_n$-algebra, freely generated by $\cxx$ 
(\cite[Lemma 2.1]{Higg}). This follows in our case  from the definition above, and the remarks following it, together with  Lemma
\ref{lem:stdform} below.
\end{remark}

\begin{lemma}\label{lem:stdform}
Let $U$ be an equivalence class of the congruence $\mathfrak{q}$ on $W_{\Omega}(\cxx)$.
 Then there exists a unique minimal length element $u$ in $U$. 
The unique minimal length elements of equivalence classes are precisely the 
standard forms of Higman.
\end{lemma}

\begin{longversion}
\begin{proof}
The proof depends on an explicit description of the congruence $\qf$. First, 
recall from 
 the definition above 
Proposition \ref{equivTOcong} that  the elementary translations derived from $\Om$ by specialisation in $W_\Om(\cxx)$ 
are the maps $\a_i$ and  the maps taking $u\in W_\Om(\cxx)$ to $w_1\cdots w_{i-1}uw_{i}\cdots w_{n-1}\l$, 
for some $i\in \{1,\ldots ,n\}$ and some fixed $(n-1)$-tuple $(w_1,\ldots ,w_{n-1})$ of elements
of $W_\Om(\cxx)$. 
A translation of $W_\Om(\cxx)$ is a composition of elementary translations (or the identity). 
Now, for $(u,v)\in W_\Om(\cxx)^2$, define the \emph{translation closure} of $(u,v)$ to be the subset 
\[\overline{(u,v)}=\{(u',v')\in W_\Om(\cxx)^2|(u',v')=(u\t,v\t),\textrm{ for some translation } \t 
\textrm{ of }W_\Om(\cxx)\}.\] 
 Next define the \emph{translation closure} $\bar \Si_n$ of $\Si_n$ to be 
\[\bar \Si_n=\bigcup_{(u,v)\in \Si_n} \overline{(u,v)}.\]

We claim that $\qf$ is the equivalence generated by $\bar \Si_n$. Temporarily denote this equivalence by $\pf$. 
As $\bar \Si_n$ is closed under translation so is $\pf$, and it follows (see Proposition \ref{equivTOcong}) 
that $\pf$ is a congruence; so $\pf\supseteq \qf$, as $\pf\supseteq \Si_n$. Furthermore, if $\pf'$ is a congruence containing
$\Si_n$ then $\pf'$  is closed under translation, so contains $\bar \Si_n$. Thus $\pf'\supseteq \pf$. In particular $\qf \supseteq \pf$,
as required. Therefore $\qf$ is the equivalence generated by $\bar \Si_n$, as claimed. 

Now we shall show that $\bar \Si_n$ has the following 2 properties.  If $(a,b)\in \bar \Si_n$ then 
\be[(I)]
\item \label{it:rew1} $|a|>|b|$ and 
\item \label{it:rew2} there exist $\Om$-rows $w_0$ and $w_1$ in $W(\Om,\cxx)$ 
and $(u,v)\in \Si_n$ such that 
$a=w_0uw_1$ and $b=w_0vw_1$. 
\ee
If $(a,b)\in \bar \Si_n$ then $(a,b)$ is obtained by
applying a sequence of $t$ elementary translations to some $(u,v)\in \Si_n$. 
(\ref{it:rew1}) and (\ref{it:rew2}) are proved by induction on the number $t$ of translations required.
If $t=0$ then $(a,b)\in \Si_n$ and \ref{it:rew1} holds by definition of $\Si_n$ and \ref{it:rew2} holds with
$w_0$ and $w_1$ trivial. Assume both these results hold for elements obtained by application 
of at most $t-1$ elementary translations 
to an element of $\Si_n$. We have $(a,b)=(c,d)\t$, for some elementary translation $\t$ and some 
$(c,d)\in \bar \Si_n$, which is obtained from $(u,v)$ by application of $t-1$  
elementary translations. From the inductive
hypothesis $|c|>|d|$ and, since every translation changes the length of left and right hand sides of a pair by the
same amount, it follows that $|a|>|b|$. Also, from the inductive hypothesis $c=w_0uw_1$ and $d=w_0vw_1$, for 
some $\Om$-rows $w_0$ and $w_1$. Depending on the type of $\t$ we have 
$(a,b)=(c,d)$, or 
$(a,b)=(c\a_i,d\a_i)= (w_0uw_1\a_i,w_0vw_1\a_i)$ or 
\begin{align*}
(a,b) & =(s_1 \cdots s_{i-1} c s_i \cdots s_{n-1}\l,s_1\cdots s_{i-1} d s_i \cdots s_{n-1}\l)\\
& =(s_1\cdots s_{i-1}  w_0uw_1 s_i\cdots s_{n-1}\l,s_1\cdots s_{i-1} w_0vw_1 s_i\cdots s_{n-1}\l),
\end{align*} 
where $s_1,\ldots, s_{n-1}\in W_\Om(\cxx)$ and $1\le i\le n$. 
In all cases $(a,b)$ can be seen to have
the form required by (\ref{it:rew2}). By induction (\ref{it:rew1}) and  (\ref{it:rew2}) hold for all $(a,b)$. 

We regard $\bar \Si_n$ as a reduction system on $W_\Om(\cxx)$ (see for example \cite{BookandOtto}) and
write $a \RAS{}{}b$ if $(a,b)\in \bar \Si_n$  and $a \RAS*{}b$ if $(a,b)$ is in the reflexive, transitive closure of $\bar \Si_n$.
(Thus $a \RAS*{}b$ if and only if $a=b$ or there is a sequence $a=a_0,\ldots, a_n=b$, such that $a_i\RAS{}{} a_{i+1}$.)
As $\qf$ is the reflexive, symmetric, transitive closure of $\bar \Si_n$, the first statement of the lemma will follow if we 
show that $\bar \Si_n$ is 
\be[(a)]
\item\label{it:term} terminating (every sequence $a_0\RAS*{} \cdots \RAS*{} a_n \RAS*{}\cdots$ is eventually stationary) and 
\item\label{it:conf} locally confluent: whenever $b\LAS{}{} a \RAS{}{} c$ there exists $d$ such that $b\RAS{*}{} d \LAS{*}{} c$.
\ee 
As $\bar \Si_n$ is length reducing it is certainly terminating, so we must show it is locally confluent.
Before embarking on the proof of this fact, we consider the ways in which it is possible for words of $W_\Om(\cxx)$ to
overlap. To this end suppose that  $p=ab$, and $q=bc$ are elements of  $W_\Om(\cxx)$, with $b$ non-trivial. 
If $a$ is non-trivial then
we have $1=v(p)=v(a)+v(b)$ and $v(a)\ge 1$, so $v(b)\le 0$. However (see Proposition \ref{propC3.1}) 
as $b$ is an initial segment of $\qf$, this means that $b$ must
be trivial, a contradiction. Hence if $p$ and $q$ overlap then $p$ is a subword of $q$ or vice-versa. Assuming that 
$q$ is a proper subword of $p$ we then have $p=p_0qp_1$, where one of $p_0$, $p_1$ is non-trivial. If $p_1$ is trivial then
$1=v(p)=v(p_0q)=v(p_0)+v(q)=v(p_0)+1$, so $v(p_0)=0$, contradicting Proposition \ref{propC3.1} again. We conclude that if 
$p$ and $q$ overlap and are not equal then one is a subword of the other and, assuming $q$ is a subword of $p$, 
\begin{equation}\label{eq:olap}
p=p_0qp_1,\textrm{ with } p_1\textrm{ non-trivial}. 
\end{equation}  

Now suppose that  $b\LAS{}{} a \RAS{}{} c$. From property (\ref{it:rew2}) above, we have $\Om$-rows $w_0,w_1$, $w'_0$ and $w'_1$, and 
elements $(u,v), (r,s)\in \Si_n$, such that $(a,b)=(w_0uw_1,w_0vw_1)$ and $(a,c)=(w'_0rw'_1,w'_0sw'_1)$. If $a$ factors as 
$w_0uwrw'_1$ then setting $d=w_0vwsw'_1$ we have $b=w_0vwrw'_1\RAS{}{} d\LAS{}{}  w_0uwsw'_1=c$, so local confluence holds in this case.
Similarly, if $a=w'_0rwuw_1$ then we have local confluence. Assume then that $a$ has no such factorisation. This means that 
$u$ and $r$ are overlapping subwords of $a$. Therefore, interchanging $b$ and $c$ if necessary, we may assume that 
$r$ is subword of $u$. Thus, there exist $\Omega$-rows $p$ and $q$ such that $u=prq$.  
We may therefore restrict to the case where $a=u$, $b=v$ and $c=psq$. 
We consider in turn the 
various forms that $u$ may take.

First consider the case $u=w\a_1\cdots w\a_n\l$, for some $w\in W_\Om(\cxx)$. Suppose the first letter
of $r$ occurs in $w\a_j$ and the last in $w\a_{j+k}$, for some $j,k$ with $1\le j$, $k>0$.
Then $w\a_{j+k}$ and $r$ overlap, and, applying \eqref{eq:olap}, $w\a_{j+k}$ must  be a subword of $r$ and, 
given the position of its last letter, 
$r$ must therefore have terminal segment  $w\a_{j+k}$. However, this contradicts \eqref{eq:olap}.    
Applying \eqref{eq:olap} again, $r$ may be a subword of $w$ in  $w\a_i$ or may equal $w\a_i$, for 
some $i\in \{1,\ldots, n\}$. 
 If $r$ is a subword of $w$ in $w\a_i$ then $u=w\a_1 \cdots w\a_{i-1}w'rw''\a_i w\a_{i+1}\cdots w\a_n\l$, 
where $w=w'rw''$ and we have 
\begin{align*}
v=w=w'rw''&\LAS{}{}u=w\a_1 \cdots w\a_{i-1}w'rw''\a_i w\a_{i+1}\cdots w\a_n\l\\
&\RAS{}{}w\a_1 \cdots w\a_{i-1}w'sw''\a_i w\a_{i+1}\cdots w\a_n\l=c
\end{align*}
and 
$c=w'rw''\a_1 \cdots w'rw''\a_{i-1}w'sw''\a_i w'rw''\a_{i+1}\cdots w'rw''\a_n\l$ 
so 
\[v=w'rw''\RAS{}{}w'sw''\LAS{}{}w'sw''\a_1 \cdots w'sw''\a_{i-1}w'sw''\a_i w'sw''\a_{i+1}\cdots w'rw''\a_n\l\LAS{*}{}c.\]
On the other hand if 
$u=w\a_1 \cdots w\a_{i-1}rw\a_{i+1}\cdots w\a_n\l$ with $r=w\a_i$ then $(r,s)\in \Si_n$ 
implies $w=v_1\cdots v_n\l$, for some $v_j\in W_\Om(\cxx)$,   
and $s=v_i$, so $c=w\a_1 \cdots w\a_{i-1} v_i w\a_{i+1}\cdots w\a_n\l$. 
In this case 
$w\a_j\RAS{}{}v_j$, for all $j$, so 
\[c\RAS{}{}v_1 \cdots w\a_{i-1} v_i w\a_{i+1}\cdots w\a_n\l \RAS{}{}\cdots\RAS{}{} v_1 \cdots v_{i-1} v_i w\a_{i+1}\cdots w\a_n\l
\RAS{}{}\cdots\RAS{}{}v_1\cdots v_n\l\]
%
and 
\[v=v_1\cdots v_n\l\LAS{*}{} c.\]
Thus, in the case where $u=w\a_1\cdots w_n\l$, we have local confluence. 

Now consider the case where $u=w_1\cdots w_n\l\a_i$, for $i\in \{1,\ldots, n\}$, so $v=w_i$. 
As in the previous part of the proof, $r$ must be a subword of $w_j$, for some $j\in\{1,\ldots, n\}$. 
Suppose then that 
$w_j=w'rw''$. If $i=j$ we have 
\[b=w'rw''\LAS{}{} u=w_1\cdots w_{i-1} w'rw''w_{i+1} \cdots w_n\l\a_i\RAS{}{}w_1\cdots w_{i-1} w'sw''w_{i+1} \cdots w_n\l\a_i=c\]
and 
\[b=w'rw''\RAS{}{}w'sw'' \LAS{}{} w_1\cdots w_{i-1} w'sw'' w_{i+1}\cdots w_n\l\a_i=c.\]
If $i\neq j$ then 
\[b=w_i\LAS{}{} u=w_1\cdots w_{j-1}w'rw''w_{j+1}\cdots w_n\l\a_2\RAS{}{}w_1\cdots  w_{j-1}w'sw''w_{j+1}\cdots  w_n\l\a_i=c\]
and 
\[b=w_i\LAS{}{} w_1\cdots w_{j-1}w'sw''w_{j+1}\cdots w_n\l\a_i=c.\]

In all cases we have local confluence, so we conclude that (\ref{it:conf}) holds for $\bar \Si_n$. Therefore,
(from \cite[Section 1.1]{BookandOtto}, for example) every equivalence class of $\qf$ contains a unique element 
which is not the left hand side of any element of $\bar \Si_n$: such elements of $W_\Om(\cxx)$ are 
called \emph{irreducible} elements. As  $\bar \Si_n$ is length reducing it follows that the unique irreducible 
element of an equivalence class is an element of minimal length in its equivalence class. 

To prove the second statement of the Lemma note that every standard form is irreducible, so is of minimal
length in its equivalence class. Conversely, given an irreducible element a straightforward induction on
its length shows that it is a standard form.  
\end{proof}
\end{longversion}
\begin{shortversion}
To prove Lemma \ref{lem:stdform}, one can use a standard argument which proves a statement of this form in an algebra of an 
appropriate type. Details may be found in \cite[Lemma 2.4.5]{Bar14}.
\end{shortversion}

Let $y$ be the minimal length representative of its equivalence class in $V_{n,r}$ \emph{i.e.} let $y$ be a standard form. 
Then the \emph{length} of  the equivalence class of $y$ is the length of $y$, denoted $|y|$, and  
 the \emph{$\lambda$-length} of the equivalence class of $y$ is the number of times the symbol $\lambda$ occurs in $y$. 

Now that we have a concrete description of the free algebra $V_{n,r}$ in the variety $\mathcal{V}_n$, we recall  
those results of \cite[Section 2]{Higg},  required in the sequel.

\begin{lemma}[cf. {\cite[Lemma 2.3]{Higg}}]\label{anyexpansion}
Let $B$ be a basis of  $V_{n,r}(\cxx)$.
\be
\item\label{it:bases1}
Every expansion of $B$ is a basis of $V_{n,r}(\cxx)$.
\item\label{it:bases2}
Every contraction of $B$ is a basis of $V_{n,r}(\cxx)$.
\ee
\end{lemma}

\begin{proof}
\be
\item
Let $Y$ be a $d$-fold expansion of $B$, where $d\ge 0$. 
Arguing by induction, we assume that every $d$-fold expansion of $B$ is a basis of  
$V_{n,r}$ and show that any simple expansion of $Y$ is also a basis.
Let $y\in Y$ and let $Y'$ be the simple expansion 
\[Y'=(Y\setminus \{y\})\cup\{y\alpha_1,\ldots ,y\alpha_n\}.\]
Since $y=y\alpha_1\cdots y\alpha_n\lambda$, the set $Y'$ generates $V_{n,r}$.
It remains to show that $Y'$ is a basis for $V_{n,r}$. 

Given $\mathcal{A}\in \cV_n$ and a map $\theta:Y'\to \mathcal{A}$, we shall show that 
there is a unique homomorphism $\bar\theta: V_{n,r}\to \mathcal{A}$ extending $\theta$. 
Firstly, define $\theta^\ast$ from $Y$ to $\mathcal{A}$ by 
$y'\theta^*=y'\theta$, for $y'\in Y\setminus \{y\}$, and 
$y\theta^*=y\alpha_1\theta \cdots y\alpha_n\theta\lambda$. 
As $Y$ is a basis, there is a unique homomorphism $\bar\theta^*$ from $V_{n,r}$ 
to $\mathcal{A}$ extending $\theta^*$. 
Now
\[(y\alpha_i)\bar\theta^*=(y\bar\theta^*)\alpha_i=(y\theta^*)\alpha_i=
(y\alpha_1\theta \cdots y\alpha_n\theta\lambda)\alpha_i=y\alpha_i\theta.\]
Hence $\bar\theta^*$ also extends $\theta$. 
Furthermore, any other homomorphism which extends $\theta$ must equal $\bar\theta^*$, since any such map must be defined on $Y$ in the same way as $\theta^\ast$. 
\item This is proved in the same way as \ref{it:bases1}. \hfill \qedhere
\ee 
\end{proof}

The final statement of Corollary \ref{cor:commexp} forms a partial  converse to this lemma, for finite bases. 
Mostly we work with
 bases for $V_{n,r}(\cxx)$ which are expansions of $\cxx$, so we make the following definition. 
\begin{definition}\label{def:Abasis}
Let $A=\{\alpha_1,\ldots ,\alpha_n\}\subset \Omega$. 
An $A$\emph{-basis} of $V_{n,r}(\cxx)$ is an expansion of $\cxx$.
\end{definition}

If $\cA=(S,\Om)$ is an $\Om$-algebra with carrier $S$ then we may form the $A$-algebra $(S,A)$ and the $\{\l\}$-algebra $(S,\{\l\})$,
where the elements of $A$ and $\{\l\}$ have actions inherited from $\cA$. We call these, respectively, the $A$\emph{-algebra} and 
$\{\l\}$\emph{-algebra of} $\cA$. A subset $U$ of $V_{n,r}$ is said to be \emph{$A$-closed} if $u\a_i \in U$, for all $\a_i\in A$, and 
an $A$-closed subset is called an $A$\emph{-subalgebra} of (the $A$-algebra of) $V_{n,r}$.  
Similarly $W\subseteq V_{n,r}$ is called a $\{\l\}$\emph{-subalgebra} (of the $\{\l\}$-algebra of $V_{n,r}$) if it is  $\{\l\}$\emph{-closed}: 
that is  if $w\l\in W$, for all $w\in W$.  

\begin{definition}\label{def:Asubalg}
Let $Y$ be a subset of $V_{n,r}$. The $A$-subalgebra generated by $Y$ is denoted $Y\mA$. 
The $\{\l\}$-subalgebra generated by $Y$ is denoted $Y\ml$. 
\end{definition}

The free monoid on a set $L$ is denoted $L^*$. If $Y$ is a subset of $V_{n,r}(\cxx)$ then $YA^*=\{y\G \mid y\in Y, \G\in A^*\}$ is
 $A$-closed, and it follows that $Y\mA=YA^*$. If in addition $Y\subseteq \cxx\mA$, then $y\G$ is a standard form for
all $y\in Y$ and $\G\in A^*$. In the sequel we write $Y\mA\ml$ for $(Y\mA)\ml$.

\begin{lemma}\label{lem:expbasis}
Let $B$ be an $A$-basis and $Y$ a finite basis for $V_{n,r}(X)$. If $B\subseteq Y\mA$ then $B$ is an expansion of $Y$. 
\end{lemma}
\begin{proof}
Since $Y$ is finite, there exists an expansion of $Y$ contained in $B\mA$. Let $d$ be minimal such that there is a $d$-fold expansion 
of  $Y$ contained in $B\mA$,
and let $W$ be such a $d$-fold expansion.
Each $w\in W$ is of the form $w=b\G$, for some $b\in B$ and $\G\in A^*$. As $B \subseteq Y\mA$ we
have $b=y\D$, for some $y\in Y$ and $\D\in A^*$; so $w=y\D\G$. Also, as $w\in W$, there exists $y'\in Y$ such that $w=y'\G'$,
as part of an expansion of $Y$. As $Y$ is a basis it follows that $y=y'$ and $\D\G=\G'$.

Suppose that $\G\neq 1$, so that 
$\G=\G_0\a_j$, for some $\a_j\in A$ and $\G_0\in A^*$.
As $W$ is an expansion of $Y$ it follows that $y\D\G_0\a_i\in W$, for
all $i\in \{1,\ldots, n\}$. Furthermore $y\D\G_0\in B\mA$, so the union
\begin{equation*}
W' = (W\setminus \{y\D\G_0\a_i \mid 1\le i\le n\}) \cup \{y\D\G_0\}
\end{equation*}
is contained in~$B\mA$.
Now~$W'$ is a simple contraction of~$W$, so $W'$ is a basis by Lemma~\ref{anyexpansion}.
But~$W'$ is a $(d-1)$-fold expansion of~$Y$, which contradicts the minimality of $d$.
So $\G=1$ and $w\in B$, and hence $W\subseteq B$. 

Conversely, if $b\in B$ then $b=y\G$, for some $y\in Y$ and $\G\in A^*$.
So either $b\D=y\G\D \in W$ for some $\D\in A^*$, or 
$y\G_0=w\in W$, where $\G=\G_0\G_1$. In the first case, $b\D=w\in B$ implies $w=b$ and $\D=1$. In the second case, 
$b=y\G=y\G_1\G_0=w\G_0$, with $w\in B$, so again $w=b$ and $\G_0=1$. Thus $B\subseteq W$. 
\end{proof}

A word $\Gamma\in A^*$ is called {\em primitive} if it is not a proper 
power of another word.
Explicitly, this means that if $\Gamma$ is non-trivial and $\Gamma \in \{\Delta\}^*$,  for some 
$\Delta \in  A^*$, then  $\Gamma=\Delta$.
\begin{proposition}[\cite{Lothaire}, Proposition 1.3.1, Chapter 1]\label{Lothaire1}
If $\Gamma^n=\Delta^m$ with $\Gamma, \Delta \in  A^*$ and $n,m\ge 0$, 
there exists a word $\Lambda$ such that $\Gamma, \Delta \in \{\Lambda\}^*$. 
In particular, for each word $\Gamma\in  A^*$, there exists a unique primitive word $\Lambda$ such that  $\Gamma\in \{\Lambda\}^*$.
\end{proposition}

\begin{proposition}[\cite{Lothaire}, Proposition 1.3.2, Chapter 1]\label{wordsinA}
Two words $\Gamma, \Delta \in  A^*$ commute if and only if they are powers of the same word. 
More precisely, the set of words commuting with a word $\Gamma\in  A^*$ is a monoid generated by a single primitive word.
\end{proposition}

\begin{lemma}[{\cite[Section 2, Lemma 2.2]{Higg}}]\label{2.2H}
Let $Y$ be a subset of $V_{n,r}$ and let $W$ be the $\Om$-subalgebra of $V_{n,r}$ generated by $Y$. 
Then 
\be
\item $W=Y\mA\lp \l\rp$ and 
\item for all $w\in W$, the set $w\mA\setminus Y\mA$ is finite.
\ee 
\end{lemma}

\begin{proof}
\be
\item
Let $w\in W$. Then there exists a finite subset $Y_0$ of $Y$ such that $w$ belongs to
the $\Omega$-subalgebra $W_0$ of $V_{n,r}$ generated by $Y_0$. Let $Z$ be an expansion of $\cxx$ such that 
$|Z|\ge |Y_0|$. 
Choose a surjection $\b$ of $Z$ onto $Y_0$. As $V_{n,r}$
is freely generated by $Z$ we may extend $\b$ to a homomorphism from $V_{n,r}$ to $W_0$. Let 
$w_0$ be the preimage of $w$ under this homomorphism and let $l$ be the $\l$-length of the
standard form of  $w_0$ over $Z$. 
By a straightforward induction on $l$ it is apparent that
$w_0\in Z\mA\lp \l\rp$. Hence the image $w$ of $w_0$ in $W_0$ belongs to $Y_0\mA\lp \l\rp
\subseteq Y\mA\lp \l\rp$, as required. 
\item
As in the previous part of the proof, we may assume that $W$ is freely generated by $Y$.  
Let $w\in W$ and let $l$ be the $\l$-length of the
standard form of  $w$ over $Y$. 
Then  $w\alpha_{i_{1}}\cdots \alpha_{i_{r}} \in Y\langle A \rangle$, whenever $r\geq l$. 
Hence, the only elements of the set difference $w\langle A \rangle \setminus Y\langle A \rangle$ are those of the form $w\alpha_{i_{1}}\cdots \alpha_{i_{r}}$ with $r<l$, and there are only finitely many of these since 
we only have $n$ choices for each $\alpha_{i_{j}}$.
\hfill \qedhere
\ee
\end{proof}


\begin{lemma}[{\cite[Section 2, Lemma 2.4]{Higg}}]\label{HigLemma2.4} Let  $\cxx$ be a set of size $r\ge 1$ and let 
$X\subseteq V_{n,r}(\cxx)$ be an expansion of $\cxx$.
If $U$ is a subset of $V_{n,r}(\cxx)$ contained in $X\mA$, then the following are equivalent:
\begin{enumerate}
\item\label{it:241} $U=X\mA\cap Y\mA$, for some  generating set $Y$ of $V_{n,r}$,
\item\label{it:242} $U$ is $A$-closed and $X\mA\setminus U$ is finite, 
\item\label{it:243} $U=Z\mA$ for some expansion $Z$ of $X$.
\end{enumerate}
Moreover, if $Y$ in statement~\ref{it:241} is a finite basis for $V_{n,r}(\cxx)$ then $Z$ in statement~\ref{it:243} is an expansion of $Y$. 
\end{lemma}
 
\begin{proof}
Firstly, let $U=X\mA\cap Y\mA$. Since $U$ is the intersection of $A$-closed sets, it is also $A$-closed. By Lemma~\ref{2.2H}, $X\mA \setminus Y\mA$ is finite and therefore $X\mA \setminus U$ is finite. So \ref{it:241} implies \ref{it:242}.

Secondly, assume that $U$ is $A$-closed and $X\mA \setminus U$ is finite. 
We will prove statement~\ref{it:243} by induction on the size of $|X\mA \setminus U|$. 
If $|X\mA \setminus U|=0$, then statement~\ref{it:243} holds with $Z=X$. 
Otherwise, $|X\mA \setminus U|>0$ and we choose an element $w\in X\mA \setminus U$ whose length~$|w|$ is maximal. Then the set $U^*=U\cup \{w\}$ is $A$-closed and $|X\mA \setminus U^*|=|X\mA \setminus U|-1$.

By induction, there is an expansion $Z^*$ of $X$ such that $U^*=Z^*\mA$. 
The element $w$ belongs to $Z^*$, otherwise $w$ would have the form $w=z\alpha_{i_1}\cdots \alpha_{i_t}$, where $z\in Z^*$ and $t>0$, and hence $z\in U^*\setminus \{w\}=U$. However, $U$ is $A$-closed and so this would imply that $w\in U$, a contradiction. If we take 
\[Z=(Z^*\setminus \{w\}) \cup\{w\alpha_i \mid 1\leq i\leq n\},\]
then this is again an expansion of $X$ and by the choice of $w$ we have $w\alpha_i\in U$, for all $i$. Therefore $U=Z\mA$ and \ref{it:242} implies \ref{it:243}.

For the last implication: if $U=Z\mA$ for some expansion $Z$ of $X$, then $U=X\mA\cap Y\mA$, 
with $Y=Z$, and so \ref{it:243} implies \ref{it:241}.

Finally, let $U = X\mA \cap Y\mA$ as in statement~\ref{it:241}, so that $U = Z\mA$ by statement~\ref{it:243}.
In particular this means that $Z \subseteq Y\mA$.
As~$Z$ is an expansion of~$X$, it is also an expansion of~$\cxx$; then Lemma~\ref{lem:expbasis} tells us that~$Z$ is a basis of~$V_{n,r}(\cxx)$.
Now suppose that~$Y$ is a basis.
Since~$Y$ is finite, we can apply Lemma~\ref{lem:expbasis} to see that~$Z$ is an expansion of~$Y$. \qedhere

\end{proof}

\begin{corollary}[cf. {\cite[Corollary 1, page 12]{Higg}}] \label{cor:commexp}
Let $B$ and $C$ be finite bases of~$V_{n,r}(\cxx)$.  Then $B$ and $C$ have a common  
expansion  $Z$, which may  be chosen such that 
$Z\mA=B\mA\cap C\mA$. In particular, every finite basis of  $V_{n,r}(\cxx)$ may be obtained from $\cxx$ by an expansion followed by a contraction.
\end{corollary}
\begin{proof}
Let $f$ be the homomorphism from $V_{n,r}(\cxx)$ to $V_{n,|B|}(B)$ defined by mapping $b\in B\subseteq V_{n,r}(\cxx)$ to $b\in V_{n,|B|}(B)$,
 for all $b\in B$. 
As this is a bijection between bases,~$f$ is an isomorphism. 
 Let $C'=Cf$, so $C'$ is a basis for $ V_{n,|B|}(B)$. From Lemma \ref{HigLemma2.4}, $B$ and $C'$ have 
a common expansion $Z'$ such that $B\mA\cap C'\mA=Z'\mA$. Then $B$ and $C$ have common expansion $Z=Z'f^{-1}$, and the
remainder of  the first statement of the lemma follows. The final statement follows on taking $B$ to be an arbitrary finite free generating set and $C=\cxx$. 
\end{proof}

\begin{corollary}[{\cite[Corollary 2, page 12]{Higg}}]\label{cor:H2} $V_{n,r}\cong V_{n,s}$ if and only if $r\equiv s \mod n-1$.
\end{corollary}
\begin{proof}
If $r\equiv s \mod n-1$ then it follows from Lemma \ref{anyexpansion} that $V_{n,r}\cong V_{n,s}$.
Conversely, let $\theta$ be an isomorphism from $V_{n,r}(X)$ to $V_{n,s}(Y)$, where  $X$ and $Y$ are sets of size $r$ and 
$s$, respectively. Then $X\theta$ is 
a basis of  $V_{n,s}(Y)$ of size $r$. From Corollary \ref{cor:commexp}, there is a common expansion
$Z$ of $X\theta$ and $Y$. If $Z$ is a $d$-fold expansion of $X\theta$ and an $e$-fold expansion of $Y$ then
$r+(n-1)d=|Z|=s+(n-1)e$, so $r\equiv s \mod(n-1)$, as claimed.
\end{proof}

We could henceforth restrict to $V_{n,r}$, where $1\le r\le n-1$. However, we 
do not need to do this for what follows here, and it is convenient to allow arbitrary positive values
of $r$, and multiple instances of the same algebra.

\begin{definition}
Let $u,v$ be elements of $V_{n,r}$. Then, $u$ is said to be a \emph{proper initial segment} of $v$ if $v=u\Gamma$ for some non-trivial $\Gamma\in A^*$. If $u=v$ or $u$ is a proper initial segment of $v$
then $u$ is called an \emph{initial segment} of $v$ .
\end{definition}

\begin{lemma}[{\cite[Section 2, Lemma 2.5(i)-(iii)]{Higg}}]\label{HigmanLemma2.5}
Let $B$ be an $A$-basis of $V_{n,r}$ and $V$ a subset of $B\mA$. 
\begin{enumerate}
\item\label{it:251} If $B$ and $V$ are finite, then $V$ is contained in an expansion of $B$ if and only if the following condition is satisfied:
\[\textrm{no element of $V$ is a proper initial segment of another.}\tag{$\dagger$}\label{eq:dagger}\]
\item\label{it:252} If $B$ and $V$ are finite, then $V$ is an expansion of $B$ if and only if \eqref{eq:dagger} is satisfied and for each $u\in B\mA$ there exists $v\in V$ such that one of $u,v$ is an initial segment of the other.
\item\label{it:253} $V$ is a set of free generators for the $\Omega$-subalgebra it generates if and only if \eqref{eq:dagger} is satisfied. 
\end{enumerate}

\end{lemma}

\begin{proof}
\begin{enumerate}
\item If $V$ is contained in an expansion of $B$ then, using Lemma \ref{anyexpansion}.\ref{it:bases1},  \eqref{eq:dagger} is satisfied. 

Suppose $V$ satisfies \eqref{eq:dagger} and write 
\[U=B\mA\setminus \{\text{proper initial segments of elements of $V$} \}.\]
Then \eqref{eq:dagger} implies that $V\subseteq U$. Also, $U$ is $A$-closed and $B\mA \setminus U$ consists of initial segments of the elements of the finite set $V$, so it is finite. Thus, by Lemma \ref{HigLemma2.4}, there is an expansion $Z$ of $B$ such that $U=Z\mA$. Therefore, $U\subseteq Z\mA$, and this implies that $V\subseteq Z$ (for an element of $Z\mA\setminus Z$ has a proper initial segment in $Z\subseteq U$ so it can not be in $V$ by the definition of $U$). Hence, $V$ is contained in an expansion of $B$.
\item If $V$ is an expansion of $B$ then \eqref{eq:dagger} is satisfied and for each $u\in B\mA$ there exists $v\in V$ such that one of $u,v$ is an initial segment of the other.

Suppose $V$ satisfies \eqref{eq:dagger} and for each $u\in B\mA$ there exists $v\in V$ such that one of $u,v$ is an initial segment of the other. By Part 1, $V$ is contained in an expansion $Z$ of $B$. If $V\ne Z$ then there is an element $z\in Z\setminus V$ and hence by the hypothesis there exists $v\in V$ such that one of $v$ or $z$ is an initial segment of the other. But no element of $Z$ can be an initial segment of another, so this is a contradiction and hence $V=Z$.

\item If $V$ is a set of free generators for the $\Omega$-subalgebra it generates then \eqref{eq:dagger} is satisfied. 

Suppose \eqref{eq:dagger} is satisfied. If $V$ is not a free generating set then the same is true of some finite subset $V_0$ and clearly \eqref{eq:dagger} is also satisfied with $V$ replaced by $V_0$. 
Then $V_0\subseteq B_0\mA$ for some finite subset $B_0$ of $B$. 
As \eqref{eq:dagger} holds, it follows from Part \ref{it:251} that $V_0$ is a subset of an expansion
$Z_0$ of $B_0$. However, this means that $V_0$ is a subset of a basis of $V_{n,r}$, a contradiction. 
\qedhere
\end{enumerate}
\end{proof}


\begin{corollary}\label{commonexpansion}
Let $Y_i$ be a finite basis for $V_{n,r}$, for $i=1,\ldots, m$. Then there is
  a unique minimal common expansion $Z$ of all the $Y_i$, and $Z$ 
satisfies $Z\mA=\cap_{i=1}^m(Y_i\mA)$.
\end{corollary}

\begin{proof}
For $m=2$, from Corollary \ref{cor:commexp} we have a common expansion $Z$ of $Y_1$ and 
$Y_2$ such that 
$Z\mA=Y_1\mA \cap Y_2\mA$. 
Furthermore, if $W$ is a common expansion of $Y_1$ and $Y_2$ then, from Lemma \ref{HigmanLemma2.5}, 
$W\subseteq Z\mA$, which implies that $W$ is
an expansion of $Z$. 

For $m>2$, let $Z\mA=\cap_{i=1}^{m-1}(Y_i\mA)$ and $V=Z\mA\cap Y_m\mA$, where we assume inductively that $Z$ is the unique minimal expansion of $Y_1,\ldots ,Y_{m-1}$. From the previous paragraph there exists a unique minimal expansion $W$ of $Z$ and $Y_m$ such that $W\mA=V$. It follows that the result holds for $Y_1,\ldots ,Y_m$ and hence by induction for all $m$. 
\end{proof}

\begin{corollary}\label{lem:Abasis}
Let $Y$ be a finite basis  and let $B$ be an $A$-basis of $V_{n,r}(\cxx)$. If $Y\subseteq B\mA$ then $Y$ is an expansion of $B$: i.e.\ $Y$ is an $A$-basis.
\end{corollary}
\begin{proof}
As $Y\subseteq B\mA$ and $Y$ is a basis, $Y$ satisfies \eqref{eq:dagger} from Lemma \ref{HigmanLemma2.5}.\ref{it:253}. If $u\in B\mA$ then $u\in Y\mA\ml$, 
so for some $\G,\D\in A^*$ and $y\in Y$ we have $u\G=y\D$. As $u\in B\mA$ and $y\in Y\subseteq B\mA$ there exist $b,b'\in B$ and $\L, \L'\in A^*$ such that
$u = b\L$ and  $y = b'\L'$, so $b\L\G=b'\L'\D$, and therefore $b=b'$. Thus $b\L\G=b\L'\D$, so either $u=b\L$ is an initial segment of $y=b\L'$, or 
vice-versa. Hence, from Lemma \ref{HigmanLemma2.5}.\ref{it:252}, $Y$ is an expansion of $B$.  
\end{proof}

\begin{lemma}[{\cite[Section 2, Lemma 2.5(iv)]{Higg}}]
Let $B$ be an $A$-basis of $V_{n,r}$. Let $Y$ and $Z$ be $d$-fold expansions of $B$, for $d\ge 1$. If $Y\ne Z$ then some element of $Y$ is a proper initial segment of an element of $Z$.
\end{lemma}

\begin{proof}
If no element of $Y$ is a proper initial segment of an element of $Z$ then, from Corollary \ref{cor:commexp}, 
$Y\subseteq Z\mA$. Then Lemma \ref{HigmanLemma2.5} implies that $Y$ is an expansion of $Z$. 
However, $Y$ and $Z$ are both $d$-fold expansions of $B$ and thus $Y=Z$. 
This competes the proof.
\end{proof}

\begin{lemma}\label{AJDLEMMAX}
Let $u\in V_{n,r}$ and let $d$ be a non-negative integer. 
\begin{enumerate}
\item\label{it:lx1} If $v\in V_{n,r}$ then $u=v$ if and only if  $u\Gamma=v\Gamma$, for all $\G\in A^*$ of length $d$.
\item If $S$ is an $\Omega$-subalgebra of $V_{n,r}$ then $u\in S$ if and only if $u\Gamma\in S$, for all $\Gamma \in A^*$ of length $d$. 
\end{enumerate}
\end{lemma}
\begin{proof}
\begin{enumerate}
\item If $u=v$ then $u\Gamma=v\Gamma$ for all $\Gamma \in A^*$ of length $d$. 

We shall show that given $d\ge 0$ we have
\begin{equation} \label{eq:same_descendants_implies_same_words} \tag{$\ast$}
	\text{$u,v\in V_{n,r}$ satisfy $u\Gamma=v\Gamma$ for all $\Gamma\in A^*$ of length $d$ \quad$\implies$\quad $u=v$.} 
\end{equation}
If $d=0$ this holds trivially; to proceed we use induction on $d$.
Our hypothesis is that for all $d'$ such that $0\le d'<d$, the implication \eqref{eq:same_descendants_implies_same_words} holds with $d'$ instead of $d$.
Suppose then that $u,v\in V_{n,r}$ and $u\Gamma=v\Gamma$ for all $\Gamma$ of length $d$.
We may uniquely write $\Gamma = \Delta\alpha_i$, where $1 \leq i \leq n$ and $\Delta \in A^*$ has length~$d-1$.
Write~$u\Delta$ as a contraction $u \Delta = u\Delta\alpha_1 \dots u\Delta\alpha_n \lambda$.
Each string~$\Delta\a_j$ has length~$d$, so $u\Delta\alpha_j = v\Delta\alpha_j$ for each~$j$.
Then the contraction above is equal to $v\Delta\a_1 \dots v\Delta\a_n\lambda=v\Delta$, and so $u\Delta = v\Delta$.

Now apply this argument to all strings~$\Gamma$ of length~$d$.
In doing so we will use every length~$d-1$ string~$\Delta$ ($n$ times), and so $u\Delta = v\Delta$ for every $\Delta$ of length~$d-1$.
By the inductive hypothesis we conclude $u=v$.

\item 
The proof is similar to that of part \ref{it:lx1}. \qedhere
\end{enumerate}
\end{proof}

\section{The Higman-Thompson groups \GnrHeading}\label{HigmanThompsonG21}
In this section we define the groups which form the object of study in this paper.
Throughout the remainder of the paper, we assume that $n\ge 2$, 
and that  $V_{n,r}=V_{n,r}(\cxx)=W_\Om(\cxx)/\qf$, 
where $\cxx=\{x_1,\ldots, x_r\}$.  When $r=1$ we let $\cxx=\{x\}$.

When we discuss automorphisms of $V_{n,r}$ we assume that they are 
given by listing the images of a (finite) basis of $V_{n,r}$. 
For instance, let $\psi \in V_{n,r}$ be defined by the bijection $\psi:Y\to Z$, where $Y$ and $Z$ are bases of $V_{n,r}$. 
If we expand $y\in Y$ to form $Y'=Y\setminus \{y\}\cup\{y\alpha_1,\ldots, y\alpha_n\}$, 
the result~$Y'$ is also a basis by Lemma~\ref{lem:expbasis}.
As $y\alpha_i\psi=y\psi\alpha_i=z\alpha_i$ for $i=1,\ldots, n$, we see that the automorphism $\psi$ 
induces an expansion $Z'$ of $Z$ such that $Y'\psi=Z'$.
Thus, if $Y$ and $Z$ are not expansions of $\cxx$, we can find $Y'$ and $Z' = Y'\psi$ contained in $\cxx\mA$ and redefine $\psi$ in terms of $Y'$ and $Z'$.   
In other words, we may always describe an automorphism by a bijection
between $A$-bases.

As bijections between bases are not particularly easy to read, we represent automorphisms using pairs of rooted forests.
An $n$\emph{-ary rooted tree} is a tree with a single distinguished \emph{root} vertex of degree $n$, such that all other vertices
have degree $n+1$ or $1$. If a vertex $v$ is at distance $d\ge 1$ from the root then the $n$ vertices incident to $v$ and not on the
path to the root are its \emph{children}. Vertices of degree $1$ are called \emph{leaves}. 
An  $n$-ary rooted tree is said to be $A$\emph{-labelled} if the edges joining a vertex $v$ to its $n$ children are labelled 
with the elements $\a_i\in A$, so that two edges joining $v$ to different children are labelled differently.    
An $A$-labelled, $r$-rooted, $n$-ary forest is a disjoint union
of $r$ rooted, $A$-labelled, $n$-ary trees. 

Let $T$ be such a forest consisting of trees $T_1, \dotsc, T_r$.
For each $1 \leq i \leq r$, we identify the root of~$T_i$ with the generator $x_i \in \cxx$ of $V_{n,r}(\cxx)$.
We proceed by recursively identifying vertices of $T_i$ with elements of $\{x_i\}\mA\subseteq V_{n,r}$.  
Suppose that $v \in T_i$ is not a leaf, and that~$v$ has been identified with~$x_i\Gamma$ for some $\Gamma \in A^*$.
Then~$v$ has~$n$ children $c_1, \dotsc, c_n$, where~$c_j$ is the child connected to~$v$ by an edge labelled~$a_j$.
For each~$1 \leq j \leq n$ we identify~$c_j$ with~$x_i\Gamma\alpha_j$;
this identifies each vertex of $T$ with a uniquely determined element of $\cxx\mA$. Furthermore, by construction,
the  leaves of $T$ correspond to an expansion of $\cxx$. We use such trees to represent automorphisms, as in the following example.

\begin{example}\label{snf0}
Let $n=2$, $r=1$, $\cxx=\{x\}$ and let $\psi$ be the element of $G_{2,1}$ corresponding to the bijective map between $A$-bases 
 $Y=\{x\alpha_1^2, x\alpha_1\alpha_2,x\alpha_2\}$ and 
$Z=Y\psi=\{x\alpha_1, x\alpha_2\alpha_1,x\alpha_2^2\}$ given by
\[x\alpha_1^2\psi=x\alpha_1, \, x\alpha_1\alpha_2\psi=x\alpha_2\alpha_1, \, 
x\alpha_2\psi=x\alpha_2^2.\]
The $A$-labelled binary trees corresponding to these bases are shown below. The labelling of edges is not shown,
 but edges from a vertex to its children are always ordered from left to right in the order $\a_1,\ldots ,\a_n$. Thus the leaves of the left hand tree 
correspond to $Y$ and the leaves of the right hand tree to $Z$. The numbering below the leaves determines the mapping $\psi$; by taking 
leaf labelled $j$ on the left to leaf labelled $j$ on the right.  
\begin{center}
$\psi:$
\begin{minipage}[b]{0.3\linewidth}
\centering
\Tree  [  [. [.1  ] [.2  ] ]. [.3  ] ] \quad $\longrightarrow$\Tree  [ 1  [ 2 3 ] ]  
\end{minipage}
\end{center}
\end{example}

\begin{definition}[\cite{Higg}]
The \emph{Higman-Thompson group $G_{n,r}$} is  the group of $\Omega$-algebra automorphisms of $V_{n,r}$.
\end{definition}
Note that the largest Thompson group~$V$ is isomorphic to $G_{2,1}$, because the $A$-labelled trees we have described are exactly the tree-pair diagrams used to represent elements of~$V$.

\begin{lemma}[{\cite[Lemma 4.1]{Higg}}]\label{4.1H}
If $\{\psi_{1},\ldots ,\psi_{k}\}$ is a finite subset of $G_{n,r}$ and $X$ is an $A$-basis of $V_{n,r}$, 
then there is a unique minimal expansion $Y$ of $X$ such that $Y\psi_{i} \subseteq X \langle A \rangle$, for $i=1,\ldots ,k$. 
That is, any other expansion of $X$ with this property is an expansion of $Y$.
\end{lemma}

\begin{proof} 
For each $i$, $X\psi_i^{-1}$ is a generating set for $V_{n,r}$, but may not be a subset of $\XA$.  
Let $U_i=\XA \cap X\psi_i^{-1}\mA$. Then, by Lemma \ref{HigLemma2.4}, $U_i$ is $A$-closed and there exists an expansion $Y_i$ of $X$ such that $U_i=Y_i\mA$.
Now, Corollary \ref{commonexpansion} gives a unique minimal common expansion $Y$,
of the $Y_i$'s, 
and  $Y\mA=\cap_{i=1}^k(Y_i\mA)$.
Then, for all $i$, $Y\subseteq Y_i\mA=U_i\subseteq X\psi_i^{-1}\mA$,
so $Y\psi_i\subseteq \XA$. 

Let $Z$ be an expansion of $X$. If $Z\psi_i\subseteq \XA$, for all $i$, then (by the definition of $U_i$)  
$Z\subseteq U_i=Y_i\mA$, so $Z\subseteq \cap_{i=1}^k(Y_i\mA)=Y\mA$. Hence, from Lemma \ref{HigLemma2.4},  
$Z$ is an expansion of $Y$.
\end{proof}
\begin{definition}\label{def:ume}
Let $\{\psi_{1},\ldots ,\psi_{k}\}$ be a finite subset of $G_{n,r}$ and let $X$ be an $A$-basis of $V_{n,r}$. The expansion $Y$ 
of $X$ given by Lemma \ref{4.1H} is 
called the \emph{minimal expansion of} $X$ \emph{associated to} $\{\psi_{1},\ldots ,\psi_{k}\}$.
\end{definition}

\subsection{Semi-normal forms} \label{SNF}

Let $\psi\in G_{n,r}$, let $X$ be an $A$-basis of $V_{n,r}$, 
and $y\in V_{n,r}$. The $\psi$\emph{-orbit} of $y$ is the set $\mathcal{O}_y = \{y\psi^n \mid n\in \ZZ\}$. 
We consider how $\psi$-orbits intersect the $A$-subalgebra $\XA$.  
To this end an $X$\emph{-component} of the $\psi$-orbit of $y$ is  a maximal subsequence   $\cC$ 
of the sequence  $(y\psi^i)_{i=-\infty}^{i=\infty}$   
such that all elements of $\cC$ 
 are in $\XA$. More precisely, $\cC$ must satisfy
\be
\item\label{it:Acomp1} 
if $y\psi^p$ and $y\psi^q$ belong to $\cC$, where $p<q$ then $y\psi^k$ belongs to $\XA$, for all $k$ such that $p\le k\le q$; and 
\item\label{it:Acomp2} 
$\cC$ is a maximal subset of the $\psi$-orbit of $y$ for which statement~\ref{it:Acomp1} holds.
\ee
Note: Higman \cite[Section~9]{Higg} refers to $X$-components as ``orbits in $X\mA$''. 

First we distinguish the 
 five possible types of $X$-component of $\psi$ by giving them names.
\begin{enumerate}
\item \emph{Complete infinite $X$-components.} For any $y$ in such an $X$-component, $y\psi^{i}$ belongs to $X\mA$ for all $i\in \mathbb{Z}$, and the elements $y\psi^{i}$ are all different.
\item \emph{Complete finite $X$-components.} For any $y$ in such an $X$-component, $y\psi^{i}=y$ for some positive integer $i$, and $y, y\psi, \ldots  , y\psi^{i-1}$ all belong to $X\mA$. 
\item \emph{Right semi-infinite $X$-components.} For some $y$ in the $X$-component, $y\psi^{i}$ belongs to $X\mA$ for all $i\geq 0$, but $y\psi^{-1}$ does not. The elements $y\psi^{i}$, $i\geq 0$, are then necessarily all different.
\item \emph{Left semi-infinite $X$-components.}  For some $y$ in the $X$-component, $y\psi^{-i}$ belongs to $X\mA$ for all $i\geq 0$, but $y\psi$ does not. The elements $y\psi^{-i}$, $i\geq 0$, are then necessarily all different.
\item \emph{Incomplete finite $X$-components.} For some $y$ in the $X$-component and some non-negative integer $i$ we have $y, y\psi, \ldots , y\psi^{i}$ belonging to $X\mA$ but $y\psi^{-1}$ and $y\psi^{i+1}$ do not.
\end{enumerate}

\begin{example}\label{orbiteg}
Let $n=2$, $r=1$, $\cxx=\{x\}$. Let our bases be
\[Y=\{x\alpha_1^3,x\alpha_1^2\alpha_2,x\alpha_1\alpha_2,x\alpha_2\alpha_1,x\alpha_2^2\}
\quad \text{and} \quad
Z=\{x\alpha_1^2,x\alpha_1\alpha_2\alpha_1,x\alpha_1\alpha_2^2,x\alpha_2^2,x\alpha_2\alpha_1\}. \]
Define the automorphism $\psi$ by $Y\psi=Z$, with the  ordering given above.
\begin{center}
$\psi:$
\begin{minipage}[b]{0.5\linewidth}
\centering
\Tree [ [ [ [.1  ] [.2  ] ] [.3  ] ]   [ 4 5  ] ]  \quad $\longrightarrow$\Tree [  [   [.1  ] [ [.2  ] [.3  ] ]   ] [ 5 4  ] ]  
\end{minipage}
\end{center}
Then $Y$ is the minimal expansion of $\cxx$ associated to $\psi$.
Take the basis~$X$ to be just $X = \cxx$.  
The $X$-component of  $x\alpha_1^3$ is left semi-infinite
\[\cdots \mapsto x\alpha_1^4\mapsto x\alpha_1^3\mapsto x\alpha_1^2,\]
and the $X$-component of $x\alpha_1\alpha_2$ is  right semi-infinite:
\[x\alpha_1\alpha_2  \mapsto x\alpha_1\alpha_2^2\mapsto x\alpha_1\alpha_2^3\mapsto \cdots.\]
 The $X$-component of $x\alpha_1^2\alpha_2$ is complete infinite
\[ \cdots\mapsto x\alpha_1^4\alpha_2\mapsto x\alpha_1^3\alpha_2 \mapsto x\alpha_1^2\alpha_2  \mapsto x\alpha_1\alpha_2\alpha_1  \mapsto x\alpha_1\alpha_2^2\alpha_1\mapsto\cdots , \]
 and $(x\alpha_2\alpha_1, x\alpha_2^2)$ is a  complete finite $X$-component. 
We have $x\a_2=x\a_2\a_1xa_2^2\l$, so $x\a_2\psi= xa_2^2x\a_2\a_1\l$ and $x\a_2\psi^2=x\a_2$; therefore
$(x\a_2)$ is an incomplete finite $X$-component.
\end{example}

Let $\psi\in G_{n,r}$, let $X$ be an $A$-basis of $V_{n,r}$, let 
$Y$ be the minimal expansion of $X\mA$ associated to $\psi$ and let $Z=Y\psi$. Then, as discussed above,
$Y$ and $Z$ are both expansions of $X$. 
From Lemma \ref{HigLemma2.4}, both $X\mA\setminus Z\mA$ 
and 
$X\mA\setminus Y\mA$ 
are finite. Furthermore, as $|Y|=|Z|$, both $X$ and $Y$ are $d$-fold expansions, for some $d$, so  
$|X\mA\setminus Z\mA|= |X\mA\setminus Y\mA|$. 

By definition 
$Y\mA= \XA\cap \XA\psi^{-1}$, and moreover
$\psi$ maps no proper contraction of $Y$ into $X\mA$. Hence 
\[Z\langle A \rangle= Y\langle A \rangle\psi =X\langle A \rangle\psi \cap X\langle A \rangle.\]
Thus, if $u\in X\mA\setminus Z\mA$ then $u\not\in X\mA\psi$, so $u\psi^{-1}\not\in X\mA$ and hence $u$ is an initial element either of an incomplete finite $X$-component or of a right semi-infinite $X$-component \emph{i.e.} in an $X$-component of type (3) or (5). Similarly,  if $v\in X\mA\setminus Y\mA$ then $v\not\in X\mA\psi^{-1}$, so $v\psi\not\in X\mA$ and hence $v$ is a terminal element either of an incomplete finite $X$-component or of a left semi-infinite $X$-component \emph{i.e.} in an $X$-component of type (4) or (5).

If $\mathcal{C}$ is an $X$-component of type (3) or (5), then by definition $\mathcal{C}$  has an initial element $u$: that is $u\psi^{-1}\not\in X\mA$. Then $u\not\in X\mA\psi$, and so $u\in X\mA\setminus Z\mA$. 
Similarly, if $\mathcal{C}$ is an $X$-component of type (4) or (5), then  $\mathcal{C}$ has a terminal element  $v$: 
that is  $v\psi\not\in X\mA$. Again,  $v\not\in X\mA\psi^{-1}$ 
and so $v\in X\mA\setminus Y\mA$.

Let $u$ be an initial element of an incomplete finite $X$-component $\mathcal{C}$. By the above, $u\in X\mA\setminus Z\mA$ and by definition of an incomplete finite $X$-component, there is some non-negative integer $k$ such that $u,u\psi,\ldots ,u\psi^k$ all belong to $\XA$ but $u\psi^{k+1}$ does not. Since $u\psi^{k}$ is the terminal element of the incomplete finite $X$-component $\mathcal{C}$, 
we have  $u\psi^{k}\in X\mA\setminus Y\mA$. Therefore, the initial elements of incomplete finite $X$-components in $X\mA\setminus Z\mA$ and terminal elements of incomplete finite $X$-components in $X\mA\setminus Y\mA$ pair up. 

Given that the initial and terminal elements of the incomplete finite $X$-components must be in one-to-one correspondence, all other elements of
 $|X\mA\setminus Z\mA|$ (respectively $ |X\mA\setminus Y\mA|$) 
are initial (respectively terminal) elements in right (respectively left) semi-infinite $X$-components. 
Hence there are as many right semi-infinite $X$-components  as left semi-infinite $X$-components. 

The above is  summarised in a lemma.
\begin{lemma}[{\cite[Lemma 9.1]{Higg}}]\label{9.1H}
Let $\psi$ be an element of $G_{n,r}$ and let $X$ be an $A$-basis of $V_{n,r}$. There are only finitely many $X$-components of $\psi$ of types (3--5) 
and there are as many of type (3) as of type (4). If $Y$ is the minimal expansion of $X\mA$ associated to $\psi$ and $Z=Y\psi$
then 
\begin{itemize}
\item $Y\mA=X\mA\cap X\mA\psi^{-1}$ and $Z\mA=X\mA\psi\cap X\mA$;
\item $X\mA\bh Z\mA$ is exactly the set of initial elements of $X$-components of types (3) or (5); and
\item $X\mA\bh Y\mA$ is exactly the set of terminal elements of $X$-components of types (4) or (5).
\end{itemize}
\end{lemma}
\begin{example}\label{orbiteg1}
In Example \ref{orbiteg}, we have  $\XA\setminus Z\mA=\{x, x\a_1, x\a_1\a_2, x\a_2\}$ and  
$\XA\setminus Y\mA=\{x, x\a_1, x\a_1^2,x\a_2\}$. The incomplete finite $X$-components are 
$(x)$, $(x\a_1)$ and $(x\a_2)$, while $x\a_1\a_2$ is an initial element of a right semi-infinite $X$-component and
$x\a_1^2$ is a terminal element of a left semi-infinite $X$-component. All other $X$-components of elements of $\XA$ are 
complete.
\end{example}

\begin{definition}[{\cite[Section 9]{Higg}}]
An element $\psi$ of $G_{n,r}$ is in \emph{semi-normal form} with respect to the $A$-basis $X $ if
no element of $\XA$ 
is in an incomplete finite  
$X$-component of $\psi$.
\end{definition}

\begin{lemma}[{\cite[Lemma 9.2]{Higg}}]\label{9.9H}
Let $\psi\in G_{n,r}$ and let $X$ be an $A$-basis of $V_{n,r}$. There exists an expansion of $X$ with respect to which $\psi$ is in semi-normal form.
\end{lemma}
\begin{proof}
Let $\psi\in G_{n,r}$. 
We prove the lemma by induction on the number of elements  in $X\mA$ which belong to an incomplete finite $X$-component.
Note that Lemma \ref{9.1H} shows us that this number is finite.
If there are no such elements then we are done. 

Suppose then that there exists an element $u$ in $X\mA$ which belongs to an 
incomplete finite $X$-component. Thus, there exist $y\in X$ and  $\G\in A^*$ such that $u=y\G$ and 
  some minimal $m,k\in \mathbb{N}_0$ such that $u\psi^{-(m+1)},u\psi^{k+1}\not\in X\mA$.
It follows that $y\psi^{-(m+1)},y\psi^{k+1}\not\in X\mA$, so that $y$ is also
in an incomplete finite $X$-component. Let $X'$ be the simple expansion $X' =X\backslash \{y\} \cup\{y\alpha_1,\ldots ,y\alpha_n\}$. 
Then $X'$ is a $A$-basis for $V_{n,r}$ and $\XA\setminus X'\mA=\{y\}$.
Thus the number of elements of $X''\mA$ in an incomplete finite $X''$-component is one less than 
the number of elements of  $X\mA$ in an incomplete finite $X$-component. Hence, by induction, there exists an expansion of $X$ with respect to 
which $\psi$ is in semi-normal form. 
\end{proof}

\begin{remark} \label{rem:snf_start_point}
Continuing the discussion above Lemma \ref{9.1H}, observe that if $u\in \XA$ and $u\notin Y\mA\cup Z\mA$ then
$u$ is both the initial and terminal element of an $X$-component of $\psi$; so $(u)$ constitutes an incomplete finite
$X$-component.  Therefore, when implementing the argument of Lemma \ref{9.9H} to find a semi-normal form for $\psi$,
we may pass immediately to a minimal expansion containing no elements of $\XA\setminus (Y\mA\cup Z\mA)$: that is an expansion
minimal amongst those contained in $Y\mA\cup Z\mA$. 
\end{remark}
\begin{example}\label{ex:snf1}
Let $n=2$, $r=1$, $\cxx=\{x\}$ and let $\psi$ be the automorphism of Example \ref{snf0}.
Here $Y=\{x\alpha_1^2, x\alpha_1\alpha_2,x\alpha_2\}$ is the minimal expansion of $\cxx$ associated to $\psi$ and 
$Z=Y\psi=\{x\alpha_1, x\alpha_2\alpha_1,x\alpha_2^2\}$. 
In this example, $\xA\setminus (Y\mA\cup Z\mA)=\{x\}$ and 
the minimal expansion of $\cxx$ not containing $x$ is $X=\{x\a_1,x\a_2\}$. Then $Y$ remains the minimal expansion
of $X$ associated to $\psi$,  
$X\mA\setminus Z\mA=\{x\a_2\}$ and $X\mA\setminus Y\mA=\{x,x\a_1\}$.  As $x\a_1$ is the terminal element
of a  left semi-infinite $X$-component, while  $x\a_2$ is the initial element
of a  right semi-infinite $X$-component it follows that $\psi$ 
is in semi-normal form with respect to $X$. 
\end{example}
\begin{example}\label{ex:snf2}
Let $n=2$, $r=1$, $\cxx =\{x\}$ and let $\psi$ be the element of $G_{2,1}$ corresponding to the bijective map:
\[x\alpha_1^2\psi=x\alpha_2^2 ,\ x\alpha_1\alpha_2\psi=x\alpha_2\alpha_1 , \ x\alpha_2\psi=x\alpha_1.\]
\begin{center}
$\psi:$
\begin{minipage}[b]{0.3\linewidth}
\centering
\Tree  [  [. [.1  ] [.2  ] ]. [.3  ] ] \quad $\longrightarrow$\Tree  [ 3  [ 2 1 ] ]  
\end{minipage}
\end{center}
Again, $Y=\{x\alpha_1^2, x\alpha_1\alpha_2,x\alpha_2\}$ is the minimal expansion of $\cxx$ associated to $\psi$ and setting 
$Z=Y\psi=\{x\alpha_1, x\alpha_2\alpha_1,x\alpha_2^2\}$, 
the minimal expansion of $\cxx$ contained in $Y\mA\cup Z\mA$ is  $X_1=\{x\a_1,x\a_2\}$; and 
$Y$ is still the minimal expansion of $X_1$ associated to $\psi$. 
 However   
$(x\a_2, x\a_1)$ is an incomplete finite $X_1$-component, so $\psi$ is not in semi-normal form with respect to $X_1$. 
As $x\a_1$ is in an incomplete finite $X_1$-component, we first take the simple expansion of $X_1$ at $x\a_1$, giving  $X_2=Y$. 
As  $x\a_2\psi=x\a_1\notin X_2\mA$, $(x\a_2)$ is now an incomplete finite $X_2$-component, so $\psi$ is not in semi-normal form with respect to $X_2$.
 We take a further  
simple expansion of $X_2$ at $x\alpha_2$,  to obtain a new $A$-basis $X_3=\{x\alpha_1^2,  x\alpha_1\alpha_2, 
x\alpha_2\alpha_1, x\alpha_2^2\}$.  Then $\psi$ maps $X_3$ to itself: 
\[x\alpha_1^2\psi=x\alpha_2^2 ,\ x\alpha_1\alpha_2\psi=x\alpha_2\alpha_1 , \ x\alpha_2\alpha_1\psi=x\alpha_1^2,x\alpha_2^2=x\alpha_1\alpha_2.\]
\begin{center}
$\psi:$
\begin{minipage}[b]{0.4\linewidth}
\centering
\Tree  [  [. [.1  ] [.2  ] ]. [3 4  ] ] \quad $\longrightarrow$\Tree [ [ [.3  ] [.4  ] ] [2 1 ]  ]  
\end{minipage}
\end{center}
As all elements of $X_3$ are in the same complete  finite $X_3$-component, $\psi$ is in semi-normal form with respect to $X_3$. 
The minimal expansion of $X_3$ associated to $\psi$ is just $X_3$.
\end{example}
\begin{example}\label{orbiteg2}
The automorphism $\psi$ of Example \ref{orbiteg1} is not in semi-normal form with respect to $X$ or 
$X_1=\{x\a_1,x\a_2\}$, as both  $x\a_1$ and $x\a_2$ are in incomplete finite $X$-components. However, $\psi$ is in 
semi-normal form with respect to $X_2=\{x\alpha_1^2,  x\alpha_1\alpha_2, 
x\alpha_2\alpha_1, x\alpha_2^2\}$. The minimal expansion of $X_2$ associated to $\psi$ is the $A$-basis $Y$ of 
Example \ref{orbiteg}.
\end{example}
The following, which follows directly from the definitions, 
summarises the possibilities for the intersection with $X\mA$ of the orbit of an element under 
an automorphism in semi-normal form.
\begin{corollary}\label{cor:orbit-type}
Let $\psi$ be an element of $G_{n,r}$ in semi-normal form with respect to the $A$-basis $X$, let $v\in V_{n,r}$ and let 
$\cO_v$ be the $\psi$-orbit of $v$. Then $\cO_v$ has one of the following six types.
\be
\item\label{it:orbit1} $\mathcal{O}_v \cap X\mA=\emptyset$.
\item\label{it:orbit2} $\mathcal{O}_v$ is finite and $O_v\subseteq X\mA$, so $\mathcal{O}_v$ is a complete finite $X$-component.
\item\label{it:orbit3} $\mathcal{O}_v$ is infinite and $O_v\subseteq X\mA$, so $\mathcal{O}_v$ is a complete infinite $X$-component.
\item\label{it:orbit4} $\mathcal{O}_v \cap X\mA$ consists of a unique left semi-infinite $X$-component. 
\item\label{it:orbit5} $\mathcal{O}_v \cap X\mA$ consists of a unique right semi-infinite $X$-component.
\item\label{it:orbit6}  $\mathcal{O}_v \cap X\mA$ is the disjoint union of a left semi-infinite $X$-component and a right semi-infinite $X$-component.
\ee
\end{corollary}
\begin{remark}\label{rem:pondexist}
As can be seen from Example \ref{ex:pond} below,
there are automorphisms for which orbits of the final type in this list exist. 
In fact we shall show in Example~\ref{cannot_expand_to_hide_pond} that there exist automorphisms which have such
orbits with respect to \emph{every} semi-normal form. 
This means that \cite[Lemma~9.6]{Higg} is false.
Consequently, the algorithms \cite[Lemma~9.7]{Higg}
for determining if two elements of~$V_{n,r}$ belong to a single orbit, 
and  \cite[Theorem~9.3]{Higg} for conjugacy of automorphisms are incomplete. 
\end{remark}

\begin{definition}
Let $\psi$ be an element of $G_{n,r}$ in semi-normal form with respect to the $A$-basis $X$, and let   $\cO$ be a $\psi$-orbit of type \ref{it:orbit6}, as given in Corollary \ref{cor:orbit-type}. Then $\cO$ is called a  
\emph{\pond* orbit} with respect to $X$. The subsequence~$P \subset \cO$ of elements not in $X\mA$ is 
called a \emph{\pond*}.
The \emph{width of~$P$} is one more than number of elements in~$P$; this is the number of times we need to apply $\psi$ to get from the endpoint of one semi-infinite $X$-component to the other.
\end{definition}

\begin{example}\label{ex:pond}
Let $n=2$, $r=1$ and $V_{2,1}$ be free on $\cxx = \{x\}$. Let
 \[ Y = \{ x\a_1^4, x\a_1^3\a_2, x\a_1^2\a_2\a_1, x\a_1^2\a_2^2, x\a_1\a_2, x\a_2\a_1, x\a_2^2 \}, \] 
 \[ Z = \{ x\a_1^2, x\a_1\a_2\a_1^2, x\a_1\a_2\a_1\a_2, x\a_1\a_2^2\a_1, x\a_1\a_2^3, x\a_2\a_1, x\a_2^2 \} \] 
and let $\psi \in G_{2,1}$ be determined by the bijection $Y \to Z$ illustrated below.

\begin{center}
$\psi:$
\begin{minipage}[b]{0.7\linewidth}
\centering
\Tree [ [ [ [1 2  ] [ 3 4  ] ] [.5  ] ]   [ 6 7  ] ]  \quad $\longrightarrow$\Tree [  [   [.1  ] [ [  2  5  ] [7 6  ] ]   ] [ 4 3  ] ]  
\end{minipage}
\end{center}
As usual, $Y$ is the minimal expansion of $\cxx$ associated to $\psi$ and $Z = Y\psi$. The minimal expansion of $\cxx$ contained in 
$Y\mA \cup Z\mA$ is $X = \{ x\a_1^2, x\a_1\a_2, x\a_2\a_1, x\a_2^2 \}$. Two of these elements are endpoints of semi-infinite $X$-components, 
whereas the other two belong to complete infinite $X$-components. 
\begin{equation} \label{eqn:pond_comp_LSI}
	\cdots    \mapsto x\a_1^4        \mapsto x \a_1^2
\end{equation}
\begin{equation} \label{eqn:pond_comp_RSI}
	x\a_1\a_2 \mapsto x(\a_1\a_2)^2  \mapsto \cdots
\end{equation}
\begin{equation} \label{eqn:pond_comp_complete_inf_1}   
	\cdots \mapsto x\a_1^4\a_2^2  \mapsto x\a_1^2\a_2^2 \mapsto x\a_2\a_1 \mapsto x\a_1\a_2^3 \mapsto x(\a_1\a_2)^2\a_2^2 \mapsto \cdots
\end{equation}
\begin{equation} \label{eqn:pond_comp_complete_inf_2}
\cdots \mapsto x\a_1^4\a_2\a_1 \mapsto x\a_1^2\a_2\a_1 \mapsto x\a_2^2 \mapsto x\a_1\a_2^2\a_1 \mapsto x(\a_1\a_2)^2\a_2\a_1 \mapsto \cdots
\end{equation}
Thus $\psi$ is in semi-normal form with respect to $X$. Now let us compute the $\psi$-orbit of the element $x\a_1^2\a_2$.
\begin{equation} \label{eqn:pond_example}
	\cdots x\a_1^6\a_2 \mapsto x\a_1^4\a_2 \mapsto x \a_1^2 a_2 \mapsto x\a_2^2 x\a_2\a_1 \lambda \mapsto x\a_1\a_2^2 \mapsto x(\a_1\a_2)^2\a_2 \mapsto \cdots
\end{equation}
Figure~\ref{fig:pond_orbit_illustration} illustrates the orbit~\eqref{eqn:pond_example}, which consists of two semi-infinite $X$-components and a single element $x\a_2^2 x\a_2\a_1\lambda$ (the \pond*) outside of $X\mA$.
In this case, the \pond* has width $1 + 1 = 2$.

\begin{figure}
\centering
\includegraphics{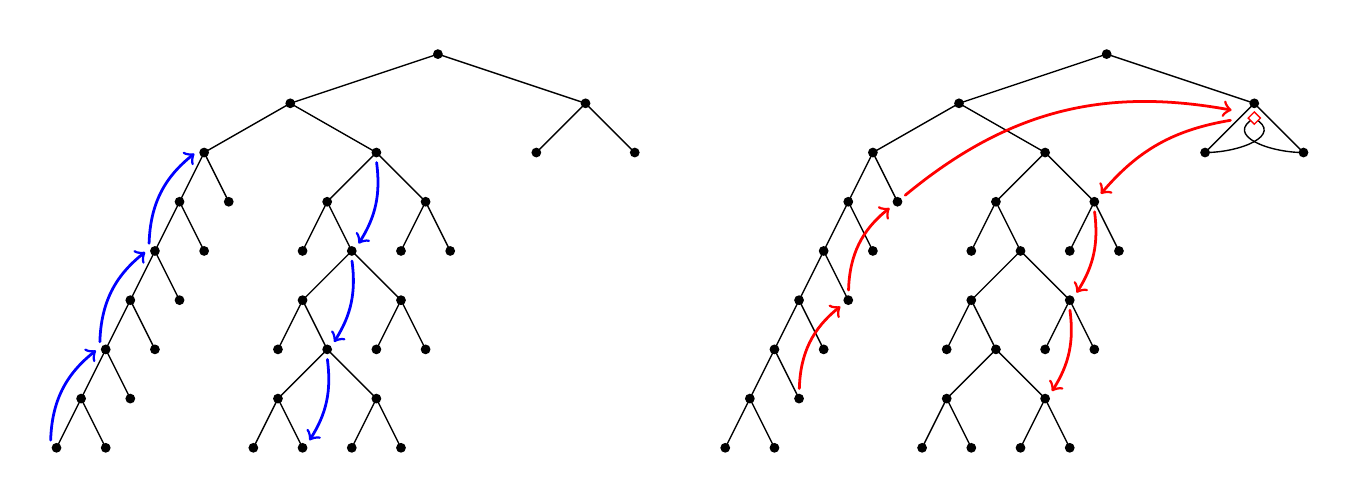}
\caption[An illustration of the \pond* orbit and the semi-infinite components above it.]{
The binary trees above represent a finite subset of $\cxx\mA$, as described in the introduction to Section~\ref{HigmanThompsonG21}.
On the left we have annotated this tree, highlighting the semi-infinite $X$-components \eqref{eqn:pond_comp_LSI} and \eqref{eqn:pond_comp_RSI}.
Below these components sit the \pond* orbit~\eqref{eqn:pond_example}, which is shown on the right tree.
Note that the element $w = x\alpha_2^2x\alpha_2\alpha_1\lambda \notin \cxx\mA$ does not correspond to a vertex of this tree; we have represented it as a `phantom' vertex
\tikz \node[inner sep=1.8, fill=white, draw=red, rotate=45]{};
below $x\alpha_2$ whose left child is~$x\alpha_2^2$ and whose right child is~$x\alpha_2\alpha_1$---a `twisted' version of $x\alpha_2$.
}
\label{fig:pond_orbit_illustration}
\end{figure}
\end{example}

\begin{lemma}[{\cite[Lemma 9.3]{Higg}}]\label{ABC}
Let $\psi$ be an element of $G_{n,r}$ in semi-normal form with respect to the $A$-basis $X $. 
Suppose that $x$ is an element of $X $. Exactly one of the following holds.
\be[(A)]
\item There exists $\Gamma\in A^*$ such that $x\Gamma$ is in a complete finite $X$-component. In this case 
$x$ itself belongs to a complete finite $X$-component, which consists of elements of $X $,  and we say $x$ is of type (A).
\item There exist  $\Gamma, \Delta \in A^*$, with $\Gamma\neq \Delta$, such that 
$x\Gamma$ and $x\Delta$ belong to the same $X$-component. In this case there exists $\Lambda\in A^*$ and $m\in \ZZ \setminus \{0\}$ with $|m|$ 
minimal, 
such that  $x\psi^{m}=x\Lambda$; we say $x$ is of type (B). If $m > 0$ then the $X$-component containing $x$ is right semi-infinite; 
if $m < 0$  then the $X$-component containing $x$ is left semi-infinite.
\item
$x$ is not of type (A) or (B) above and there exists some $z \in X $ of type (B) and non-trivial 
$\Delta \in \langle A \rangle$ such that $x\psi^{i}=z\Delta$. In this case the $X$-component containing $x$ is infinite; and
 we say $x$ is of type (C).
\ee 
\end{lemma}
\begin{proof}
\be[(A)]
\item
If $x$ belongs to an infinite $X$-component of $\psi$ (of types (1), (3) or (4) that is), then so does $x\G$, a contradiction. 
As $\psi$ is in semi-normal form with respect to $X $ it follows that $x$ is in a complete finite $X$-component. Let $d$ be 
the smallest positive integer such that $x\psi^d=x$.
For each $1\le i\le d-1$ write $x\psi^i =z \D$ for some $z\in X $ and $\D\in A^*$.  
Then $z$ must also belong to a complete finite $X$-component, so we can write $z\psi^{d-i}=y\G$ for some  $y\in X $ and $\G\in A^*$. 
Then $x =x\psi^d=z\D\psi^{d-i}=z\psi^{d-i}\D=y\G\D$. From Lemma \ref{HigmanLemma2.5}, we have $y=x$ and $\G=\D=\e$, so 
$x\psi^i=z\in X $, as claimed.
\item
If $x$ belongs to a finite $X$-component then, from (A), the $X$-component of $x\G$ consists of elements
$z\G$, where $z\in X $, contrary to the hypotheses of (B). Therefore $x$ belongs to an infinite $X$-component of $\psi$. 
Without loss of generality we may assume that there is  $i>0$ such that $x\G\psi^i=x\D$. 
Suppose first that $x\psi^k\in X\mA$, for all $k\ge 0$. Then 
$x\psi^i=v\L$, for some $v\in X $ and $\L\in A^*$, and thus  $x\D=x\G\psi^i=v\L\G$;  so $v=x$ and $\D=\L\G$, and 
we obtain $x\psi^i=x\L$.

Similarly, if $x\psi^{-k}\in X \mA$, for all $k\ge 0$, then $x\psi^{-i}=x\L'$, 
for some $\L'\in A^*$, with $\G=\L'\D$. Note that if $x\psi^k\in X \mA$ for all $k$, then  $x=x\L\L'$, which forces  
 $\L=\L'=\e$, so $\G=\D$, a contradiction. Hence the final statement of (B) holds.
\item
In this case $x$ must belong to an infinite $X$-component, as (A) does not hold. As $X $ is finite there is $z\in X $ such that 
$z\G$ and $z\D$ belong to the $X$-component of $x$, for distinct $\G$ and $\D$ in $A^*$; and then $z$ is of type (B), as required.
\ee
\end{proof}

\begin{definition}\label{Proper}
Let $u\in V_{n,r}$ and $\psi \in G_{n,r}$. If $u\psi^d=u\G$ for some $d\in \mathbb{Z}\setminus \{0\}$ and some $\G\in A^*\setminus \{1\}$,  
then $u$ is a 
\emph{characteristic element} for $\psi$.  
If $u$ is a characteristic element for $\psi$ 
then the \emph{characteristic} of $u$ is the pair $(m,\Gamma)$ such that 
$m\in \mathbb{Z}\setminus \{0\}$, $\Gamma\in A^*$ with
\begin{itemize}
\item $u\psi^m=u\Gamma$ and
\item for all $i$ such that $0<|i|<|m|$, $u\psi^i\not\in u\mA$.
\end{itemize}
In this case $\Gamma$ is called the \emph{characteristic multiplier} and $m$ is the \emph{characteristic power} 
for $u$, with respect to $\psi$. 
\end{definition}

From the definition, 
if $\psi$ is in semi-normal form with respect to $X$ then an element $x\in X$ is of type (B) if and only if $x$ is a characteristic element: 
in which case it 
follows from Lemma \ref{AJDLEMMA2} below that the $\psi$-orbit of $x$ is of type \ref{it:orbit4} or \ref{it:orbit5} in Corollary \ref{cor:orbit-type}.
On the other hand, if $x\in X$ has type (C) then the $\psi$-orbit of $x$ may be of types \ref{it:orbit3}, \ref{it:orbit4}, \ref{it:orbit5}  or \ref{it:orbit6}
in Corollary \ref{cor:orbit-type}.
\begin{example}
In Example \ref{orbiteg2}, the automorphism $\psi$  is in semi-normal form with respect to
an $A$-basis $X$. The elements $x\a_2\a_1$ and $x\a_2^2$ of $X$ are of type (A). 
The element $x\alpha_1^2\in X$ is of type (B) with characteristic 
$(-1,\a_1)$, while  $x\alpha_1\alpha_2\in X$ is of type (B) with characteristic $(1,\alpha_2)$; and both of these elements are endpoints of their semi-infinite $X$-components.

In Example \ref{ex:pond} the elements $x\a_2\a_1$ and $x\a_2^2$ of $X$ are of type (C), are not characteristic and belong to complete infinite $X$-components. The elements $x\a_1^2\a_2$ and $x\a_1\a_2^2$ in the \pond* orbit \eqref{eqn:pond_example} are also type (C) and non-characteristic, but belong to semi-infinite $X$-components.
\end{example}

\begin{lemma}\label{lem:char}
If $u\in V_{n,r}$ is a characteristic element  for $\psi \in G_{n,r}$ then
\begin{enumerate}
\item\label{it:char1} the characteristic $(m,\Gamma)$ is uniquely determined, and 
\item\label{it:char2} if $v$ is in the same $\psi$-orbit as $u$ then $v$ is a characteristic element with the same characteristic as $u$. 
\end{enumerate}
\end{lemma}
\begin{proof}
To see part~\ref{it:char1},  suppose that $u$ has characteristic $(m,\Gamma)$. 
If $u\psi^{m'}=u\Delta$ and for all $0< |k| < |m'|$ we have 
 $u\psi^k\not\in u\mA$, then $|m'|\ge |m|$ by Definition \ref{Proper}, so  $m=\pm m'$. 
If $u\psi^{-m}=u\Delta$ then $u=u\psi^m\Delta=u\Gamma\Delta$, which cannot happen as $\Gamma \neq 1$. 

For part~\ref{it:char2}, let $u\psi^r=v$. For all $k$ such that $u\psi^k=u\Delta$ with $\Delta\in A^*$, we have 
\[v\psi^k=u\psi^r\psi^k=u\psi^k\psi^r=u\Delta\psi^r=u\psi^r\Delta=v\Delta.\]
Interchanging $u$ and $v$ we see also that whenever $v\psi^k=v\Delta$ then $u\psi^k=u\Delta$.
\end{proof}
From Lemma \ref{lem:char}, if a $\psi$-orbit has a characteristic element, then every $X$-component of this $\psi$-orbit contains 
a characteristic element, and all these elements have the same characteristic. Bearing this in mind we make the following
definition. 

\begin{definition}\label{charOrbit}
Let $\psi \in G_{n,r}$ have an $X$-component~$\mathcal{C}$ containing a characteristic element $u$. 
Then we define the \emph{characteristic} of $\mathcal{C}$ to be equal to the characteristic of $u$.
\end{definition}


\begin{theorem}[{\cite[Theorem 9.4]{Higg}}]\label{T9.4H}
Let $\psi\in G_{n,r}$ be in semi-normal form with respect to $X$. Then $\psi$ 
is of infinite order if and only if it has a characteristic element~$u$. 
Moreover, if $\psi$ is of infinite order then we may assume that~$u \in X$.
\end{theorem}

\begin{proof}
If $u$ is a   characteristic element for $\psi$ with characteristic $(m,\Gamma)$ then $u\psi^{m}=u\Gamma$, so $u\psi^{mq}=u\G^q$. So for sufficiently large $q$, $u\psi^{mq}\in X\mA$. Then $u\psi^{mq}$ also has characteristic $(m, \Gamma)$ by Lemma~\ref{lem:char}.
Write $u\psi^{mq}=x\D$, for
some $x\in X$ and $\D\in A^*$. Now $x\D\G=u\psi^{mq}\G=u\psi^{m(q+1)}=x\D\psi^m$, so from Lemma \ref{ABC}, $x$ has type (B). Thus we may
assume  
$u\in X$.  Now 
\begin{align*}
u\psi^{mj}&=u\psi^m\psi^{m(j-1)}=u\Gamma\psi^{m(j-1)}=u\Gamma\psi^m\psi^{m(j-2)}\\
&=u\psi^m\Gamma\psi^{m(j-2)}=u\Gamma^2\psi^{m(j-2)}=\cdots =u\Gamma^j,
\end{align*}
for $j\in \mathbb{N}$. Since $\Gamma$ is a characteristic multiplier, 
the elements $u\Gamma^{j}$ are all different for $j\in \mathbb{N}$, so $\psi$ has infinite order.

Conversely, if $\psi$ has no  characteristic element, then certainly there are none in $X$,  so $X$ 
has no elements of type (B) nor type (C). Thus all elements of $X$ are of type (A), 
as $\psi$ is in semi-normal form; whence $\psi$ is a permutation of $X$ and has finite order. 
\end{proof}

\begin{lemma}\label{AJDLEMMA2} Let $\psi$ be in semi-normal form with respect to an $A$-basis $X$ and let $u\in V_{n,r}$. If 
$u$ has characteristic $(m,\Gamma)$ then
the $\psi$-orbit of $u$ has precisely one $X$-component, which 
is semi-infinite (right semi-infinite if $m>0$ and left semi-infinite if $m<0$) and 
consists of elements of the form $x\L$, where $x\in X$ is of type (B) and $\L\in A^*$.

Furthermore, if $x\L$ belongs to
the $X$-component  of the $\psi$-orbit of  $u$, where  $x\in X$ and $\L\in A^*$, then $x$ has characteristic $(m,\G_1\G_0)$,
where $\G=\G_0\G_1$, $\L=(\G_1\G_0)^p\G_1=\G_1\G^p$, $p\ge 0$, and  $\G_0$ is non-trivial. 
\end{lemma}
\begin{proof}
As $u\psi^m=u\G$ we have $u\psi^{mq}=u\G^q$, for all integers $q \geq 0$, and choosing $q$ sufficiently large $u\G^q\in \XA$.
Thus we may assume that $u\in \XA$.
Let $\mathcal C$ denote the $X$-component containing $u$ and write $u=x\Lambda$, where $x\in X$ and $\Lambda\in A^*$.

Assume first that $m>0$. As $u$ has characteristic $(m,\G)$, both $x\L$ and $x\L\G$ belong to $\mathcal{C}$, so
$x$ is of type (B) by Lemma \ref{ABC}. 
Suppose there is an integer $K\ge 0$ such that $u\psi^{-k}\in \XA$, for all $k\ge K$. 
(That is, suppose that the $\psi$-orbit of $u$ contains a left semi-infinite $X$-component.)
Let $\Lambda=\Lambda_0\G^t$, where $\Lambda_0$ 
has no terminal segment equal to $\G$. Then, for $j$ such that $m(j+1)\ge K$ and $j\ge t$, 
$u\psi^{-m(j+1)}\in \XA$, so $u\psi^{-m(j+1)}=z\Xi$ for some $z\in X$ and $\Xi\in A^*$.
From Lemma \ref{lem:char} we see that $z\Xi$ has characteristic $(m,\G)$. Hence 
\[z\Xi\G^{j+1}=z\Xi\psi^{m(j+1)}=u=x\Lambda_0\G^t,\]
which implies  $z=x$ and $\Xi\G^{j-t+1}=\Lambda_0$, a contradiction. As $\psi$ is in semi-normal form with
respect to $X$ and $\mathcal C$ is not a complete $X$-component, the $\cC$ must be right semi-infinite. 
We have just shown  the $\psi$-orbit of $u$ contains no left semi-infinite $X$-component, so $\cC$ is the unique $X$-component
of this $\psi$-orbit. 

For the second part of the lemma,
suppose $x$ has characteristic $(k,\Omega)$. The $X$-component of $x$ cannot be left semi-infinite, or else 
$x\Lambda\psi^{-i}\in \XA$ for all $i \ge 0$; this would mean that $\cC$ is not right semi-infinite. Hence 
$x$ is in a right semi-infinite $X$-component and $k>0$. 
If $\Lambda=\Omega^j\Lambda_1$ then $x\Lambda_1\psi^{kj}=x\Omega^j\Lambda_1=u$ and so $\cC$ contains $x\L_1$; and it suffices to prove 
the Lemma under the assumption that 
that $\Lambda$ has no initial segment equal to $\Omega$.

Suppose that 
 $m=kp+r$, where $0\le r<k$. 
Then $x\Lambda\psi^{kp}= x\Omega^p\Lambda$ and $x\Omega^p\Lambda\psi^r=x\Lambda\psi^{kp+r}=x\Lambda\psi^m
=x\Lambda\G$. However, as $x$ is in a right semi-infinite $X$-component, $x\psi^r=z\Xi$, for some $z\in X$ and 
$\Xi\in A^*$. Thus $x\Lambda\G=x\Omega^p\Lambda\psi^r=x\psi^r\Omega^p\Lambda=z\Xi\Omega^p\Lambda$, 
which implies that $z=x$ and $\Lambda\G=\Xi\Omega^p\Lambda$. Now, as  $x\psi^r=x\Xi$, $0\le r<k$ and
$x$ has characteristic $(k,\Omega)$, it must be that $r=0$, $m=kp$ and $\Xi=\e$. 
We have now  $\Lambda\G=\Omega^p\Lambda$, and as $\Lambda$ has no initial 
segment equal to $\Omega$ it follows that $\Omega=\Lambda\Omega_1$. 
Now  
$u\psi^k=x\L\psi^k=x\psi^k\L=x\Om\L=x\L\Om_1\L=u\Om_1\L$, so $k\ge m$, by definition of characteristic. Therefore $k=m$ and $\G=\Om_1\L$, 
completing the proof
 in the case $m>0$. 

In the case when $m<0$ the result follows from the above on replacing $\psi$ by $\psi^{-1}$.
\end{proof}

An element $w$ of the free monoid $A^*$ is said to be \emph{periodic with period} $i$ 
if $w=a_1\cdots a_n$, where $a_j\in A$, and $a_{k}=a_{k+i}$, for $1\le k\le n-k$. In this sense, in Lemma \ref{AJDLEMMA2} above,  
$\L=(\G_1\G_0)^p\G_1$ is periodic of period $m$. 
\begin{lemma}\label{AJDLEMMACM}
Let $\psi \in G_{n,r}$ and $u\in V_{n,r}$ such that
$u\psi^k=u\D$, where $\D\neq \e$. Then 
 $u$ has characteristic $(m,\G)$ with respect to $\psi$, where 
 $k=mq$ and $\D=\G^q$, for some positive integer $q$. 
 \end{lemma}
\begin{proof}
Let $\psi$ be in semi-normal form with respect to $X$, and let $(m,\G)$ be the characteristic of $u$. 
Suppose first that $k>0$.
As in the proof of Lemma \ref{AJDLEMMA2}, we may assume that $u\in X\mA$, the $X$-component of $u$ is right semi-infinite 
and that there  exist $x\in X$ and $\G_1\in A^*$ such that $u=x\G_1$ and   $x$ has
characteristic power $m$. 
 Then $k\ge m$, say $k=mq+s$, where $0\le s<m$ and $q\ge 1$. Let $x\psi^s=y\L'$, where $y\in X$ and $\L'\in A^*$. 
Now   
 $x\G_1\D=u\D=u\psi^k=u\psi^{mq+s}=u\G^q\psi^s=x\psi^s\G_1\G^q=y\L'\G_1\G^q$.  Hence  $x=y$ and so $s=0$ and 
 $k=mq$. 
 Moreover $x\L\D=u\D=u\psi^k=u\psi^{mq}=u\G^q=x\L\G^q$, so $\L\D=\L\G^q$, 
from which $\D=\G^q$, as required.

If $k<0$, replace $\psi$ with $\psi^{-1}$ in the argument above. We have $u\psi^{-k}=u\D$, so from the
previous part of the proof, $u$ has characteristic $(m,\G)$, with respect to $\psi$, 
where $-k=mq$, $q>0$, 
and $\D=\G^q$. If follows that $u$ has characteristic $(-m,\G)$, with respect to 
$\psi$,  
and $-m=kq$, completing the proof. 
\end{proof}

\begin{corollary}\label{cor:allininf}
 Let $\psi$ be in semi-normal form with respect to an $A$-basis $X$ and let $u\in V_{n,r}$. Then there exists 
an element $\L\in A^*$ such that $u\L$ belongs to a complete $X$-component of $\psi$. 
\end{corollary}
\begin{proof}

Multiplying by a sufficiently long element of $A^*$ we may, as usual, assume that $u\in X\mA$, so 
$u$ belongs to either a complete or a semi-infinite $X$-component of $\psi$. 
There are finitely many semi-infinite  $X$-components (Lemma~\ref{9.1H}). If $S$ is a characteristic semi-infinite 
$X$-component with characteristic $(m,\G)$ then, from Lemma \ref{AJDLEMMA2}, 
elements of $S$ have the form $x\L$ where $x\in X$, $\L\in A^*$ and, for  all but finitely many elements  
of $S$, $\L$ is periodic
of period $m$.

Let $F_S$ be the finite subset of elements of $A^*$ such that $\L\in F_S$ only if 
$x\L\in S$ and $\L$ is not periodic of period $m$. Let $F_0$ be the union of the $F_S$ over all characteristic semi-infinite
$X$-components. 
 If $S$ is  non-characteristic then, from Lemma \ref{ABC}, $S$ contains an element $z\D$, where 
$z\in X$ of type (B), with characteristic $(m',\G')$, say. 
It follows,   from Lemma \ref{AJDLEMMA2} again, that all but finitely many elements
of $S$  have the form $x\L\D$ where $x\in X$, $\L\in A^*$ and $\L$ is periodic
of period $m'$. 
 This time, let $F_S$ be the finite subset of elements of $A^*$ such that $\L\D\in F_S$ only if 
$x\L\D\in S$ and $\L$ is not periodic of period $m'$. Let $F_1$ be the union of the $F_S$ over all non-characteristic semi-infinite
$X$-components. 

Let $M$ be the maximum of lengths of elements of $F_0\cup F_1$ 
and assume  $u=x\G$, where $x\in X$, $\G\in A^*$. 
Choose element $\Xi$ of $A^*$ such that $\G\Xi$ has length greater than $M$, 
is  not periodic and does not factor as $\L\D$, where $\L$ is periodic and $\D\in F_1$. 
Then $u\Xi=x\G\Xi$ cannot belong to  a semi-infinite $X$-component, so must belong to a complete $X$-component.
\end{proof}

\subsection{Quasi-normal forms}\label{sec:qnf}
Quasi-normal forms are particular semi-normal forms 
 which give representations of automorphisms minimising the number of elements in \pond* orbits. In \cite[Section 9]{Higg} it is claimed
that if an automorphism is given with respect to a quasi-normal form, then it has no \pond* orbits. In this section we shall see that this
is not the case.  
\begin{definition}[{\cite[Section 9]{Higg}}]
An element $\psi$ of $G_{n,r}$ is in \emph{quasi-normal} form with respect to the $A$-basis $X$ if it is in semi-normal form with respect to $X$, but not with respect to any proper contraction of $X$.
\end{definition}

It follows from Lemma \ref{4.1H} that for $\psi \in G_{n,r}$ there exists an $A$-basis $X$ with respect to 
which $\psi$ is in quasi-normal form. 
For instance, the automorphisms $\psi$ in Examples~\ref{ex:snf1}, \ref{orbiteg2} and \ref{ex:pond} are in quasi-normal form with respect to the bases $X$ in those examples. Additionally, the automorphism $\psi$ of Example \ref{ex:snf2} is in quasi-normal form with respect to the basis $X_3$.   

\begin{lemma}[cf. {\cite[Lemma 9.7]{Higg}}] \label{lem:qnf}
Given an element $\psi \in G_{n,r}$ there exists a unique $A$-basis, denoted $X_\psi$, 
with respect to which $\psi$ is in quasi-normal form. 
Furthermore $X_\psi$ may   be effectively constructed. 

\end{lemma}

\begin{proof} 
Assume $\psi$ is given by listing the images of elements of $X$, where $X$ is an $A$-basis of $V_{n,r}$. 
We  modify $X$ to find an $A$-basis $X'$ with respect to which $\psi$ is in semi-normal form. 
For each $y\in X$ we can list elements of the $\psi$-orbit of $y$.
\[\ldots ,y\psi^{-3}, y\psi^{-2}, y\psi^{-1}, y, y\psi, y\psi^{2}, y\psi^{3},\ldots \]
We enumerate the forward sequence $(y\psi^{m})_{m \geq 0}$, until we reach $m\ge 0$ such that
\begin{description}
\item[(1F)] either $y\psi^{m}\in X\mA$ with $y\psi^{m+1}\not\in X\mA$, or
\item[(2F)] for some $0\le l<m$, $\hat y \in X$ and $\Gamma, \Delta \in A^*$ we have 
$y\phi^l=\hat y\Gamma$ and  $y\phi^m=\hat y\Delta$.
\end{description}
Similarly, we enumerate the backwards sequence $(y\psi^{-k})_{k \geq 0}$ until we reach $k\ge0$ such that
\begin{description}
\item[(1B)] either $y\psi^{-k}\in X\mA$ with $y\psi^{-({k+1)}}\not\in X\mA$ or,
\item[(2B)] for some $0\le l<k$, $\hat y \in X$ and $\Gamma, \Delta \in A^*$ we have 
$y\phi^{-l}=\hat y\Gamma$ and  $y\phi^{-k}=\hat y\Delta$.
\end{description}
\item 
Given $y\in X$, the forward part of the process above produces a sequence of elements of $X\mA$, until it halts. As 
$X$ is finite, if it does not halt at step (1F) then it halts at step (2F); so always halts. 
Similarly, the backward part of the process always halts.  

If some $y$ satisfies \textbf{(1F)} and \textbf{(1B)}, then $y$ is in an incomplete $X$-component, so $\psi$ is not in semi-normal form with respect to $X$. In this case  
we take a simple expansion $X'$ of   $X$ at the element $y$.
Next, use the proof of Lemma \ref{9.9H} to find an expansion~$X''$ of~$X'$ with respect to which $\psi$ is 
in semi-normal form.
We now replace $X$ with $X''$ and return to the start of this proof.
Repeating as necessary, eventually we shall find $X$ such that
 no $y\in X$ satisfies both \textbf{(1F)} and \textbf{(1B)}.
The repetition terminates because the number of elements $x'' \in X''$ belonging to incomplete $X''$ components is strictly smaller than the corresponding number for~$X$.

At this stage, every $y \in X$ satisfies one of \textbf{(2F)} and \textbf{(2B)}, so $\psi$ is in semi-normal form with respect to $X$ by Lemma \ref{ABC}.
We can now test all the contractions of the $A$-basis $X$ to find an
expansion of $\cxx$ with respect to which $\psi$ is in a quasi-normal form. 

For uniqueness, we will argue by contradiction. Let $\psi$ be in quasi-normal form 
with respect to $X_1$ and $X_2$, with $X_1\ne X_2$. 
Since $X_1,X_2$ are expansions of $\cxx$, (without loss of generality) there exists a simple contraction $X_1'$ of 
$X_1$ which contains an element $y$ of $X_2\setminus X_1$. Then $X_1'\mA=X_1\mA\cup \{y\}\mA$
and, as 
$\psi$ is in semi-normal form with respect to $X_2$, it is also in semi-normal form with respect to $X_1'$, contrary to the
definition of quasi-normal form. 
\end{proof}

\begin{remark} \label{snfs_are_expansions}
Let $\psi \in G_{n,r}$ be in quasi-normal form with respect to $X$. The proof of this lemma illustrates that if $\psi$ is in semi-normal form with respect to $X'$, then $X'$ is an expansion of $X$. The converse is false: it is not true in general that $\psi$ is in semi-normal form with respect to all expansions of $X$.
\end{remark}

\begin{lemma}\label{lem:compcheck}
Let $\psi \in G_{n,r}$ be in semi-normal form with respect to an $A$-basis $X$ and let  
 $u,v \in X\mA$. Then we can effectively decide whether or not $u,v$ are in the same $X$-component, 
and if so, find the integers $m$ for which $u\psi^{m}=v$. 
\end{lemma}

\begin{proof}
As $u\in X\mA$, we have $u=y\Lambda$ where $y\in X$ and $\Lambda\in A^*$. 
We now run the process of Lemma \ref{lem:qnf} on $y$. If the process halts with $y\psi^m=y$, for some $m$ then we 
may list the elements $u\psi^i=y\psi^i\Lambda$, $i=0,\ldots ,m-1$, of the (complete finite) $\psi$-orbit of $u$. In this case  $v$ is in the
same $\psi$-orbit as $u$ if and only if it appears in the list,  so we are done.

Otherwise the process halts at least one of the states (2F) and (2B). 
We obtain $\widetilde y\in X$ and integers $k\neq l$ such that
$y\phi^k=\widetilde y \Lambda_1$ and $y\phi^l=\widetilde y \Lambda_2$, where~$\Lambda_1$ and~$\Lambda_2$ are distinct elements of~$A^*$. It follows from Lemma \ref{ABC}
that $\widetilde y$ is of type (B). As~$u$ and $u\phi^k=y\Lambda\phi^k=\widetilde y \Lambda_1\Lambda$ are~$k$ steps apart in the same $X$-component, we may replace 
$u=y\Lambda$ with $\widetilde u = \widetilde y \Lambda_1\Lambda$. Therefore we now assume that $u=y\Lambda$, where 
$y$ is of type (B).

Now, when we run the process of Lemma \ref{lem:qnf} on $y$ it halts either at (2F) and (1B) or else at (1F) and (2B). 
Suppose first the forward part halts at  (2F). Then $y$ is in a right semi-infinite $X$-component and 
there is a minimal positive integer $m$ such that $y\psi^m=y\Gamma$, with $\Gamma\neq 1$.  
That is $y$ has characteristic $(m,\G)$, with $m>0$. 
Set $u_{0}=y\Lambda_{0}$.
If $\Lambda=\Gamma^{i}\Lambda_{0}$ where $\Lambda_{0}$ 
has no initial segment $\Gamma$, then
\[u_{0}\psi^{m i}=y\Lambda_{0}\psi^{m i}=y\psi^{m i}\Lambda_{0}=y\Gamma^{i}\Lambda_{0}=y\Lambda=u,\]
 so $u_{0}$ is $mi$ steps away from $u$ in the $\psi$-orbit of~$u$. We may 
  replace $u=y\Lambda$ by $u_{0}=y\Lambda_0$.
  This allows us to assume from now on that $\Lambda$ has no initial 
segment equal to the 
characteristic multiplier $\Gamma$ of~$y$. 

Next we run the process of  Lemma \ref{lem:qnf} on $u$ instead of $y$. As $y$ is in a right semi-infinite $X$-component the forward part of the
process halts at (2F). 
We obtain a list of elements of the $X$-component of $u$ of the form 
\begin{equation}\label{eq:uorb}
z_r\Phi_r,\,\cdots \,,z_1\Phi_1,\,u=y\Lambda,\, y_1\Gamma'_1\Lambda,\,\ldots\, ,y_{m-1}\Gamma'_{m-1}\Lambda,\, y\Gamma\Lambda,
\end{equation} 
where $y_j, z_j\in X$, $\G_j', \Phi_j\in A^*$, $z_j\Phi_j=u\psi^{-j}$, for $1\le j\le r$ and for some $r\ge 0$, and
$y\psi^s=y_s\G_s'$, for $0<s<m$. (The $y_i$'s must be distinct otherwise $u$ would have characteristic power less than $m$.)
We proceed differently based on which state the backwards enumeration finishes in.

\textbf{Case (1B).}
If  the backward part  of the process halts at (1B) then $z_r\Phi_r\psi^{-1}=u\psi^{-r-1}\notin X\mA$. In this case, 
the entire $X$-component of $u$ consists of the elements on this list together with elements 
\[y_i\Gamma'_i\G^q\Lambda,\quad \text{with $q>0$ and $0< i\le m$,}\]  
where we set $y_0=y$, $\G_0'=\G$. 

As $v\in X\mA$ we also have $z\in X$ and $\Delta$ in $A^*$ such that 
$v=z\Delta$. If $z$ is in a finite $X$-component then $v$ cannot belong to the same $X$-component as $u$, so we assume
$z$ is in an infinite $X$-component. As in the case of $u$, we may adjust $v$ so that $z$ is of type (B). 
As before we find a characteristic multiplier $\Phi$ for $z$ and,  
replacing $\Delta$ with a shorter element if necessary, we may assume that $\Delta$ has no initial segment 
equal to $\Phi$.   

If $v=u\psi^d$, where $d\ge 0$, then $v=y_i\Gamma'_i\G^q\Lambda$, for some $q\ge 0$ and $i$ with $0\le i< m$.
In this case, $z=y_i$ and by Lemma \ref{AJDLEMMA2} and our assumption
on $v$ we have $q=0$,  so $v=y_i\Gamma'_i\Lambda$, which appears on  list \eqref{eq:uorb}. Assume then that 
$v=u\psi^d$, where $d< 0$.  
As the backward part  of the enumeration of the $\psi$-orbit of $u$ halts at (1B), 
the $X$-component of $u$ has initial element $z_r\Phi_r$, and $v$ must appear on   list \eqref{eq:uorb}. 

\textbf{Case (2B).} On the other hand, 
if the backward part of the process stops at (2B) then $u$ is in a complete infinite $X$-component and, for some $s$ with $0\le s\le r$, we 
have $z_r=z_s$ (and $r$ is minimal with this property).  It follows that $z_s$ is of type (B) and in a left semi-infinite $X$-component. 
Again, we may assume that $v=z\Delta$, where $\Delta \in A^*$,  $z\in X$ is of type (B) and has characteristic multiplier $\Phi$, such that
$\Delta$ has no initial segment equal to $\Phi$. As before, if $v=u\psi^d$ with $d\ge 0$, then $v$ appears on  list \eqref{eq:uorb}. 
Assume then that 
$v=u\psi^d$, where $d< 0$. 
Repeating the argument above, using the left semi-infinite $X$-component of $z_s$ instead of the 
right semi-infinite $X$-component of $y$, it follows again that $v$ appears on  list \eqref{eq:uorb}.

Therefore, in the case where $y$ is in a right semi-infinite $X$-component we have $v$ in the $X$-component of $u$ if and only if 
$v$ lies on the list \eqref{eq:uorb}; and we may compute $m$ such that $u\psi^m=v$, if this is the case. 
Finally, if the enumeration of the  $X$-component of $y$  halts at steps 
(1F) and (2B) then the process is essentially the same, except that we deal with a left, rather than a right, semi-infinite
$X$-component of $y$. 
\end{proof}

This procedure allows us to decide if two given words in $X\mA$ belong to the same $X$-component so, 
if there are no pond orbits, we may decide if two such words belong to the same $\psi$-orbit.
On the other hand, as  the enumeration of components always stops once we fall outside $X\mA$, 
we  cannot detect when a pair of elements lie in the same $\psi$-orbit but on opposite sides of a pond.
We demonstrate below 
that there exist automorphisms for which every semi-normal form has a \pond*; thus we require a strategy to deal with \pond*s.

\begin{lemma} \label{cannot_expand_to_hide_pond}
Let $\psi \in G_{n,r}$ be in semi-normal form with respect to $X$, and suppose that some $\psi$-orbit $\mathcal O$ contains a \pond* with respect to $X$. If $\psi$ is in semi-normal form with respect to an expansion $X'$ of $X$, then $\mathcal O$ is also a pond-orbit with respect to $X'$.
\end{lemma}

\begin{proof}
Let us write $\mathcal O$ as
\[ \mathcal O\colon \qquad 
	\dotsc \: l\psi^{-t}, \: \dotsc \: l\psi^{-1}, \: l, \: p_1, \: \dotsc \: p_k, \: r, \: r\psi, \: \dotsc \: r\psi^s, \: \dotsc \]
where $l, r \in X\mA$ are endpoints of semi-infinite $X$-components and the $p_i \notin X\mA$ form a pond of length $k$.
To begin with we claim that, for sufficiently large $s\ge 0$, we have $r\psi^s\in X'\mA$. Indeed, 
because $r$ belongs to a semi-infinite $X$-component, Lemma \ref{ABC} implies that 
there is some $s' \geq 0$ for which $r\psi^{s'} = r'\D$, where $\D \in A^*$ and $r' \in X$ has characteristic $(m, \G)$. 
Therefore, for all $q\ge 0$,
\[
	r\psi^{s'+mq} = r\psi^{s'}\psi^{mq} = r'\D\psi^{mq} = r'\psi^{mq}\D = r'\G^q\D.
\]
By taking $q$ sufficiently large, we can ensure that $r\psi^{s'+mq} \in X'\mA$.
This works because the difference $X\mA \setminus X'\mA$ is finite.
So we can find $s \geq 0$ such that $r\psi^s \in X'\mA$.
Similarly, there is some $t \geq 0$ for which $l\psi^{-t} \in X'\mA$.

Since $X'\mA \subset X\mA$, it follows that each $p_i \notin X'\mA$.
Appealing to Corollary~\ref{cor:orbit-type}, the only possibility is that $\mathcal O$ is a \pond*-orbit with respect to $X'$.
\end{proof}
Notice that the \pond* width with respect to $X'$ is at least the previous width~$(k + 1)$ with respect to~$X$. Additionally, if $\psi$ was in \emph{quasi}-normal form with respect to $X$, this (with Remark~\ref{snfs_are_expansions}) shows that every semi-normal form $X'$ for $\psi$ contains the \pond* given above.
Example~\ref{ex:pond} shows that this possibility does occur.

\begin{lemma} \label{lem:can_find_ponds}
Given an element $\psi \in G_{n,r}$ in semi-normal form with respect to an $A$-basis $X$ we may 
effectively construct the set $P(\psi)$ of the triples $(l, k, r)$ such that $r$ (resp.~$l$) is the initial (resp.\ terminal) word in a right (resp.\ left) semi-infinite $X$-component, and $k$ is the width of the pond between them.   
\end{lemma}
\begin{proof}
Let $Y$ be the minimal expansion of~$X$ associated to~$\psi$ and let $Z=Y\psi$.
Since there are no incomplete $X$-components, Lemma~\ref{9.1H} tells us that the set of initial elements of 
right semi-infinite $X$-components is $R=X\mA\bh Z\mA$.
This is finite, so we may enumerate this effectively.
The same is true for the set $L=X\mA\bh Z\mA$ of terminal elements of left semi-infinite $X$-components. 
To enumerate $P(\psi)$, for each $(l,r)\in L\times R$ 
we need to solve the equation $r=l\psi^k$ for some $k$, or to determine that there is no solution.
A solution exists if and only if $r\G=l\G\psi^k$, for all $\G\in A^*$.

With this in mind, first find $\G\in A^*$ such that 
$l\G$ is in a complete infinite $X$-component.
We do this by enumerating the words $\G$ of length $1, 2, \dotsc$ and applying the process of Lemma \ref{lem:qnf} to each element $l\Gamma$ in turn.
We stop when we find $\G$ such that the process halts at (2F) and (2B).
Now use Lemma \ref{lem:compcheck} to determine whether $r\G$ and $l\G$ are in the same $X$-component.
If not, then there cannot exist an element of the form $(l, k, r) \in P(\psi)$; that is~$l$ and~$r$ are not joined by a \pond*. 
 
Assume then that $r\G=l\G\psi^k$, for some $k$.
We can test to see if $r=l\psi^k$ directly, which holds if and only if $(l, k, r) \in P(\psi)$.
If the equality were false, is it possible that $(l, k', r) \in P(\psi)$ for a different $k' \neq k$?
This would mean that $r=l\psi^{k'}$, so 
$l\psi^k\G=r\G=l\G\psi^{k'}$ and thus $l\G\psi^{k-k'}=l\G$.
As $l\G$ belongs to a complete infinite $X$-component, this 
means $k=k'$; so the answer to our previous question is `no'.
In this situation there are no elements of the form $(l, k', r)$ in~$P(\psi)$.
\end{proof}
In practice, when enumerating the sets $L$ and $R$ in the proof above, we need consider only non-characteristic elements, as Lemma 
\ref{AJDLEMMA2} implies that no characteristic element belongs to a pond orbit.

\begin{example} \label{ex:pond_no_other_ponds}
Let $\psi$ and $X$ be the automorphism and basis described in Example~\ref{ex:pond}; we noted above that $\psi$ is in quasi-normal form with respect to $X$.
 We claim that this $\psi$-orbit is the only \pond* orbit with respect to $X$.

The endpoints of semi-infinite $X$-components are precisely
\begin{equation*}
	L = X\mA \setminus Y\mA = \{ x\a_1^2, x\a_1^3, x\a_1^2\a_2 \}    
	\quad\text{and}\quad
	R = X\mA \setminus Z\mA = \{ x\a_1\a_2, x\a_1\a_2\a_1, x\a_1\a_2^2 \}.
\end{equation*}
The four endpoints $x\a_1^2$, $x\a_1^3$, $x\a_1\a_2$ and $x\a_1\a_2\a_1$ have characteristics $(-1, \a_1^2)$, $(-1, \a_1^2)$, $(1, \a_1\a_2)$ and $(1, \a_2\a_1)$ respectively.
Are the two remaining endpoints $l = x\a_1^2\a_2$ and $r = x\a_1\a_2^2$ separated by a pond? (We saw before in computation \eqref{eqn:pond_example} that they are, but to illustrate Lemma~\ref{lem:can_find_ponds} we'll remain ignorant of this.)

Multiplying by~$\Gamma=\alpha_1$ we obtain $l\Gamma = x\a_1^2\a_2\a_1$, which is in a complete infinite $X$-component. 
We also see that $r\Gamma = l\psi^2\Gamma = x\a_1\a_2^2\a_1$ is in this component, so we have a candidate \pond* width of $k = 2$.
Fortunately we directly computed that $r = l\psi^2$ in \eqref{eqn:pond_example}, so $P(\psi) = \{(x \a_1^2 a_2, 2, x\a_1\a_2^2)\}.$
\end{example}

\begin{lemma}[cf. {\cite[Lemma 9.7]{Higg}}]\label{9.7H}
Let $\psi \in G_{n,r}$ and 
 $u,v \in V_{n,r}$. Then we can effectively decide whether or not $u,v$ are in the same $\psi$-orbit, 
and if so, find the integers $m$ for which $u\psi^{m}=v$. 
\end{lemma}
\begin{proof}
For a fixed integer $s\ge 0$ we have 
$u\psi^{m}=v$ if and only if $(u\Gamma)\psi^{m}=u\psi^m\Gamma=v\Gamma$
for all $\Gamma\in A^*$ of length $s$ (using Lemma \ref{AJDLEMMAX}). Now, suppose 
that we have an algorithm $\mathcal{A}$ 
 to decide whether $v'=u'\psi^m$ for some $m$, given elements $u',v'$ of $X\mA$ (and to return $m$, if so).
Then if $u,v$ are 
arbitrary elements of $V_{n,r}$ 
  we may
choose $s$ such that $u\Gamma$ and $v\Gamma$ belong to $X\mA$, for all $\Gamma \in A^*$ of length $s$, and  input all
these elements to 
the algorithm $\mathcal{A}$ in turn. In  the light of the previous remark, this allows us to determine 
whether or not $u$ and 
$v$ belong to the same $\psi$-orbit (and to return  appropriate $m$, if so). Hence  we may assume $u,v\in X\mA$.

By Corollary \ref{cor:orbit-type},  $u$ and $v$ belong to the same $\psi$-orbit if and only if 
either they belong to the same $X$-component of a $\psi$-orbit, or they belong to different $X$-components of a single 
\pond* orbit.  We may use Lemma \ref{lem:compcheck} to decide whether or not $u$ and $v$ both belong to the same $X$-component.
If so we are finished. If not, and both belong to semi-infinite $X$-components, then for each triple $(l, k, r)$ in $P(\psi)$ we check whether~$u$ belongs to the same component as~$l$ or~$r$.

If~$u$ belongs to neither component then $u$ is not in a \pond* orbit, and thus~$u$ and~$v$ do not share a $\psi$-orbit.
Otherwise if~$u$ and~$l$ (resp.~$r$) share an $X$-component, we run the same check on~$v$ and~$r$ (resp.~$l$).
If the check determines that~$v$ is not in the $X$-component in question, then~$u$ and~$v$ do not share a $\psi$-orbit.
Else we have $l = u\psi^a$ and $v = r\psi^b$ for some~$a$ and~$b$, so $v = u\psi^{a + k + b}$.
\end{proof}

\begin{example} \label{ex:ost}
Let $\psi$ be the automorphism of Examples~\ref{ex:pond} and~\ref{ex:pond_no_other_ponds}, which is in quasi-normal form with respect to $X = \{ q_1 = x\a_1^2, q_2 = x\a_1\a_2, q_3 = x\a_2\a_1, q_4 = x\a_2^2 \}$.
The elements $q_1$ and $q_2$ have characteristics $(-1, \a_1^2)$ and $(1, \a_1\a_2)$ respectively, whereas $q_3$ and $q_4$ belong to complete infinite $X$-components such that $q_3\psi = q_2\a_2^2$ and $q_4\psi^{-1} = q_1\a_2\a_1$.

\begin{enumerate}
\item 
We wish to test if $u = x\a_1\a_2^2\a_1^2\a_2 = q_2\a_2\a_1^2\a_2$ and $v = x\a_2\a_1^2 = q_3\a_1$ belong to the same $\psi$-orbit.
Because $q_3$ is not characteristic, Lemma~\ref{lem:compcheck} first replaces $v=q_3\a_1$ with $v' = v\psi = q_3\psi\a_1 = q_2\a_2^2\a_1$, which begins with the
characteristic element $q_2$ of $X$. 
Enumerating the $X$-component containing $u$ gives us a specific instance of list~(\ref{eq:uorb})
\[\label{eq:uorb_example}
	x\a_1^4\a_2\a_1^2\a_2     \mapsto x\a_1^2\a_2\a_1^2\a_2       \mapsto x\a_2^2\a_1\a_2 \mapsto
     \underbrace{x\a_1\a_2^2\a_1^2\a_2}_u \mapsto x(\a_1\a_2)^2\a_2\a_1^2\a_2  \tag{\ref{eq:uorb}'}
\]
once the enumeration has halted at stages (2F) and~(2B). Since $v'$ does not lie on this list, we conclude that $v'$ does not belong to the $X$-component of $u$, 
so neither does $v$.

We now need to check if $u$ and $v$ are separated by a \pond*.
In Example~\ref{ex:pond_no_other_ponds} we showed that $\psi$ has only one \pond*-orbit, and referring to the computation~\eqref{eqn:pond_example} we see that neither $u$ nor $v$ belong to this orbit.
Hence $u$ and $v'$ do not share a $\psi$-orbit.

\item
Now let us test if $u$ and $w = x\a_1^4\a_2\a_1^2\a_2 = q_1\a_1^2\a_2\a_1^2\a_2$ share a $\psi$-orbit.
We remove the characteristic multiplier $\a_1^2$ of $q_1$ from $w$, obtaining $w' = q_1\a_2\a_1^2\a_2$ where $w'\psi^{-1} = w$.
From list~(\ref{eq:uorb_example}) we notice that $u\psi^{-2} = w'$, so $u\psi^{-3} = w$.

\item
Let $u=x\a_1^8\a_2$, $v = x\a_1^4\a_2\a_1$ and $w=x(\a_1\a_2)^3\a_2$.
In terms of $X$, these are $u = q_1\a_1^6\a_2$, $v = q_1\a_1^2\a_2\a_1$, and $w = q_2(\a_1\a_2)^2\a_2$.
Since $q_1$ and $q_2$ are characteristic, we remove copies of the characteristic multipliers.
We obtain $u' = q_1\a_2 = u\psi^3$, $v' = q_1\a_2\a_1 = v'\psi$ and $w' = q_2\a_2 = w\psi^{-2}$.
Enumerating the $X$-component of $u'$ gives us
\[
	\dots \mapsto x\a_1^4\a_2 \mapsto x\a_1^2\a_2 = u',
\]
(halting at stages (1F) and~(2B)) and we see that neither $v'$ nor $w'$ are in this list.
However, $u'$ is adjacent to a \pond*.
Referring once more to Example~\ref{ex:pond}, we see that the corresponding endpoint is $\bar u = u'\psi^2 = x\a_1\a_2^2$.
Its $X$-component begins 
\[
	\bar u = x\a_1\a_2^2 \mapsto x(\a_1\a_2)^2\a_2 \mapsto \dots
\]
Since this list does not contain $v'$, we conclude that $u$ and $v$ do not share a $\psi$-orbit.
On the other hand, we note that $w' = \bar u$ belongs to the list.
Hence $u$ and $w$ belong to the same $\psi$-orbit, and having kept track of the various powers, we calculate that
\[
	w\psi^{-2} = w' = \bar u = u'\psi^2 = u\psi^3\psi^2 \implies w = u\psi^7 .	
\]
\end{enumerate}
\end{example}


\section{The Conjugacy problem}\label{VPVRI}
For a group with presentation $G = \langle X \mid R\rangle$, the conjugacy problem is to determine, 
 given words $g, h\in \FF(X)$ whether or not $g$ is conjugate to $h$ in $G$; denoted $g\sim h$. 
The strong form, which we consider here, requires us to produce a conjugator $c\in \FF(X)$ when $g$ is conjugate to $h$, i.e.\ an element~$c$ such that
$c^{-1}gc=_Gh$.  We say the conjugacy problem is \emph{decidable} if there is an algorithm which for inputs
$g$ and $h$ outputs ``yes'' if they're conjugate and ``no'' otherwise. The stronger form is decidable if there is an algorithm which produces a conjugator~$c$ in the ``yes'' case.
Note that the word problem is the special case of the conjugacy problem where $h=1$.

As pointed out at the beginning of Section \ref{HigmanThompsonG21}, an element $\psi$ of $G_{n,r}$ may be uniquely represented by the 
triple $(Y,Z,\psi_0)$, where $Y$ is the minimal expansion of $\psi$,  $Z=Y\psi$ and $\psi_0$ is a bijection between  $Y$ and $Z$, 
namely $\psi_0=\psi|_Y$.  
This triple is called a  \emph{symbol} for $\psi$.  
In \cite[Section~4]{Higg} a finite presentation of $G_{n,r}$ is given, with generators the symbols $(Y,Z,\psi_0)$ such that 
$Y$ is a $d$-fold expansion of $\cxx$, for $d\le 3$. As we may effectively enumerate symbols and 
effectively construct the symbol for $\psi_1\psi_2$, from the 
symbols for $\psi_1$ and $\psi_2$, words in Higman's generators effectively determine symbols and vice-versa. 
 Therefore when we consider algorithmic problems in $G_{n,r}$ we may work with symbols for automorphisms, 
and leave the presentation in the background. That is, we always assume that automorphisms are given as
maps between bases of $V_{n,r}$ (from which a symbol may be computed). 
As minimal expansions are unique it follows immediately that the word
problem is solvable in $G_{n,r}$. In this section we give an algorithm for the conjugacy problem
in $G_{n,r}$, based on (a  complete version of) Higman's solution. 

\subsection{Higman's \texorpdfstring{$\psi$}{psi}-invariant subalgebras}\label{sec:pri}

Let $\psi$ be an element of $G_{n,r}$. Higman defined two $\Omega$-subalgebras of $V_{n,r}$ determined by $\psi$,
namely
\begin{itemize}
\item
the $\Omega$-subalgebra  $V_{P,\psi}$ generated by the set of 
elements of $V_{n,r}$ which belong to finite $\psi$-orbits.  
\item 
the $\Omega$-subalgebra  $V_{RI,\psi}$ generated by the set of  characteristic elements for $\psi$.
\end{itemize} 
Where there is no ambiguity, we will write $V_{P}$ for $V_{P,\psi}$ and $V_{RI}$ for $V_{RI,\psi}$.

If $u\in V_{n,r}$ then the $\psi$-orbit of $u$ is identical to the $\psi$-orbit of $u\psi$; so $u$ is in a finite
$\psi$-orbit if and only if $u\psi$ is in a finite $\psi$-orbit. 
 From Lemma \ref{lem:char}, an element $u$ is a  characteristic element for $\psi$ if and only if
$u\psi$ is a  characteristic element for $\psi$. Therefore $V_{P,\psi}$ and $V_{RI,\psi}$ are $\psi$-invariant 
subalgebras of $V_{n,r}$. (A subalgebra $S$ is $\psi$\emph{-invariant} if $S\psi=S$.) Hence  $\psi_{P}=\psi|_{V_{P,\psi}}$
is an automorphism of $V_{P,\psi}$ and $\psi_{RI}=\psi|_{V_{RI,\psi}}$ is an automorphism of $V_{RI,\psi}$.

If $\psi$ and $\phi$ are conjugate elements of $G_{n,r}$ and $\rho^{-1}\psi\rho=\phi$ for some conjugator $\rho\in G_{n,r}$, then 
for all $\G\in A^*$ we have $u\phi^m=u\G$ if and only if $u \rho^{-1}\psi^m\rho=u\G$ if and only
if $(u \rho^{-1})\psi^m=(u\rho^{-1})\G$. Thus $u$ is in a finite $\phi$-orbit if and only if $u\rho^{-1}$ is 
in a finite $\psi$-orbit (taking $\G=\e$) and $u$ is a  characteristic element for $\phi$ if and 
only if $u\rho^{-1}$ is a  characteristic element for $\psi$ ($\G\neq \e$). It follows that  
the restriction  $\rho|_{V_{P,\psi}}$ of $\rho$ to $V_{P,\psi}$ maps $V_{P,\psi}$ isomorphically to $V_{P,\phi}$, and similarly
  $\rho|_{V_{RI,\psi}}$ is an isomorphism from $V_{RI,\psi}$ to $V_{RI,\phi}$.

Now suppose that $\psi$ is in semi-normal form with respect to an $A$-basis $X$.  Partition $X$ into 
\[X_P=X_{P,\psi}=\{y\in X \mid\text{$y$ is of type (A)} \}\]
and 
\[X_{RI}=X_{RI,\psi}=\{y\in X \mid \text{$y$ is of type (B) or (C)}  \}.\]
\begin{theorem}[{\cite[Theorem 9.5]{Higg}}] \label{decomp}
Let $\psi$ be an element of $G_{n,r}$, in semi-normal form with respect to $A$-basis $X$. Then, with the notation above,
the following statements hold. 
\begin{enumerate}
\item\label{it:decomp1} $V_{n,r}=V_{P}\ast V_{RI}$, the free product of the $\psi$-invariant subalgebras $V_{P}$ and $V_{RI}$.
\item\label{it:decomp2} $V_P=X_P\mA \langle \lambda\rangle$ and $V_{RI}=X_{RI}\mA \langle \lambda\rangle$; that is, $V_P$ ($V_{RI}$) is generated by~$X_P$ ($X_{RI}$).
\item\label{it:decomp3} Given $\psi, \varphi, \rho \in G_{n,r}$ define six restrictions as follows.
\begin{align*}
\psi_P   &=\psi|_{V_{P,\psi}}  & \varphi_P   &=\varphi|_{V_{P,\varphi}}  & \rho_P   &=\rho|_{V_{P,\psi}} \\
\psi_{RI}&=\psi|_{V_{RI,\psi}} & \varphi_{RI}&=\varphi|_{V_{RI,\varphi}} & \rho_{RI}&=\rho|_{V_{RI,\psi}}
\end{align*}
We have $\rho^{-1}\psi\rho=\varphi$ if and only if $\rho_{P}^{-1}\,\psi_{P}\,\rho_{P}=\varphi_{P}$ and $\rho_{RI}^{-1}\,\psi_{RI}\,\rho_{RI}=\varphi_{RI}$.
\end{enumerate}
\end{theorem}
\begin{proof}
Write $W_P=X_P\mA\ml$ and $W_{RI}=X_{RI}\mA\ml$. As $X$ is the disjoint union
of $X_P$ and $X_{RI}$, we have  $V_{n,r}=W_P\ast W_{RI}$, using Lemma \ref{2.2H}. We shall show that 
$V_P=W_P$ and $V_{RI}=W_{RI}$. By definition, $W_P\subseteq V_P$. If $x\in X_{RI}$ is of type (B) then $x\in V_{RI}$, by definition.
If $x\in X_{RI}$ is of type (C) then there exists $z\in X_{RI}$, of type (B), and $\D\in A^*$,  such that $x\psi^i=z\D$. As $z\in V_{RI}$, 
so is $z\D$, and as $V_{RI}$ is $\psi$-invariant we have $x=z\D\psi^{-i}\in V_{RI}$. Hence $W_{RI}\subseteq V_{RI}$. 

To see that $V_P\subseteq W_P$, let $u\in V_{n,r}$ have a finite $\psi$-orbit. Choose $d\in \NN$ such that, $u\G\in X\mA$, 
for all $\G\in  A^*$ of length $d$. For each such $\G$ write $u\G=x\D$, where $x\in X$ and $\D\in A^*$.
As $u$ is in a finite $\psi$-orbit so is $u\G$, so $x\in X_P$ and thence $u\G=x\D \in W_P$. As this holds for all $\G$ in $A^*$ of 
length $d$, we have $u\in W_P$, by Lemma \ref{2.2H}. Hence $V_P\subseteq W_P$. 

To see that $V_{RI}\subseteq W_{RI}$, we first show that $W_{RI}$ is $\psi$-invariant. 
Let $Y$ be the  minimal expansion of $X$ associated to $\psi$ and let $x\in X_{RI}$. 
Then choose $d$ such that $x\G \in Y\mA$, for all $\G\in A^*$ of length $d$. 
Given such a $\G$, write $x\G=y\D$ for $y\in Y$ and $\D\in A^*$. 
Then  $x\G\psi=y\psi\D\in X\mA$, so  
$x\G\psi =z\L$, for some $z\in X$ and $\L\in A^*$. 
Moreover, $z$ must have type (B) or (C), as $x$ does, so $x\G\psi \in X_{RI}\mA\subseteq W_{RI}$. This holds for all
$\G$ of length $d$, so again $x\psi \in  W_{RI}$. It follows that $W_{RI}\psi \subseteq W_{RI}$.

Repeating the same argument, using 
$Z=Y\psi$ instead of $Y$ and $\psi^{-1}$ instead of $\psi$ gives $W_{RI}\psi^{-1} \subseteq W_{RI}$; so $W_{RI}$ is $\psi$-invariant as claimed. 
Now 
 let $u\in V_{n,r}$ be a characteristic element for $\psi$. Then, from Lemma \ref{AJDLEMMA2}, we have $u\psi^i=x\L$, for 
some integer $i$, $x\in X_{RI}$ and $\L\in A^*$. Thus $u=x\L\psi^{-i}\in W_{RI}$, as $W_{RI}$ is $\psi$-invariant; and we have $V_{RI}\subseteq W_{RI}$.
  This proves \ref{it:decomp1} and \ref{it:decomp2} of the Theorem, and \ref{it:decomp3} then follows from the discussion preceding the
statement of the Theorem.
\end{proof}
Note that in the case that $\rho^{-1}\psi\rho=\varphi$ in the theorem above we have $\rho=\rho_{P}\ast\rho_{RI}$ an isomorphism 
from $V_{P,\psi}\ast V_{RI,\psi}$ to $ V_{P,\varphi}\ast V_{RI,\varphi}$, both of which are isomorphic to $V_{n,r}$. 

\begin{example}
Let $\psi$ be as in Example \ref{orbiteg}. Then 
$X_P=\{x\a_2\a_1,x\a_2^2\}$ and $X_{RI}= \{x\a_1^2,x\a_1\a_2\}$.  
Thus 
$\psi_P$ is the automorphism of $V_P=X_P\mA\lp \lambda\rp$ defined by 
\[x\alpha_2\alpha_1\mapsto x\alpha_2^2, \qquad x\alpha_2^2\mapsto x\alpha_2\alpha_1. \]
Let $Y_{RI}=\{x\alpha_1^3,x\alpha_1^2\alpha_2,x\a_1\a_2\}$ and $Z_{RI}=\{x\alpha_1^2,x\alpha_1\alpha_2\alpha_1,x\alpha_1\alpha_2^2\}$, both
of which are expansions of $X_{RI}$. 
Then 
 $\psi_{RI}$ is the automorphism of $V_{RI}=X_{RI}\mA\lp \lambda\rp$ defined by 
\[x\alpha_1^3\mapsto x\alpha_1^2, \qquad  x\alpha_1^2\alpha_2\mapsto x\alpha_1\alpha_2\alpha_1, \qquad x\alpha_1\alpha_2\mapsto x\alpha_1\alpha_2^2.\]
\end{example}

Theorem \ref{decomp} allows us to decompose the conjugacy problem for $(\psi, \phi)$ into 
conjugacy problems for $(\psi_{P}, \phi_P)$ and $(\psi_{RI}, \phi_{RI}$).
Indeed, $V_{P}\cong V_{n,|X_P|}$ and $V_{RI}\cong V_{n,|X_{RI}|}$, 
and we  regard $\psi_{P}$  and $\psi_{RI}$ as automorphisms of $V_{n,|X_P|}$ 
and $V_{n,|X_{RI}|}$, respectively. It turns out that $\psi_{P}$  and $\psi_{RI}$
are each of particularly simple types; so if we 
can solve the conjugacy problem for these simple types of automorphism, 
then we can solve it in general. In the remainder
 of this subsection we describe in detail how this decomposition works.  

First consider a single automorphism $\psi\in G_{n,r}$, where $\psi$ is
in semi-normal form with respect to an $A$-basis $X$. 
As before, we take $V_{n,r}$ to be the free $\cV_n$ algebra on a set 
$\cxx$ of size $r$, so that $X$ is an expansion of $\cxx$. Let
$X_P$ and $X_{RI}$ be defined as above, let $Y$
be the minimal expansion of $X$ associated to $\psi$ and let $Z=Y\psi$.
As $Y$ is an expansion of $X$, for all $x\in X$ the set
$Y_x=Y\cap \{x\}\mA$ is an expansion of $\{x\}$, by Lemma \ref{HigmanLemma2.5}.
Therefore $Y_{P}=Y\cap X_{P}\mA$  is an expansion of $X_{P}$, and 
 $Y_{RI}=Y\cap X_{RI}\mA$ is an expansion of $X_{RI}$. Similarly,
$Z_{P}=Z\cap X_{P}\mA$ and 
 $Z_{RI}=Z\cap X_{RI}\mA$ are expansions of $X_{P}$ and $X_{RI}$, respectively.
In fact, as $\psi$ permutes the elements of $X$ with type~(A), $\psi_P$ permutes
the elements of $X_P$, so $X_P=Y_P=Z_P$. 
Therefore $\psi_P$ is an automorphism of $V_P=X_P\mA\ml$, which permutes
the elements of $X_P$.

For all $y\in Y_{RI}$ we have $y\psi=z\in Z$; moreover $z \in X_{RI}\mA$ because $V_{RI}$ is $\psi$-invariant, so 
$Y_{RI}\psi=Z_{RI}$. Now $\psi_{RI}$ is an
automorphism of $V_{RI}$, where $V_{RI}$ is freely generated by $X_{RI}$, 
and $Y_{RI}$ is the minimal expansion of $X_{RI}$ associated to 
$\psi_{RI}$ (as $Y$ is the minimal expansion of $X$ associated to $\psi$).
Furthermore $Y_{RI}\psi_{RI}=Z_{RI}$ and 
if  $u$ is an element of $X_{RI}\mA$ such that $u\psi\in \XA$ then
$u\psi\in \XA\cap V_{RI}=X_{RI}\mA$; so 
no element of $X_{RI}\mA$ is in an incomplete finite $X_{RI}$-component of $\psi_{RI}$. 

To summarise, let $|X_P|=a$, $|X_{RI}|=b$ and let 
$X_P=\{x_1,\ldots, x_a\}$ and $X_{RI}=\{x_{a+1},\ldots, x_{a+b}\}$, where
$x_i\in \cxx\mA$. 
 Then, regarding the $x_i$ as new generators, we may
view $V_P$ as $V_{n,a}$, the free $\cV_n$ algebra on $\{x_1,\ldots, x_a\}$,
and $V_{RI}$ as $V_{n,b}$, the free $\cV_n$ algebra on $\{x_{a+1},\ldots, x_{a+b}\}$.
We regard $\psi_P$ and 
$\psi_{RI}$ as elements of $G_{n,a}$ and $G_{n,b}$, respectively.
 In this case,  
 $\psi_P$ (resp.~$\psi_{RI}$) is in quasi-normal form with respect to the 
$A$-basis $X_P$ (resp.~$\psi_{RI}$)
We write all elements of $Y$ and $Z$ in terms of the $x_i$, rather
than as expansions of elements of $\cxx$.)   

\begin{example}\label{ex:sub2full}
Let $n=2$, $r=1$ and $V_{2,1}$ be free on $\cxx = \{x\}$. Let
 \[Y=\{x\a_1^4,x\a_1^3\a_2,x\a_1^2\a_2,x\a_1\a_2\a_1,x\a_1\a_2^2,x\a_2\a_1,x\a_2^2\a_1,x\a_2^3\}\] 
and 
\[Z=\{x\a_1^3,x\a_1^2\a_2\a_1,x\a_1^2\a_2^2,x\a_1\a_2\a_1,x\a_1\a_2^2,x\a_2\a_1^2,x\a_2\a_1\a_2,x\a_2^2\}\]
and let $\psi$ be the element of $G_{n,r}$ determined by the bijection illustrated below.
\begin{center}
$\psi:$
\begin{minipage}[b]{0.8\linewidth}
\centering
\Tree [ [ [ [. 1 2   ] [.3  ] ] [. 4 5  ] ]    [.  6  [ 7 8  ] ].  ]    \quad $\longrightarrow$ \Tree [ [ [ 1   [ 2 3  ] ] [. 5 4  ] ]    [. [ 6 7  ] [.8  ] ].  ]  
\end{minipage}
\end{center}
Then $Y$ is the minimal expansion of $\cxx$ associated to $\psi$. The minimal expansion of $\cxx$ contained in 
 $Y\mA\cup Z\mA$ is 
\[X=\{x\a_1^3,x\a_1^2\a_2,x\a_1\a_2\a_1,x\a_1\a_2^2,x\a_2\a_1,x\a_2^2\}.\] 
Then $X\mA\setminus (Y\mA\cap Z\mA)=\{ x\a_1^3,x\a_1^2\a_2,x\a_2\a_1,x\a_2^2\}$.
The $X$-components of these elements are
\begin{align*}
\cdots \mapsto x\a_1^4 &\mapsto x\a_1^3    &    x\a_1^2\a_2 &\mapsto x\a_1^2\a_2^2 \mapsto \cdots \\
\cdots \mapsto x\a_2^3 &\mapsto x\a_2^2    &    x\a_2\a_1   &\mapsto x\a_2\a_1^2  \mapsto \cdots,
\end{align*}
so $\psi$ is in quasi-normal form with respect to $X$. 
Introduce new generators  $x_1=x\a_1^3$, $x_2=x\a_1^2\a_2$, $x_3=x\a_1\a_2\a_1$, $x_4=x\a_1\a_2^2$, $x_5=x\a_2\a_1$ and  $x_6=x\a_2^2$. 
Then $X_P=\{x_3,x_4\}$  and  $X_{RI}=\{x_1,x_2,x_5,x_6\}$.

Let $V_{2,2}$ be free on $\{x_3,x_4\}$. Then, as an element of $G_{2,2}$ the map $\psi_P$ is the map sending 
$x_3$ to $x_4$ and $x_4$ to $x_3$. 
Let $V_{2,4}$ be free on $\{x_1,x_2,x_5,x_6\}$. We have 
\[Y_{RI}=\{x\a_1^4,x\a_1^3\a_2,x\a_1^2\a_2,x\a_2\a_1,x\a_2^2\a_1,x\a_2^3\}=\{x_1\a_1,x_1\a_2,x_2,x_5,x_6\a_1,x_6\a_2\}\]
and 
\[Z_{RI}=\{x\a_1^3,x\a_1^2\a_2\a_1,x\a_1^2\a_2^2,x\a_2\a_1^2,x\a_2\a_1\a_2,x\a_2^2\}=\{x_1,x_2\a_1,x_2\a_2,x_5\a_1,x_5\a_2,x_6\},\]
so as an element of $G_{2,4}$ the map $\psi_{RI}$ is given by the following forest diagram.
\begin{center}
\begin{tikzpicture}[root/.style={circle, inner sep=1pt, fill=black, draw=none}]
\node[left] at (0, 0){$\psi_{RI}:$};

\foreach \i[count=\x] in {1, 2, 5, 6}{
	\node[root, label=above:$x_\i$] (L\i) at (    \x, 0) {};
	\node[root, label=above:$x_\i$] (R\i) at (6 + \x, 0) {}; 
}
\def\caret#1#2#3{
	\draw (#1) -- +(-0.25, -0.4) node [label=below:#2] {};
	\draw (#1) -- +(+0.25, -0.4) node [label=below:#3] {};
}
\def\leaf#1#2{ \path (#1) node [below=1.5mm] {#2}; }
\caret{L1}{1}{2}
\leaf{L2}{3};
\leaf{L5}{4};
\caret{L6}{5}{6}

\draw[->] (5, -0.2) -- (6, -0.2);

\leaf{R1}{1};
\caret{R2}{2}{3}
\caret{R5}{4}{5}
\leaf{R6}{6};

\node[right] at (10, 0){\phantom{$\psi_{RI}:$}};
\end{tikzpicture}
\end{center}
\end{example}

\begin{definition}\label{infiniteelement}
Let $\psi$ be an element of $G_{n,r}$. Then $\psi$ is called \emph{periodic} if $V_{RI}=\emptyset$ and  $\psi$ is called 
\emph{regular infinite} if $V_{P}=\emptyset$.
\end{definition}
\begin{lemma}\label{TorFixY}
Let $\psi$ be an element of $G_{n,r}$ in semi-normal form with respect to an $A$-basis $X$.  
\be
\item
$\psi$ is periodic if and only if $\psi$ permutes the elements of $X$. 
\item $\psi$ is regular infinite if and only if no element of $X$ is of type (A).
\ee
\end{lemma}
\begin{proof}
\be
\item
If $\psi$ permutes the elements of $X$ then $X$ contains no element of type (B) or (C); so $X=X_P$ and $V_{n,r}=V_P$, by Theorem \ref{decomp}. 
As $V_{n,r}$ is the free product of $V_P$ and $V_{RI}$ it follows that 
$V_{RI}=\emptyset$, so $\psi$ is periodic.

If $\psi$ is 
periodic then $X_{RI}\subseteq V_{RI}=\emptyset$, 
so $X=X_P$. Thus $X$ consists of elements of type (A), which are permuted by $\psi$, by Lemma \ref{ABC}.
\item
If $\psi$ is regular infinite then $V_{P}=\emptyset$, so $X_p=\emptyset$; i.e.\  no element of $X$ is of type (A).  
If $X$ contains no element of type (A) then $X_P=\emptyset$, and  therefore  $V_P=\emptyset$ by Theorem \ref{decomp},
so $\psi$ is regular infinite.
\qedhere
\ee
\end{proof}

It follows that, in the notation established above Example \ref{ex:sub2full},  the automorphism $\psi_{P}\in G_{n,a}$ is periodic and 
$\psi_{RI}\in G_{n,b}$ is regular infinite. Thus, the decomposition of Theorem \ref{decomp} may be viewed as 
factoring $\psi$ into a product of a periodic and a regular infinite automorphism. It remains to see 
how to regard a pair of  automorphisms in this way, simultaneously in the same algebra. 

To this end suppose that $\psi_i\in G_{n,a_i}$  is in semi-normal form with respect to an $A$-basis $X_i$,
where $|X_i|=a_i$, for $i=1,2$. 
If there exists an isomorphism $\rho:V_{n,a_1}\maps V_{n,a_2}$
with the property that $\rho^{-1}\psi_1\rho=\psi_2$ then,  from  Corollary \ref{cor:H2}, $a_1\equiv a_2\mod {n-1}$. 
Also, if  $a_1\equiv a_2 \mod n-1$ then $V_{n,a_i}$  is isomorphic to $V_{n,s}$ where $1\le s\le n-1$ and 
$s\equiv a_i$.    If this is the case then we may take an $A$-basis $\cxx_s$ of $s$ elements of $V_{n,s}$ and choose expansions
$X'_1$ and $X'_2$ of $\cxx_s$ of $a_1$ and $a_2$ elements respectively. Now let $f_i$ be the map taking $X_i$ to $X'_i$. 
Then there exists an isomorphism $\rho:V_{n,a_1}\maps V_{n,a_2}$ such that $\rho^{-1}\psi_1\rho=\psi_2$ 
 if and only if  $a_1\equiv a_2\mod n-1$ and, setting $\widehat\psi_i=f_i^{-1}\psi_if_i\in G_{n,s}$, we have 
$\rho^{-1}f_1\widehat\psi_1f_1^{-1}\rho=f_2\widehat\psi_2f_2^{-1}$: that is 
$\theta^{-1}\widehat\psi_1\theta = \widehat\psi_2$, where $\theta = f_1^{-1}\rho f_2\in G_{n,s}$.  (See Figure \ref{fig:comm_diag}.)

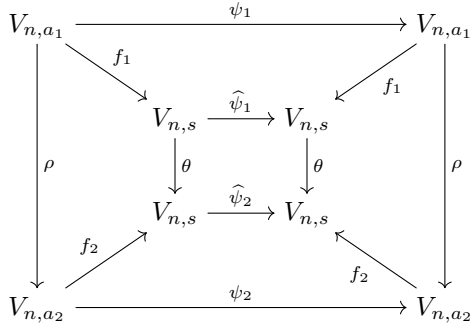
\begin{figure}
\begin{center}
\[
\begin{tikzcd}
V_{n,a_1} \arrow{rrr}{\psi_1}\arrow{rd}{f_1}\arrow{ddd}{\rho}
&
&
&V_{n,a_1}  \arrow{ld}{f_1}\arrow{ddd}{\rho}\\
&V_{n,s}\arrow{r}{\widehat\psi_1}\arrow{d}{\theta}& V_{n,s}\arrow{d}{\theta}&\\
&V_{n,s}\arrow{r}{\widehat\psi_2}& V_{n,s}&\\
V_{n,a_2}  \arrow{rrr}{\psi_2}\arrow{ru}{f_2}
& 
&
&V_{n,a_2}  \arrow{lu}{f_2}
\end{tikzcd}
\]
\end{center}
\caption{Isomorphisms of $V_{n,a_i}$ and $V_{n,s}$}\label{fig:comm_diag}
\end{figure}

Combining this with Theorem~\ref{decomp}.\ref{it:decomp1} gives a decomposition of the conjugacy problem into
the conjugacy problem for periodic and for regular infinite elements, separately.
Let $\psi$ and $\phi$ be elements of $G_{n,r}$, write $V_{n,a_1}= V_{RI,\psi}$, $\psi_1=\psi_{RI}$, $V_{n,a_2}=V_{RI,\phi}$ 
and $\psi_2=\phi_{RI}$. Using the 
procedure above, if $\rho_{RI}$ exists (in the notation of Theorem \ref{decomp}) then we may 
regard $\psi_i$, $i=1,2$,  as a regular infinite element of $G_{n,s}$, namely $\widehat \psi_i$, for appropriate $s$.  
Similarly, we may regard $\psi_P$ and $\phi_P$ as periodic automorphisms of a single algebra.  

We can now outline the algorithm for the conjugacy problem.
\subsection{The conjugacy algorithm}\label{algorithmone}

\begin{algorithm}\label{conjAlgorithm}
Let $\psi$ and $\phi$ be an elements of $G_{n,r}$. 
\be[\bfseries Step 1:]
\item\label{it:ca1} Find $A$-bases $X_\psi$ and $X_\phi$ such that $\psi$ and $\phi$ are in quasi-normal form with respect to $X_\psi$ and $X_\phi$,
respectively, as in Lemma \ref{lem:qnf}.  The sets 
$X_{P,\psi}$, $X_{RI,\psi}$, $X_{P,\phi}$ and  $X_{RI,\phi}$ are obtained as part of this process. 

If $|X_{P,\psi}|\equiv |X_{P,\phi}|\mod n-1$ and $|X_{RI,\psi}|\equiv |X_{RI,\phi}|\mod n-1$; continue. Otherwise output ``No'' and stop. 
\item \label{it:ca2}
Find the minimal expansion $Y_\psi$  of $X_\psi$ associated to $\psi$  and the minimal expansion 
$Y_\phi$  of $X_\phi$ associated to $\phi$.  (See Lemma \ref{4.1H}.)
Construct $Y_{RI,\psi}$ and  $Y_{RI,\phi}$; the sets elements of $Y_\psi$ and $Y_\phi$ which are not in finite orbits (as in the 
the discussion following Theorem \ref{decomp}). Construct $Z_{RI,\psi}=Y_{RI,\psi}\psi$ and $Z_{RI,\phi}=Y_{RI,\phi}\phi$.
\item \label{it:ca3}
For $T=P$ and for $T=RI$ carry out the following.
Find the integer $s_T$ such that $1\le s_T\le n-1$ and $s_T\equiv |X_{T,\psi}|$. Let $\cxx_T$ be a set of $s_T$ elements,
let $V_{n,s_T}$ be free on $\cxx_T$ and 
find expansions $W_{T,\psi}$ and $W_{T\phi}$ of $\cxx_T$ of sizes $|X_{T,\psi}|$ and $|X_{T,\phi}|$, respectively. Construct a map
$f_{T,\psi}$  mapping $X_{T,\psi}$ bijectively to $W_{T,\psi}$ and  $f_{T,\phi}$ mapping $X_{T,\phi}$ bijectively to $W_{T,\phi}$. 
Write $\psi_T$ and $\phi_T$ as elements of $G_{n,s_T}$, using these maps.
\item\label{it:ca4} Input $\psi_P$ and $\phi_P$ into Algorithm \ref{alg:periodic} below for conjugacy of periodic elements of  $G_{n,r}$. 
If  $\psi_P$ and $\phi_P$ are not conjugate, return ``No'' and stop. Otherwise obtain 
a conjugating element $\rho_P$. 
\item\label{it:ca5} Input $\psi_{RI}$ and $\phi_{RI}$ into Algorithm \ref{alg:reginf} below for conjugacy of regular infinite elements of  
$G_{n,s_{RI}}$. If  $\psi_{RI}$ and $\phi_{RI}$ are not conjugate, return ``No'' and stop. Otherwise obtain 
a conjugating element $\rho_{RI}$. 
\item\label{it:ca6} Return the conjugating element $\rho_P\ast \rho_{RI}$. 
\ee
\end{algorithm}

Given this algorithm we have the following theorem. 
\begin{theorem}[{\cite[Theorem 9.3]{Higg}}] \label{CPS}
The conjugacy problem is soluble in $G_{n,r}$.
\end{theorem}

\begin{proof}
Apply Algorithm \ref{conjAlgorithm}.
\end{proof}

\subsection{Conjugacy of periodic elements}\label{subsectionP}
Let $\psi\in G_{n,r}$ be  a periodic element. For $u\in V_{n,r}$ the \emph{size} of the $\psi$-orbit of $u$ is the least positive integer $d$ such that 
$u\psi^d=u$. 
\begin{definition}\label{cycletypes}
Let $\psi$ be a periodic element of $G_{n,r}$ in semi-normal form with respect to the $A$-basis $X$. The \emph{cycle type} of $\psi$ is the set 
\[T_\psi(X)=\{d\in \NN \mid \text{some $x \in $ has a $\psi$-orbit of size $d$} \}.\]
For $d\in \NN$, define the $\psi$\emph{-multiplicity} of $d$ to be $m_\psi(d,X)=D/d$, where $D$ is the number of elements of 
$X$ which belong to a $\psi$-orbit of size $d$.
\end{definition}
Note that, as $\psi$ is periodic and in semi-normal form with respect to $X$, all $X$-components of $\psi$ are (ordered) $\psi$-orbits and all
$\psi$-orbits of elements of $\XA$ are $X$-components (once ordered appropriately). Also,  
 $d\in T_\psi(X)$ if and only if $m_\psi(d,X)\neq 0$; the size of the set $X$ is $|X|=\sum_{d\in T_\psi(X)} dm_\psi(d,X)$;  if $d\in T_\psi(X)$ then $X$ contains  
$m_\psi(d,X)$ disjoint $\psi$-orbits of size $d$; and $\psi$ is a torsion element of order equal to the least common multiple of elements of $T_\psi(X)$. 

\begin{example}
Let $n=2$, $r=1$ and $V_{2,1}$ be free on $\cxx = \{x\}$. Let  
\[X=\{x\alpha_1^3,x\alpha_1^2\alpha_2,x\alpha_1\alpha_2,x\alpha_2\alpha_1^2,x\alpha_2\alpha_1\alpha_2,x\alpha_2^2\alpha_1,x\alpha_2^3\}\]
and let $\psi$ be the periodic element of $G_{2,1}$ defined by the tree pair diagram below.
\begin{center}
$\psi:$
\begin{minipage}[b]{0.7\linewidth}
\centering
\Tree  [  [. [1 2  ] [.3  ] ]. [ [. [.4  ] [.5  ] ].  [. [.6  ] [.7  ] ].  ] ] \quad $\longrightarrow$
\Tree  [  [. [3 1  ] [.2  ] ]. [ [. [.5  ] [.4  ] ].  [. [.7  ] [.6  ] ].  ] ]  
\end{minipage}
\end{center}
Then the cycle type of $\psi$ is $\{2,3\}$ with multiplicites $m_\psi(2,X)=2$ and $m_\psi(3,X)=1$.
\end{example}

\begin{lemma}\label{lem:perex}
Let $\psi$ be a periodic element of $G_{n,r}$ in semi-normal form with respect to the $A$-basis $X$ and the $A$-basis  $Z$, where $Z$ is a
$q$-fold expansion of $X$.  Then $T_\psi(X)=T_\psi(Z)$ and $m_\psi(d,X)\equiv m_\psi(d,Z) \mod{n-1}$, for all $d\in T_\psi(X)$. 
\end{lemma}

\begin{proof}
It suffices to prove the lemma in the case where~$Z$ is a \emph{simple} expansion of~$X$, because any expansion is obtained by a finite sequence of simple expansions. Suppose the expansion happens at~$w \in X$, so that $Z = (X \setminus \{ w \}) \cup \{ w\alpha_1, \dotsc, w\alpha_n \}$. To compute $T_\psi(Z)$ we need to break~$Z$ into a union of $\psi$-orbits.

Let~$d$ be the size of the $\psi$-orbit of~$w$, so that $\mathcal{O}_w = \{ w, w\psi, \dotsc, w\psi^{d-1} \}$.
For each $1 \leq i \leq n$ the orbit of~$w\alpha_i$ is $\mathcal{O}_{w\alpha_i} = \{ w\alpha_i, w\psi\alpha_i, \dotsc, w\psi^{d-1}\alpha_i\}$, which is of size at most $d$.
In fact its size is exactly~$d$: if there are integers $0 \leq j < k < d$ for which $w\psi^j\alpha_i = w\psi^k\alpha_i$, we would have $w\psi^j = w\psi^k$, which cannot occur.

Thus, in moving from~$X$ to~$Z$ we have lost~$1$ and gained~$n$ $\psi$-orbits of size~$d$; all other $\psi$-orbits inside~$Z$ are $\psi$-orbits inside~$X$.
Therefore $T_\psi(X) = T_\psi(Z)$.
In terms of multiplicities this means $m_\psi(d, Z) = m_\psi(d, X) + n - 1$ and $m_\psi(e, Z) = m_\psi(e, X)$, for every positive integer $e \neq d$; whence the result.
\end{proof}

Note that it follows from this lemma that if $\psi$ is in semi-normal form with respect to both $X$ and $X'$ then $T_\psi(X)=T_\psi(X')$, since we
may take a common expansion of both $X$ and $X'$ and then expand this to an $A$-basis $Z$ with respect to which $\psi$ is in semi-normal form. 
So from now on,  we refer to the cycle type $T_\psi$ without reference to an $A$-basis $X$.

\begin{proposition}\label{conjTorsion}
Let $\psi$ and $\varphi$ be periodic elements of $G_{n,r}$
 in semi-normal form with respect to the $A$-bases $X_\psi$ and $X_\phi$, respectively. Then 
$\psi$ is conjugate to $\varphi$ if and only if 
\be
\item\label{it:ct1} $T_\psi=T_\phi$ and 
\item\label{it:ct2} $m_\psi(d,X_\psi)\equiv m_\phi(d,X_\phi)\mod n-1$, for all $d\in \NN$. 
\ee
\end{proposition}

\begin{proof}
Assume that $\psi$ and $\phi$ are conjugate and let $\rho\in G_{n,r}$ be such that $\rho^{-1}\psi\rho=\phi$. Let $\rho$ be in semi-normal 
form with respect to $X_\rho$, let $Y$ be the minimal expansion of $X_\rho$ associated to $\rho$ and let $Z=Y\rho$. Let $W$ be a common
expansion of $X_\psi$ and $Y$ and let $\psi$ be in semi-normal form with respect to an expansion $X'_\psi$ of $W$. (Such an expansion
of $W$ exists, by Lemma \ref{9.9H}.)  As $\psi$ is periodic and in semi-normal form it permutes the elements of  $X'_\psi$, so
for all $x\in  X'_\psi$ we have $x'\in  X'_\psi$ such that $x\rho\phi=x\psi\rho=x'\rho\in X\rho$. Therefore $\phi$ permutes the elements
of $X'_\psi\rho$, so $\phi$ is in semi-normal form with respect to $X'_\phi=X'_\psi\rho$. As $X'_\psi$ is an expansion of $Y$ and $Z=Y\rho$ it follows
that $X'_\phi$ is an expansion of $Z$.

Now if $x\in X'_\psi$ and $i\in \ZZ$ then $x\rho\phi^i=x\psi^i\rho$, so we have $x\psi^d=x$ if and only if $x\rho\phi^d=x\rho$;
in other words, $x$ and $x\rho$ have orbits of equal size.
This applies to any~$x$, so $T_\psi=T_\phi$ and 
both  $X'_\psi$ and  $X'_\phi$ have 
the same 
number of elements with an orbit of size $d$. 
Therefore $m_\psi(d,X'_\psi)=m_\phi(d,X'_\phi)$, for all $d\in T_\psi=T_\phi$.
Statement~\ref{it:ct2} follows, from Lemma \ref{lem:perex} and the 
fact that $X'_\psi$ and $X'_\phi$ are expansions of $X_\psi$ and $X_\phi$, respectively. 

Conversely, suppose that statements~\ref{it:ct1} and~\ref{it:ct2} hold.
Let $T_\psi=T_\phi=\{d_1,\ldots, d_k\}$ and write $m_j=m_\psi(d_j,X_\psi)$ and $m'_j=m_\phi(d_j,X_\phi)$.
Fix $j\in \{1,\ldots ,k\}$.
Assume first that $m_j>m'_j$.
Then, by hypothesis, $m_j=m'_j+q_j(n-1)$ for some positive integer $q_j$.
Select an element $x\in X_\phi$ whose $\phi$-orbit $\cO_x$ has size $d_j$.
Let $Y_x$ be a $q_j$-fold expansion of $\{x\}$ and set $E=\{\G\in A^* \mid x\G\in Y_x\}$, so that $Y_x=xE$.
Then for each $0 \leq i < d_j$, $x\phi^iE$ is a $q_j$-fold expansion of $\{x\phi^i\}$.

For every string $\Gamma \in E$, the set $\{ x\Gamma, x\phi\Gamma, \dotsc, x\phi^{d-1}\Gamma \}$ is a $\phi$-orbit of size~$d$.
(We saw this in Lemma~\ref{lem:perex} for $\Gamma = \alpha_i$.)
Hence the set $\cO_x E=\{x\phi^i\G \mid \G\in E, 0\le i < d_j \}$ is a $q_jd_j$-fold expansion of $\cO_x$; more precisely it is a disjoint union of~$|E| = q_j(n-1)$ $\phi$-orbits of size~$d_j$.
After $\cO_x$ is expanded to $\cO_xE$, the resulting expansion $X_\phi$ has exactly $m'_j + q_j(n-1) = m'_j$ size~$d_j$ $\phi$-orbits. 

For each $j$ such that $m_j>m'_j$ apply this process to a single element of $X_\phi$ with $\phi$-orbit size $d_j$.
Dually, for each  $j$ such that $m'_j>m_j$ apply the process to an element of $X_\psi$ with $\psi$-orbit size $d_j$, interchanging the roles of $\phi$ and $\psi$.
The result is an expansion $X'_\psi$ of $X_\psi$ and
an expansion $X'_\phi$ of $X_\phi$ such that $m_\phi(d, X'_\phi) = m_\psi(d,X'_\psi)$ for every positive integer~$d$.

Now define 
$\rho:X'_\psi\maps X'_\phi$ by mapping orbits of size $d$ to each other, preserving the order within each orbit.
In detail, for each~$d$ set $m = m_\psi(d, X_\psi) = m_\phi(d, X_\phi)$.
Let $\cO_1,\ldots, \cO_m$ be the size~$d$ $\psi$-orbits (in any order) in $X'_\psi$ and let $\cO'_1,\ldots, \cO'_m$ be the size~$d$ $\phi$-orbits in $X'_\phi$ (also in any order).
Select a representative $o_i \in \cO_i$ and $o'_i \in \cO'_i$ for each of these $2m$ orbits.
We define $\rho$ by the rule $o_i\psi^j \rho = o'_i \phi^j$.
By construction we have $x\psi\rho = x\rho\phi$, for all $x\in X'_\psi$.
Hence $\rho^{-1}\psi\rho=\phi$.
\end{proof}

\begin{example}\label{periodicConjEx}
Let $n=2$, $r=1$ and $V_{2,1}$ be free on $\cxx = \{x\}$. Let 
\[X=\{x\alpha_1^4,x\alpha_1^3\alpha_2,x\alpha_1^2\alpha_2,x\alpha_1\alpha_2,x\alpha_2\alpha_1, x\alpha_2^2\}\]
and let 
$\psi$ be the periodic element of $G_{2,1}$ 
given by the tree pair diagram below. 
\begin{center}
$\psi:$
\begin{minipage}[b]{0.55\linewidth}
\centering
\Tree [ [ [ [. 1 2   ] [.3  ] ] [.4  ] ]   [. 5 6  ] ] \quad $\longrightarrow$\Tree [ [ [ [. 2 1   ] [.4  ] ] [.3  ] ]   [. 6 5  ] ]  
\end{minipage}
\end{center}
Then $\psi$ has cycle type $T_\psi=\{2\}$ and multiplicity $m_\psi(2,X)=3$. The $\psi$-orbits of elements of $X$ are 
$\cO_1=\{x\alpha_1^4, x\alpha_1^3\alpha_2\}$, $\cO_2=\{x\alpha_1^2\alpha_2,x\alpha_1\alpha_2 \}$ and $\cO_3=\{x\alpha_2\alpha_1,x\alpha_2^2\}$. 

Let $Y=\{x\alpha_1,x\alpha_2\}$ and let $\varphi$ be the periodic element of $G_{2,1}$ which swaps the elements of~$Y$.
\begin{center}
$\varphi:$
\begin{minipage}[b]{0.2\linewidth}
\centering
\Tree  [   1 [.2  ] ] \quad $\longrightarrow$\Tree [   2 [.1  ] ] 
\end{minipage}
\end{center}
Then  $\phi$ has cycle type $T_\phi=\{2\}$ and $m_\phi(2,Y)=1$.
From Proposition \ref{conjTorsion}, $\psi$ is conjugate to~$\varphi$. We can construct a conjugator 
by applying the process of the proof. 
We take the same  $2$-fold expansion of both $x\a_1$ and $x\a_2$ to give a $4$-fold expansion 
 \[Y'=\{x\alpha_1^3,x\alpha_1^2\alpha_2,x\alpha_1\alpha_2,x\alpha_2\alpha_1^2,x\alpha_2\alpha_1\alpha_2,x\alpha_2^2\}\]
of $Y$ 
such that $\phi$ is in semi-normal form with respect to $Y'$.
The $\phi$-orbits of elements of $Y'$ are $\cO'_1=\{x\a_1^3, x\a_2\a_1^2\}$, $\cO'_2=\{x\a_1^2\a_2, x\a_2\a_1\a_2\}$ and $\cO'_3=\{x\a_1\a_2,x\a_2^2\}$, so $m_\phi(2,Y')=3$.
Take the representative of each orbit to be the first element listed in its description.
The corresponding conjugator $\rho$ is the element of $G_{2,1}$ which sends $\cO_i$ to $\cO'_i$ via 
 $x\alpha_1^4\rho=x\alpha_1^3$, $x\alpha_1^3\alpha_2\rho=x\alpha_2\alpha_1^2$, $x\alpha_1^2\alpha_2\rho=x\alpha_1^2\alpha_2$, $x\alpha_1\alpha_2\rho=x\alpha_2\alpha_1\alpha_2$, $x\alpha_2\alpha_1\rho=x\alpha_1\alpha_2$ and $x\alpha_2^2\rho=x\alpha_2^2$.
\begin{center}
$\rho:$
\begin{minipage}[b]{0.55\linewidth}
\centering
\Tree [ [ [ [. 1 2   ] [.3  ] ] [.4  ] ]   [. 5 6  ] ] \quad $\longrightarrow$\Tree  [  [. [1 3  ] [.5  ] ]. [ [. [.2  ] [.4  ] ]. 6  ] ]  
\end{minipage}
\end{center}
Then $\rho^{-1}\psi \rho=\phi$. 
\end{example}

From the proof of Theorem \ref{conjTorsion} we extract the following algorithm for the conjugacy of periodic elements of $G_{n,r}$. 
\begin{algorithm}\label{alg:periodic}
Let $\psi$ and $\phi$ be periodic elements of $G_{n,r}$.
 \be[\bfseries Step 1:]
\item Construct $A$-bases $X_\psi$ and $X_\phi$ with respect to which $\psi$ and $\phi$ are in semi-normal form (Lemma \ref{9.9H}).
\item Compute  the cycle types $T_\psi$ and $T_\phi$. If $T_\psi\neq T_\phi$, output ``No'' and stop.
\item Compute $m_\psi(d,X_\psi)$ and $m_\phi(d,X_\phi)$, for all $d\in T_\psi$. If $m_\psi(d,X_\psi)\not\equiv m_\phi(d,X_\phi)\mod n-1$, output ``No'' and stop.
\item Construct $A$-bases $X'_\psi$ and $X'_\phi$ as described in the proof of Theorem \ref{conjTorsion}.
\item Choose a map $\rho$ sending $\psi$-orbits of elements  of $X'_\psi$ to $\phi$-orbits of elements of $X'_\phi$, as in the proof of the theorem, and output $\rho$. 
\ee
\end{algorithm}
\subsection{Conjugacy of regular infinite elements}\label{subsectionRI}

We begin with a necessary condition for two regular infinite elements to be conjugate. 
Let $\psi$ be a regular infinite element of $G_{n,r}$ in semi-normal form with respect to $X$. 
By Lemma \ref{9.1H}, $\psi$ has finitely many semi-infinite $X$-components, each of which has 
a characteristic element~$u$ with some characteristic~$(m, \G)$ (see Definition \ref{charOrbit}).  
If $\psi$ is also in semi-normal form with respect to~$Y$, the $\psi$-orbit of~$u$ has precisely one $Y$-component, which is again semi-infinite of characteristic $(m,\G)$. 
Therefore, the set of pairs $(m,\G)$ which are characteristics of semi-infinite $X$-components is independent of the 
choice of a basis for a semi-normal form.
With this in mind, we make the following definition.  

\begin{definition}\label{setMultipliers}
Let $\psi$ be a regular infinite element of $G_{n,r}$ in semi-normal form with respect to $X$. 
Define 
\[\cM_\psi=\{(m,\Gamma) \mid \text{$(m,\Gamma)$ is the characteristic of a semi-infinite $X$-component of $\psi$} \}.\]
\end{definition}

\begin{example} \label{lookingintheorbit} \label{ex:lio1}
We refer to the following example through the remainder of this section. 
Let $n=2$, $r=1$, $\cxx=\{x\}$ and $\varphi\in G_{2,1}$ be determined by the bijection from $A$-basis 
\[Y=\{x\alpha_1,x\alpha_2\alpha_1,x\alpha_2^2\alpha_1^2,x\alpha_2^2\alpha_1\alpha_2,x\alpha_2^3\}\]
to the  $A$-basis
\[Z=\{x\alpha_1^3,x\alpha_1^2\alpha_2,x\alpha_1\alpha_2,x\alpha_2\alpha_1,x\alpha_2^2\}\]
as illustrated below.
\begin{center}
$\varphi:$
\begin{minipage}[b]{0.5\linewidth}
\centering
\Tree [. 1 [ 2  [  [3 4 ] 5 ] ] ] \quad $\longrightarrow$\Tree [ [ [ [.1  ] [.2  ] ] [.4  ] ]   [. 5 3  ] ] 
\end{minipage}
\end{center}
Then $Y$ is  the minimal expansion of $\cxx$ associated to $\varphi$ and $Z=Y\varphi$. The elements of $\xA\setminus (Y\mA\cup Z\mA)$ are 
$x$ and $x\a_2$, so we start the search for a quasi-normal form by taking the unique minimal expansion 
$X=\{x\a_1,x\a_2\a_1,x\a_2^2\}$ of $\cxx$ not containing either of these elements.

The $X$-component of $x\a_1$ is 
\[
x\a_1\mapsto x\a_1^3\mapsto x\a_1^5\mapsto \cdots,
\]
which is right semi-infinite of characteristic $(1,\a_1^2)$.
Next, $x\a_2\a_1$ belongs to a complete infinite $X$-component: 
\[
\cdots x\a_2^5 \mapsto x\a_2^4 \mapsto x\a_2^3  \mapsto x\a_2\a_1\mapsto x\a_1^2\a_2 \mapsto x\a_1^4\a_2\mapsto x\a_1^6\a_2\mapsto\cdots
\]
Finally, the $X$-component of $x\a_2^2$ is 
\[
\cdots \mapsto x\a_2^2\a_1^4  \mapsto x\a_2^2\a_1^2 \mapsto x\a_2^2,
\]
which is left semi-infinite of characteristic $(-1,\a_1^2)$. Thus $\varphi$ is in quasi-normal form with respect to $X$.

To determine $\cM_\varphi$, we compute the sets $X\mA\setminus Y\mA=\{x\a_2^2,x\a_2^2\a_1\}$ and $X\mA\setminus Z\mA=\{x\a_1,x\a_1^2\}$. The $X$-components we have yet to calculate are those of $x\a_2^2\a_1$ and $x\a_1^2$; these are the sets $\{x\a_2^2\a_1^{2i-1} \mid i\ge 1 \}$ and $\{x\a_1^{2i} \mid i\ge 1\}$ with characteristics $(1, \a_1^2)$ and $(-1, \a_1^2)$ respectively. Hence
\[
	\cM_\phi=\{(1,\a_1^2),(-1,\a_1^2)\}. 
\]
\end{example}

\begin{lemma}\label{conjugateMultipliers}
Let $\psi$ and $\varphi$ be regular infinite elements of $G_{n,r}$ in semi-normal form with 
respect to $A$-bases $X$ and $Y$ respectively. 
Suppose that the elements are conjugate via $\rho\in G_{n,r}$ with $\rho^{-1}\psi\rho=\varphi$.
Then the sets $\mathcal{M}_{\psi}$ and $\mathcal{M}_{\varphi}$ coincide. 
Moreover, $\rho$ maps a semi-infinite $X$-component of $\psi$ into a $\phi$-orbit which contains a (unique) semi-infinite $Y$-component with the same characteristic.
\end{lemma}

\begin{proof}
If $u$ is an element of  $X\mA$ 
such that $u\psi^m=u\G$, for some $m$ and $\G$, then
\[u\rho\varphi^m=u\psi^m\rho=u\Gamma\rho=u\rho\Gamma.\]
The same argument can be applied starting with an element $v\in Y\mA$  and interchanging $\psi$ and $\varphi$. 
Hence if $u$ belongs to a  $\psi$-orbit of characteristic $(m,\G)$ then 
$u\rho$ belongs to an $\phi$-orbit of  characteristic $(m,\G)$. Thus, from Lemma \ref{AJDLEMMA2}, a $\psi$-orbit that 
contains a semi-infinite
$X$-component of 
 characteristic $(m,\G)$ is mapped by $\rho$ to a $\phi$-orbit which has a semi-infinite $Y$-component of the same  characteristic.
\end{proof}

\begin{definition}\label{EqCl}
Let $\psi$ be in semi-normal form with respect to $X$. The equivalence relation $\equiv$ on  $X$, is  that generated by the relation 
$x \equiv x'$, whenever $x\Gamma$ and $x'\Delta$ are in the same $\psi$-orbit, for some $\Gamma, \Delta \in A^*$.
\end{definition}

\begin{example}\label{ex:lio2}
Let $\phi$ be as in Example \ref{lookingintheorbit}. Then $x\a_2\a_1\phi=(x\a_1)\a_1\a_2$, so $x\a_2\a_1\equiv x\a_1$. Also, 
$x\a_2\a_1\phi^{-1}=(x\a_2^2)\a_2$, so $x\a_2\a_1\equiv x\a_2^2$. Therefore all elements of $X$ are related by $\equiv$. 
\end{example}

\begin{proposition}\label{propsplitY} Let $\psi$ be a regular infinite element in semi-normal form with respect to $X$. 
Let $X=\coprod_{i=1}^{m}\mathcal{X}_i$, where the $\mathcal{X}_i$ are the equivalence classes of $\equiv$ defined on $X$ under the action of $\psi$. 
Then $V_{n,r}$ is the free product of the $\psi$-invariant $\Omega$-subalegbras $V_1,\ldots ,V_m$, where $V_i$ is the $\Omega$-subalgebra generated by $\mathcal{X}_i$.
\end{proposition}

\begin{proof}
As $\psi$ is regular infinite, the sets $\mathcal{X}_i$ partition $X$, so $V_{n,r}$ is the free product of the $V_{i}$'s. 
To show that $V_i$ is $\psi$-invariant it suffices to show that
if $x\in \cX_i$ then $x\psi$ and $x\psi^{-1}$ are in $V_i$. To this end, choose $d\ge 0$ such that $x\psi \G$ and $x\psi^{-1}\G$
belong to $X\mA$, for all $\G\in A^*$ of length $d$.
Then for all such~$\G$ we have 
$x\psi\G=y\D$ and $x\psi^{-1}\G=z\Lambda$, for some $y,z\in X$ and $\D,\Lambda \in A^*$. By definition then $y\equiv x\equiv z$,
so $x,y,z\in \cX_i$. This implies that $x\psi\G=y\D\in V_i$ and $x\psi^{-1}\G=z\Lambda \in V_i$. This holds for 
all $\G$ of length $d$, so from Lemma \ref{AJDLEMMAX}, $x\psi$ and $x\psi^{-1}$ belong to $V_i$, as required. Hence
$V_i$ is $\psi$-invariant. 
\end{proof}

\begin{lemma}\label{lem:effsplit}
Let $\psi$ be a regular infinite element in semi-normal form with respect to $X$ and let  
 $\cX_i$, $i=1,\ldots ,m$, be the equivalence classes of $\equiv$ defined on $X$ under the action of $\psi$.
We may effectively construct the $\cX_i$.
\end{lemma}
\begin{proof}
From Lemmas \ref{lem:qnf} and \ref{4.1H}, we may effectively construct $X$, the minimal expansion
$Y$ of $\psi$ with respect to $X$, and the basis $Z=Y\psi$. For each $v\in X\cup Y\cup Z$ 
we may enumerate a finite subsequence~$C_v$ of the $X$-component of $v$
using the procedure of Lemma \ref{lem:qnf}.
Let $\equiv_0$ be the equivalence relation on $X$ generated by $y\equiv_0 z$ if $y\G$ and $z\D$ belong to $C_v$, 
for some $v\in X\cup Y \cup Z$ and $\G, \D\in A^*$. We claim that $\equiv_0 \,= \,\equiv$. 

By definition, $\equiv_0\subseteq \equiv$.   
To prove the opposite inclusion, we suppose that there exist $p\in \ZZ$, $x, y\in X$ and  $\D, \Phi\in A^*$ such that 
$x\Phi=y\D\psi^p$ and $x$ and $y$ are not related under the relation $\equiv_0$. In this case we may assume, interchanging
$x$ and $y$ if necessary, that $p>0$. Let $p$ be a minimal positive integer for which such $x,y$ exist. 
As $y\D\psi^p=x\Phi$ it follows that $y\D\psi^{p'}\in X\mA$, for $p'=1,\ldots, p-1$. Let $y\D\psi =y'\D'$, so 
$y'\D'\psi^{p-1}=x\Phi$. By
minimality of $p$ we have $y'\equiv_0 x$.

Let $\D_0$ be an initial subword of $\D$ of maximal length such that
$y\D_0\psi\in X\mA$, say $\D=\D_0\D_1$. Then $y\D_0\in Y$ and $y\D_0\psi=y''\D_0''$, for some $y''\in X$ and $\D_0''\in A^*$. 
 Now $y'\D'=y\D_0\D_1\psi=y''\D_0''\D_1$, so $y''=y'$ and $\D'=\D_0''\D_1$. Thus $y\D_0\psi=y'\D_0''$ and, as $y\D_0\in Y$,  
$y'\D_0''\in Z$ we have $y\equiv_0 y'$. Therefore 
$y\equiv_0 x$, a contradiction. We conclude that no such $p$, $x$ and $y$ exist and so $\equiv\subseteq \equiv_0$,
as required. Thus $\equiv_0=  \equiv$, and as we may effectively  compute the sets $C_v$, it follows
that we may compute the equivalence classes $\cX_i$.  
\end{proof}

\begin{lemma}\label{thetas}
Let $\psi$ be a regular infinite element in semi-normal form with respect to $X$ and let  
 $\cX_i$, $i=1,\ldots ,m$ be the equivalence classes of $\equiv$ defined on $X$ under the action of $\psi$. 
Define  
\[x\theta_{i} = \left\{ \begin{array}{rl}
 x\psi &\mbox{ if $x \in \mathcal{X}_{i}$,} \\
  x &\mbox{ if $x\in \mathcal{X}_{j}$ for $i\ne j$,}
       \end{array} \right.\]
 for $i=1,\ldots ,m$. 
Then $\theta_i$ extends to an element of $G_{n,r}$ which commutes with $\psi$ and with $\theta_j$, for all $j = 1, \dotsc, m$.
\end{lemma}

\begin{proof}
Let $V_i$ be the $\Omega$-subalgebra generated by $\cX_i$, $i=1,\ldots ,m$. 
Since $V_{n,r}=V_1\ast\cdots\ast V_m$ and  the $V_i$ are $\psi$ invariant,  we have 
$\psi=\psi_1 \ast\cdots \ast \psi_m$, where $\psi_i=\psi|_{V_i}$. Moreover $\psi_i$ is an automorphism of $V_i$. By definition, 
$\psi_i|_{\cX_i}=\psi|_{\cX_i}=\theta_i|_{\cX_i}$, so $\theta_i|_{\cX_i}$ extends to the automorphism $\psi_i$ of $V_i$. Thus
(the extension to $V_{n,r}$ of)  $\theta_i=1_{V_1}\ast \cdots \ast \psi_i\ast \cdots \ast 1_{V_m}$ is an automorphism of $V_{n,r}$. 
For $i<j$ we have $\theta_i\theta_j=1_{V_1}\ast \cdots \ast \psi_i\ast\cdots \ast \psi_j \cdots \ast 1_{V_m}=\theta_j\theta_i$, and
it follows that $\theta_i$ commutes with $\psi$.
\end{proof}

\begin{lemma}\label{firstCharEl}
Let $\psi$ and $\varphi$ be regular infinite elements of $G_{n,r}$, in semi-normal form with respect to the $A$-bases $X$ and $Y$ respectively.
Let $\cX_1, \dotsc, \cX_m$ be the equivalence classes of $\equiv$ defined on $X$ under the action of $\psi$.
Choose a representative $x_i \in \cX_i$ of type~(B) for each~$i$.
If $\psi$ and $\varphi$ are conjugate, there  
exists a conjugator $\rho$ such that $x_i\rho$ is a terminal or initial element in a semi-infinite $Y$-component of $\varphi$.
\end{lemma}

\begin{proof}
Let~$\rho' \in G_{n,r}$ be a conjugator with $\rho'^{-1}\psi\rho'=\varphi$.
We will explain how to modify~$\rho'$ to form another conjugator~$\rho$ satisfying the requirements of the lemma.
Lemma \ref{conjugateMultipliers} asserts that $x_i\rho'$ belongs to a  $\varphi$-orbit containing a semi-infinite $Y$-component, which has 
the same characteristic as $x_i$. 
Let $y_{i}\in Y\mA$ be the initial or terminal element of this $Y$-component.
Then there exists $j_i$ such that $x_i\rho'=y_i\varphi^{j_i}$, meaning that
\[y_i=y_i\varphi^{j_i}\varphi^{-j_i}=x_i\rho'\varphi^{-j_i}=x_i\psi^{-j_i}\rho'.\]
For each equivalence class $\mathcal{X}_i$, define  $\theta_i$ as in Lemma \ref{thetas} and $\rho\in G_{n,r}$ by 
\[\rho = \left( \prod_{i=1}^n \theta_i^{-j_i} \right) \rho'.\]
Then  $\theta=\prod_{i=1}^n\theta_i^{-j_i}$ commutes with $\psi$, so
$\rho^{-1}\psi\rho=\rho^{\prime -1}\theta^{-1}\psi\theta\rho'=\rho'^{-1}\psi\rho'=\phi$;
furthermore for each chosen $x_i\in \mathcal{X}_i$ we have
\[x_i\rho=x_i \left( \prod_{i=1}^n\theta_i^{-j_i} \right) \rho'=x_i\theta_i^{-j_i}\rho'=x_i\psi^{-j_i}\rho'=y_i.\]
Thus $\rho$ is the required conjugator.
\end{proof}

\begin{definition} \label{def:potential_image_endpoints}
Let $\psi$ and $\varphi$ be regular infinite elements in semi-normal form with respect to~$X$ and~$Y$ and let $\mathcal{X}_1, \dotsc, \mathcal{X}_m$, be
 the equivalence classes  of $\equiv$ 
defined on $X$ under the action of $\psi$.
We define $\cR_i(\psi,\phi)$ to be the set of pairs $(x,y)$, where  $x \in \mathcal{X}_i$ is of type (B)
and $y$ is an initial or terminal element of a semi-infinite $Y$-component of $\phi$ with the same characteristic as $x$.

Given a choice of elements
$(x_i,y_i)\in \cR_i(\psi,\phi)$ for each $1 \leq i \leq m$, let 
$\rho_0$ be the map from $\{x_1,\ldots ,x_m\}$ to $\{y_1,\ldots ,y_m\}$ given 
by $x_i\rho_0=y_i$ for each~$i$.
We define $\mathcal{R}(\psi, \varphi)$ to be the set of all such maps~$\rho_0$ constructed in this way.
\end{definition}

The set $\cR_i(\psi,\phi)$ is finite since the number of elements of type (B) 
in $X$ and the number of semi-infinite $Y$-components of $\varphi$ is finite, so $\mathcal{R}(\psi,\varphi)$ is also finite.

\begin{lemma}\label{otherCharEl}
Given $\rho_0\in R(\psi,\varphi)$,  there are finitely many ways of extending  $\rho_0$ to an element $\rho$ of $G_{n,r}$ 
such that $\varphi=\rho^{-1}\psi\rho$. Moreover the existence of such an extension $\rho$ can be effectively determined, 
and if such $\rho$ exists then the
images $y\rho$ can be  effectively determined, for all $y\in X$. 
\end{lemma}

\begin{proof}
Throughout the proof, when we say $\rho$ exists we mean that an extension $\rho$ of $\rho_0$ to an element of $G_{n,r}$ 
exists and satisfies $\varphi=\rho^{-1}\psi\rho$. 
From Lemma \ref{lem:effsplit}, 
we may effectively 
construct the equivalence classes $\cX_i$, and so also the sets $R_i(\psi, \varphi)$. 
First consider a single equivalence class $\mathcal{X}_i$. 
We are given a representative element $x_i \in \mathcal{X}_i$ of type (B) and an element $y_i$ 
such that $x_i\rho_0=y_i$,
where $y_i$ is an initial or terminal element of a semi-infinite $Y$-component of $\varphi$ with the same characteristic as~$x_i$. 

Let $x\in X$ of type (B). 
Then,  by definition of $\equiv$,  we have $x\in \cX_i$ if and only if 
 there exist elements
$x_i=u_0,\ldots ,u_t=x$ of $X$,  elements  
 $\G_j, \D_j\in A^*$ and $k_j\in \ZZ$ with 
$u_{j+1}\D_{j+1}=u_{j}\G_j\psi^{k_j}$, for $j=0,\ldots ,t-1$. 
Before going any further, we show that we may assume that $u_j$ is of type (B),
for all $j$. Suppose not, say $u_j$ is of type (C). Then, by Lemma \ref{ABC}, 
there exist $k'_j\in\ZZ$, $\Gamma'_j\in A^*$ and $u'_j\in X$ of
 type (B) such that $u_j\psi^{k'_j}=u'_j\G'_j$. 
Now 
\[u_{j-1}\G_{j-1}\psi^{k_{j-1}+k'_j}=u_j\D_j\psi^{k'_j}=u'_j\G'_j\D_j\]
and
\[u'_j\G'_j\G_j\psi^{k_j-k'_j}=u'_j\G'_j\psi^{-k'_j}\G_j\psi^{k_j}
=u_j\G_j\psi^{k_j}=u_{j+1}\D_{j+1},\]
so we may replace $u_j$ by $u'_j$. Continuing this way, eventually all $u_j$ will be
of type (B).

We show, by induction on $t$, that there are finitely many possible values of $x\rho$, for a conjugator $\rho\in G_{n,r}$
such that 
$x\varphi=x\rho^{-1}\psi\rho$ which extends~$\rho_0$. (That is, where $x_i\rho=x_i\rho_0=y_i$.)
We also describe an effective procedure to enumerate
the set of all such elements.  
 Suppose first that $t=1$, so $x=u_1$ and we have $\G=\G_0$, $\D=\D_1$ and $k=k_0$ 
such that $x_i\G\psi^k=x\D$. Given that $\rho$ exists, from Lemma \ref{conjugateMultipliers}, 
$x\rho$ belongs to a semi-infinite $Y$-component $\cC$ of $\varphi$ with the same characteristic
as $x$.  Therefore (if $\rho$ exists) there exists an element $(x,w)\in R_i(\psi,\varphi)$ such that 
$w$ is the initial or terminal element of $\cC$, as well as an integer $l$
such that $w\varphi^l=x\rho$. This implies that 
\[w\Delta\varphi^l=(x\Delta)\rho=x_i\G\psi^k\rho=x_i\G\rho\varphi^k=
x_i\rho_0\varphi^k\G,\]
so
\begin{equation}\label{eq:wexist}
w\Delta\varphi^{l-k}=x_i\rho_0\G=y_i\G.
\end{equation}

Lemma \ref{9.7H} gives an  effective procedure to  determine whether  an 
integer $l$ satisfying \eqref{eq:wexist} exists, and if so find it.
Given $\rho_0$ and $x$, the integer $k$ and the elements $\G$ and $\D$ are uniquely determined
so,  to decide whether an appropriate value $x\rho$ exists, we may check each pair $(x,w)$ in
the  set $R_i(\psi,\varphi)$ to see if \eqref{eq:wexist} holds for some $l$ or not. 
For each such  $w$ there is at most one $l$ such that \eqref{eq:wexist} has a solution
 and, as  $R_i(\psi,\varphi)$ is finite, we may effectively enumerate the  values $w\D\varphi^{l-k}$ that could be assigned
to  $x\rho$. Hence the result holds if $t=1$.  

Now assume
that $t>1$ and the result holds for all $x$ related to $x_i$ by a chain of length
at most $t-1$. Then $u_{t-1}$ is of type (B) and by assumption $u_{t-1}\rho$ may be given
 one of finitely many values, and we have a procedure to enumerate these values.
  Suppose then that $u_{t-1}\rho=v$. Now 
$x=u_m$ and we have $\G_{t-1}, \D_t \in A^*$ and $k_{t-1}\in \ZZ$ such that
$u_{t-1}\G_{t-1}\psi^{k_{t-1}}=x\D_t$. Applying the argument of the case $m=1$ 
with $u_{t-1}$, $\G_{t-1}$, $\D_t$ and $v$ in place of $x_i$, $\G$, $\D$ and $y$, we
see that a finite set of possible values for $x\rho$ may be effectively determined. 
Therefore, by 
induction, the result holds for all $x\in \cX_i$ of 
type (B).

Finally, if $x\in \cX_i$ is of type (C), then by Lemma \ref{ABC} 
there is a $z\Sigma$ in the $X$-component of $x$, for some $z$ of type (B) and 
$\Sigma\in A^*$, \emph{i.e.} $x\psi^p=z\Sigma$ for some integer $p$.
Since we have already determined the possible images of all the type (B) elements in 
$\mathcal{X}_i$, if $\rho$ exists we have, for each choice of $z\rho$, 
\[x\rho=z\Sigma\psi^{-p}\rho=z\rho\Sigma\varphi^{-p}\]
 and this determines the image of the type (C) element under $\rho$ 
(uniquely once we have made our initial choice for the image of $z\rho$).

We carry out this process on each equivalence class in turn.
If the process results in  
at least one  possible value for each element of $X$,
we obtain a potential extension $\rho$ of $\rho_0$.
For such a $\rho$ to be a genuine extension, we need to check if $\rho$ defines an automorphism of $V_{n,r}$.
This is the case if and only if the image $X\rho$ of the $A$-basis~$X$ is itself a basis for~$V_{n,r}$,
which we can effectively determine using Lemma~\ref{HigmanLemma2.5}.
(Note that $X\rho$ need not be an $A$-basis---see Example~\ref{ex:infinite_conjugacy_test} below.)
\end{proof}

We are now able to state the main result of this section. 

\begin{proposition}\label{conjInfinite}
Let $\psi$ and $\varphi$ be regular infinite elements of $G_{n,r}$ in quasi-normal form with respect to $X$ and $Y$ respectively. 
Then $\psi$ is conjugate to $\varphi$ if and only if there exists a map $\rho_0\in \mathcal{R}(\psi, \varphi)$ which extends to an element $\rho$ of $G_{n,r}$ with $\rho^{-1}\psi\rho=\varphi$.
\end{proposition}

\begin{proof}
If $\rho_0$ extends to an element $\rho \in G_{n,r}$ with $\rho^{-1}\psi\rho=\varphi$, then $\psi$ is certainly conjugate to $\varphi$.

Assume that $\psi$ is conjugate to $\varphi$. Lemma \ref{firstCharEl} tells us that there exists a conjugator $\rho$ such that, for 
each equivalence class $\mathcal{X}_i$, there exists an element $x_i$ of type (B) 
in $\mathcal{X}_i$ with $y_i= x_i\rho$  an initial or terminal element of a semi-infinite $Y$-component of $\varphi$. 
We define $\rho_0$ to be the map $x_1\mapsto y_1,\ldots ,x_m\mapsto y_m$, where $y_i=x_i\rho$ for each $i=1,\ldots ,m$. Thus, $\rho_0$ is an element of the finite set $\mathcal{R}(\psi;\varphi)$. 
Now $\rho_0$  is the restriction of $\rho$ to $\{x_1,\ldots ,x_m\}$, 
so it certainly extends to $\rho$, as required.
\end{proof}

\begin{example} \label{ex:infinite_conjugacy_test}
Let $n=2$, $r=1$ and $V_{2,1}$ be free on $\cxx = \{x\}$. Let 
\[
Y=\{x\alpha_1,x\alpha_2\alpha_1^2,x\alpha_2\alpha_1\alpha_2,x\alpha_2^2\}
\quad\text{and}\quad
Z=\{x\alpha_1^3,x\alpha_1^2\alpha_2,x\alpha_1\alpha_2,x\alpha_2\}
\]
determine the automorphism~$\psi$ as illustrated below.
\begin{center}
$\psi:$
\begin{minipage}[b]{0.4\linewidth}
\centering
 \Tree [. 1 [  [. [.2  ] [.3  ] ]. [.4  ] ].  ]    \quad $\longrightarrow$\Tree [ [ [ [.1  ] [.3  ] ] [.4  ] ]   [.2  ] ] 
\end{minipage}
\end{center}
Then $Y$ is the minimal expansion of $\cxx$ associated to $\psi$ and $Z=Y\psi$.
The only element of $\cxx\mA$ not in $Y\mA\cup Z\mA$ is $x$, so we take $X=\{x\a_1,x\a_2\}$ to be our candidate basis for a quasi-normal form. 
Then $X\mA\setminus Y\mA=\{x\a_2,x\a_2\a_1\}$ and $X\mA\setminus Z\mA=\{x\a_1,x\a_1^2\}$. 
The $X$-components of the first two elements are
\begin{align*}
	x\alpha_2 &\in \{x\alpha_2\alpha_1^{2k}\}_{k \geq 0}
	&
	x\alpha_2\alpha_1 &\in \{x\alpha_2\alpha_1^{2k+1}\}_{k \geq 0},
\intertext{both left semi-infinite with characteristic $(-1,\alpha_1^2)$.
The latter two elements' $X$-components are}
	x\alpha_1 &\in \{x\alpha_1^{2k+1}\}_{k \geq 0}
	&
	x\alpha_1^2 &\in \{x\alpha_1^{2k+2}\}_{k \geq 0},
\end{align*}
both right semi-infinite with characteristic $(1,\alpha_1^2)$.
Hence $\psi$ is in quasi-normal form with respect to $X$, both elements of $X$ are of type (B) 
 and $\cM_\psi=\{(1,\alpha_1^2), (-1,\alpha_1^2)\}$. 
As $(x\a_2)\a_2\psi=(x\a_1)\a_2$ there is one equivalence class of $\equiv$, that is $\cX_1=X$. 

Let $\varphi$ be automorphism of Examples~\ref{ex:lio1} and~\ref{ex:lio2}.  
Then $\phi$ is in quasi-normal form with respect to the $A$-basis $X_\phi=\{x\a_1,x\a_2\a_1,x\a_2^2\}$ and
$\cM_\phi=\cM_\psi$. The initial elements of right semi-infinite $X_\phi$-components are $x\a_1$ and $x\a_1^2$ and the 
terminal elements of left semi-infinite $X_\phi$-components are $x\a_2^2$ and $x\a_2^2\a_1$.

The set $\cR_1(\psi,\varphi)$ consists of the pairs $(x\a_1,x\a_1),(x\a_1,x\a_1^2),(x\a_2,x\a_2^2)$ and $(x\a_2,x\a_2^2\a_1)$.
Let us choose $x\alpha_1$ as our type (B) representative in $\cX_1$. 
We have two choices for the image of $x\a_1$ under $\rho_0$, corresponding to the two pairs $(x\a_1,x\a_1),(x\a_1,x\a_1^2)\in \cR_1$. 
 Denote these by $\rho_1$ and $\rho_2$, where  
\[x\alpha_1\rho_1=x\alpha_1 \textrm{ and }  x\alpha_1\rho_2=x\alpha_1^2.\]
Next we 
 determine the images of the other type (B) element $x\a_2$ of $X$ under the action of $\rho_1$ and $\rho_2$,
 following the proof of Lemma \ref{otherCharEl}.
 
As noted above, $x\a_1\equiv x\a_2$ because $(x\a_1)\a_2\psi^{-1}=(x\a_2)\a_2$, so in
the notation of the proof  of Lemma \ref{otherCharEl} we have $\G=\a_2$, $\D=\a_2$ and $k=-1$. Substituting these values into 
equation \eqref{eq:wexist}, we wish to find $l$ such that 
\[w\a_2\phi^{l+1}=(x\a_1)\rho_i\a_2,\]
where $i=1$ or $2$, and $w=x\a_2^2$ or $x\a_2^2\a_1$. 
Whenever we find such  an $l$ then we set $x\a_2\rho_i=w\phi^l$ and check to see if $\rho_i$ determines an automorphism.
If so, we  check if $\rho_i$ is a conjugator i.e.\ if $\rho_i^{-1}\psi\rho=\phi$.
\begin{description}
 \item[Case $\mathbf{i=1}$: ${x\a_1\rho_1=x\a_1}$.] ~
  \begin{enumerate}[(i)]
   \item When $w=x\alpha_2^2$ we have
   \begin{equation*}
   x\a_2^3\varphi^{l+1} =x\a_1\a_2
   \quad \iff \quad 
   x\a_2\a_1\varphi^{l} =x\a_1\a_2,
   \end{equation*}
   which has no solutions, as may be verified using the process of Lemma \ref{9.7H}. 
     \item When $w=x\alpha_2^2\alpha_1$ we have
   \begin{equation*}
   x\a_2^2\a_1\a_2\varphi^{l+1} =x\a_1\a_2
   \quad\iff\quad
   x\a_1\a_2\varphi^{l} =x\a_1\a_2,
   \end{equation*}
   which has solution $l=0$.
   Therefore we set $x\a_2\rho_1=x\a_2^2\a_1.$
   Now $\rho_1$ now maps $X$ to $\{x\a_1,x\a_2^2\a_1\}$, which is 
   not a basis of $V_{2,1}$ (see Lemma \ref{HigmanLemma2.5}). 
   So the set map~$\rho_1$ extends to an endomorphism which is not an automorphism of $V_{2,1}$.
  \end{enumerate}
  Neither value of $w$ results in a potential conjugator~$\rho_1$.
  
 \item[Case $\mathbf{i=2}$: ${x\a_1\rho_2=x\a_1^2}$.] ~
   \begin{enumerate}[(i)]
     \item When $w=x\alpha_2^2$ we have
       \begin{equation*}
         x\a_2^3\varphi^{l+1} =x\a_1^2\a_2
         \quad\iff\quad
         x\a_2\a_1\varphi^{l} =x\a_1^2\a_2
       \end{equation*}
       which has solution $l=1$. Therefore we set 
       \begin{align}\label{eq:findc}
       \begin{split}
         x\a_2\rho_2
         &=   x\alpha_2^2\phi\\
         &= ( x\a_2^2\a_1^2 )( x\a_2^2\a_1\a_2 ) \lambda  (x\a_2^3)  \lambda \phi \\
         &= ( x\a_2^2\a_1^2 \phi )( x\a_2^2\a_1\a_2 \phi ) \lambda (x\a_2^3 \phi)  \lambda  \\
         &= ( x\a_2^2)( x\a_1\a_2 )      \lambda (x\a_2\a_1) \lambda
       \end{split}
       \end{align}
       In this case $x\alpha_2^2$ is in $X_\phi\mA\setminus W\mA$, where $W$ is the minimal expansion associated to  $\phi$;
       this is why the standard form of $x\a_2\rho_2$ is written using contraction operations~$\lambda$.
       
       To define $\rho_2$ in terms of $X\mA$, we must take an expansion of $X$ at $x\a_2$.
       We take the minimal expansion which allows us to define the map into $\cxx\mA$;
       namely $\{x\a_2\a_1^2, x\a_2\a_1\a_2, x\a_2^2\}$. 
       From \eqref{eq:findc} we obtain
       \begin{align*}
         x\a_2\a_1^2\rho_2&=(x\alpha_2^2)\a_1^2\phi=x\a_2^2\\
         x\a_2\a_1\a_2\rho_2&=(x\alpha_2^2)\a_1\a_2\phi=x\a_1\a_2 \\
         x\a_2^2\rho_2&=(x\alpha_2^2)a_2\phi=x\a_2\a_1.
       \end{align*}
       We see that $\rho_2$ maps the expansion $\{x\a_1, x\a_2\a_1^2, x\a_2\a_1\a_2, x\a_2^2\}$ of $X$ to 
       $\{x\a_1^2, x\a_2^2,x\a_1\a_2,x\a_2\a_1\}$ which is a basis for $V_{2,1}$; so $\rho_2$ determines an element of $G_{2,1}$. 
       It can be verified $\rho_2^{-1}\psi\rho_2=\phi$, so $\rho_2$ is a conjugator. At this point we could stop
       but we give the final case for completeness.
       
     \item When $w=x\alpha_2^2\alpha_1$ we have 
       \begin{equation*}
         x\a_2^2\a_1^3\a_2\psi^{l+1} =x\a_1^2\a_2
         \quad \iff \quad
         x\a_1\a_2 =x\a_1^2\a_2,
       \end{equation*}
       which has no solutions.
   \end{enumerate}
\end{description}
We find one conjugating element $\rho_2$ and we see that $\psi$ and $\phi$ are conjugate via~$\rho_2$.
\end{example}
The algorithm for the conjugacy of regular infinite  elements of $G_{n,r}$ is as  follows. 
\begin{algorithm}\label{alg:reginf}
Let $\psi$ and $\phi$ be regular infinite elements of $G_{n,r}$.
 \be[\bfseries Step 1:]
\item Construct $A$-bases $X_\psi$ and $X_\phi$ with respect to which $\psi$ and $\phi$ are in quasi-normal form (Lemma \ref{lem:qnf}). 
\item Construct   the equivalence classes $\cX_i$, $i=1,\ldots, m$, of $\equiv$ on  $X_\psi$ (Lemma \ref{lem:effsplit}). 
\item Find the initial and terminal elements of semi-infinite $X_\phi$-components of $\phi$, by finding the minimal expansion of $X_\phi$ associated to $\phi$ (Lemma \ref{9.9H}).
\item Construct the sets $\cR_i(\psi, \phi)$. 
\item For each equivalence class $\cX_i$ of $\equiv$ on $X_\psi$ choose an element $x_i\in \cX_i$, of type (B). 
\item For each $i$ and each pair $(x_i,y)$ of $\cR_i(\psi,\phi)$, construct a map $\rho_i:\cX_i\mapsto X_\phi$, using equation \eqref{eq:wexist}, as in the
proof of Lemma \ref{otherCharEl}, if possible. In each case check that $\rho_i$ is an automorphism.
\item For each $m$ tuple $\rho_1,\ldots,\rho_m$ of automorphisms, from the previous step, check whether the map $\rho=\rho_1\ast\cdots \ast \rho_m $ conjugates
$\psi$ to $\phi$. 
\ee
\end{algorithm}

\section{The power conjugacy problem}\label{PowerConj}
For a group with presentation $\langle X\,|\,R\rangle$, the \emph{power conjugacy problem} is to determine, 
 given words $g, h\in \FF(X)$ whether or not there exist non-zero integers $a$ and $b$ such that $g^a$ 
is conjugate to $h^b$ in $G$. 
We may in addition require  that, if the answer to this question is ``yes'', then integers 
$a$ and $b$, and an element $c\in \FF(X)$, are found, such that
$c^{-1}g^ac=_Gh^b$.  We say the power conjugacy problem is \emph{decidable} if there is an algorithm which, 
given $g$ and $h$ outputs ``yes'' if they're conjugate and ``no'' otherwise. Again, the stronger form entails the 
obvious extra requirements. As before, in $G_{n,r}$ we work entirely with symbols for automorphisms, ignoring
the presentation. 

As in the case of  the conjugacy problem, we break the power conjugacy problem down into two cases; 
one for periodic elements and one for regular infinite elements. 
Then, we  construct an algorithm that  combines the two 
parts.

\subsection{The power conjugacy for periodic elements}\label{torPower}
Let $\psi$ and $\varphi$ be periodic elements of $G_{n,r}$, of order $k$ and $m$ respectively, in quasi-normal form with respect to the $A$-bases $X$ and $Y$. 
To test whether there exist $a$, $b\in \ZZ$ such that $\psi^{a}$ is conjugate to $\varphi^{b}$, we apply Proposition 
\ref{conjTorsion} to the pair $\psi^{c}$, 
$\varphi^{d}$, for all $c\in\{1,\ldots ,k\}$ and all $d\in\{1,\ldots ,m\}$. 

\subsection{Regular infinite elements}\label{regPower}

The first step is to compare the sets $\cM_\psi$ and $\cM_{\psi^a}$, $a\in \ZZ$, $|a|>1$, 
for a regular infinite automorphism $\psi$.

\begin{lemma}\label{AJDLEMMACMa} 
Let $\psi$ be a regular infinite element of $G_{n,r}$ and let $a$ be a non-negative integer. 
Then 
\begin{equation}\label{eq:pcM}
\cM_{\psi^a}=\{(m/d,\G^q)\,|\,(m,\G)\in \cM_\psi,\; \gcd(m,a)=d \textrm{ and } |a|=qd\}.
\end{equation}
\end{lemma}
\begin{proof}
Let $\psi$ be in semi-normal form with respect to $X$.
The $X$-components  of $\psi^a$ are sub-sequences of the $X$-components  of $\psi$, so 
 $\psi^a$ is also in semi-normal form with respect to $X$.
  Suppose to begin with that $a>0$.
 First we show that the right hand side of \eqref{eq:pcM} is  contained in the left hand side. 
If $(m,\G)\in \cM_\psi$ then there exists an element $u$ of $V_{n,r}$ in a semi-infinite $X$-component for $\psi$ of characteristic $(m,\G)$;
and we may assume $u\in X\mA$.  If $d=\gcd(m,a)$,  $p=m/d$, $q=a/d$ and $k=ma/d$, 
then $u(\psi^a)^p=u\psi^{mq}=u\G^q$, (as $mq$ has the 
same sign as $m$). If $a<0$ then, from the above, with $d=\gcd(m,-a)$, $p=m/d$,  $q=-a/d$ and $k=-ma/d$, 
we have $u\psi^{-ap}=u\G^{q}$.  In all cases therefore $u$ is a characteristic element of $\psi^{a}$. 
Furthermore,  if $u(\psi^{a})^r=u\D$, with $\D\neq 1$ then, from Lemma \ref{AJDLEMMACM}, $m|ar$, which we can rewrite
as $pd|qdr$, so $p|qr$. As $\gcd(p,q)=1$, this implies $p|r$, so that $|m/d|=|p|\le |r|$. Hence $u$ has 
characteristic  $(m/d,\G^q)$, with respect to $\psi^a$. As $u$ belongs to a semi-infinite $X$-component for $\psi^a$, it
follows that 
 $(m/d,\G^q)$ is in $\cM_{\psi^a}$ and so we have 
\[\cM_{\psi^a}\supseteq\{(m,\G^q)\,|\,(md,\G)\in \cM_\psi, d>0, \gcd(m,q)=1 \textrm{ and } |a|=qd\}.\]

On the other hand, suppose that $(r,\D)\in  \cM_{\psi^a}$. 
 Then again, there exists 
$u\in X\mA$ such that 
$u$ is a characteristic element of $\psi^a$, so $u\psi^{ar}=u\D$. Thus, from Lemma \ref{AJDLEMMACM}, 
$u$ is a characteristic element for 
$\psi$, with characteristic $(m, \G)\in \cM_\psi$, such that 
$m|ar$ and $\D=\G^t$, where $ar=mt$, $t>0$. Let $d=\gcd(a,m)$, $m=pd$ and  $a=qd$. Then 
$dqr=pdt$, so $qr=pt$ and $\gcd(p,q)=1$, so $r=pr'$ and $t=qt'$, for some $r',t'$. However,
we have $u(\psi^a)^p=u\psi^{dpq}=u\psi^{mq}=u\G^q$, and so, by definition of $(r,\D)\in \cM_{\psi^a}$, 
 we see that $|p|\ge |r|$, so $r'=\pm 1$. Since $a>0$, both $m$ and $r$ have the same sign, so $r'=1$. It now follows that $r=p=m/d$ and $\D=\G^q$, so  
$(r,\D)$ belongs 
to the set on the right hand side of \eqref{eq:pcM}. That is 
\[\cM_{\psi^a}\subseteq\{(m,\G^q)\,|\,(md,\G)\in \cM_\psi, d>0, \gcd(m,q)=1 \textrm{ and } |a|=qd\}.\]

If $a<0$ then the lemma follows by applying the result above to $\mathcal{M}_{\psi^{{-1}(-a)}}$, as for all $\theta\in G_{n,r}$ we have $(m,\G)\in \mathcal{M}_\theta$ if and only if $(-m,\G)\in \mathcal{M}_{\theta^{-1}}$.
\end{proof}

\begin{example}\label{pcExample1}
Let $n=2$ and $r=1$ and let  $V_{2,1}$ be free on $\cxx = \{x\}$. 
Let $\phi$ be the regular infinite element of $G_{2,1}$ defined by the bijection from
\[Y=\{x\a_1^3,x\a_1^2\a_2,x\a_1\a_2,x\a_2\},\]
to 
\[Z=\{x\a_1^2,x\a_1\a_2,x\a_2\a_1,x\a_2^2\},\]
given by  the following tree pair diagram. 
\begin{center}
$\phi:$
\begin{minipage}[b]{0.4\linewidth}
\centering
\Tree [ [ [ [.1  ] [.2  ] ] [.3  ] ]   [.4  ] ]   \quad $\longrightarrow$ \Tree [ [ [.3  ] [.1  ] ] [2 4 ]  ] 
\end{minipage}
\end{center}
Then $Y$ is the minimal expansion of $\{x\}$ associated to $\phi$. The minimal expansion of $\{x\}$ 
contained in $Y\mA\cup Z\mA$ is $X=\{x\a_1^2,x\a_1\a_2,x\a_2\}$. 
$X\mA\setminus Y\mA=\{x\a_1^2\}$ and $X\mA\setminus Z\mA=\{x\a_2\}$. The $X$-components of these elements are 
\[
\cdots \mapsto  x\a_1\a_2\a_1\mapsto x\a_1^3\mapsto x\a_1\a_2\mapsto x\a_1^2
\]
with characteristic $(-2,\a_1)$ and 
\[
x\a_2\mapsto x\a_2^2\mapsto x\a_2^3\mapsto x\a_2^4 \mapsto \cdots
\]
with characteristic $(1,\a_2)$. Hence $\phi$ is in quasi-normal form with respect to $X$ and 
 $\mathcal{M}_{\phi}=\{(\text{-}2,\a_1),(1,\a_2)\}$.

The map $\phi^2$ may be defined by the bijection from
\[U=\{x\a_1^3,x\a_1^2\a_2,x\a_1\a_2\a_1,x\a_1\a_2^2,x\a_2\}\]
to
\[V=\{x\a_1^2,x\a_1\a_2,x\a_2\a_1,x\a_2^2\a_1,x\a_2^3\}\]
given by a different tree pair diagram.
\begin{center}
$\phi^2:$
\begin{minipage}[b]{0.5\linewidth}
\centering
\Tree [ [ [ [.1  ] [.2  ] ] [ 3 4  ] ]   . 5  ]   \quad $\longrightarrow$ \Tree [ [ [.1  ] [.3 ] ]  [ 4  [ 2 5 ] ] ]
\end{minipage}
\end{center}
Then $U$ is the minimal expansion of $\{x\}$ associated to $\phi^2$ and the minimal expansion of $\{x\}$ 
contained in  $U\mA\cup V\mA$ is $X$ again. 
$X\mA\setminus U\mA=\{x\a_1^2,x\a_1\a_2\}$ and  $X\mA\setminus V\mA=\{x\a_2,x\a_2^2\}$; the corresponding $X$-components are
\begin{align*}
\cdots &\mapsto x\a_1^3\mapsto x\a_1^2
& 
\cdots &\mapsto x\a_1\a_2\a_1\mapsto x\a_1\a_2
\end{align*}
with characteristic $(-1,\a_1)$ and 
\begin{align*}
x\a_2 &\mapsto x\a_2^3\mapsto \cdots
& 
x\a_2^2 &\mapsto x\a_2^4\mapsto \cdots
\end{align*}
with characteristic $(1,\a_2^2)$.
Hence $\phi^2$ is in quasi-normal form with respect to $X$ and 
 $\mathcal{M}_{\phi^2}=\{(\text{-}1,\a_1),(1,\a_2^2)\}$, as asserted by Lemma \ref{AJDLEMMACMa}. 
\end{example}

Lemma \ref{AJDLEMMAGCD} and Proposition \ref{divisors} will allow us to find ``minimal'' pairs $(a,b)$ such that $\psi^a$ and $\varphi^b$ are conjugate. 

\begin{lemma}\label{AJDLEMMAGCD} Let $\psi$ and $\varphi$ be regular infinite elements of $G_{n,r}$ and let $c$ be an integer, such that
$c$ is coprime to $m$, for all $m\in \ZZ$ such that $(m,\G)\in \cM_\psi\cup \cM_\varphi$. Then $\psi^c\sim \varphi^c$ 
if and only if $\psi\sim \varphi$. 
\end{lemma}
\begin{proof}
If $\psi\sim \varphi$ then it is immediate that $\psi^c\sim \varphi^c$. 
For the converse, let $\rho\in G_{n,r}$ be such that  $\varphi^c=\rho^{-1}\psi^c\rho$ and 
observe that 
we may assume, without loss of generality, that $c>0$. Suppose that $\psi$ and 
$\varphi$ are in quasi-normal form with respect to $A$-bases $X$ and $Y$, respectively. From Lemma \ref{AJDLEMMACMa}, 
$\cM_{\psi^c}=\{(m,\G^c)|(m,\G)\in \cM_\psi\}$ and $\cM_{\varphi^c}=\{(m,\D^c)|(m,\D)\in \cM_\varphi\}$. 

Let $u$ be an element of $V_{n,r}$ which is characteristic for $\psi$, with $\psi$-characteristic $(m,\G)$.
Then, from Lemma \ref{AJDLEMMACMa} (and its proof), $u$ has $\psi^c$-characteristic $(m,\G^c)$ and, as  $\varphi^c=\rho^{-1}\psi^c\rho$, 
its image
$u\rho$ has   $\varphi^c$-characteristic $(m,\G^c)$. Hence, from Lemma \ref{AJDLEMMACMa} again,  $u\rho$ has   
$\varphi$-characteristic $(m,\G)$. 
As $\gcd(c,m)=1$, there exist integers $s$ and $t$ such that $ms+ct=1$. Since $\psi^c\rho=\rho\varphi^c$ we have, in 
the case where $s>0$, 
\begin{align*}
u\psi\rho& 
=u\psi^{ms+ct}\rho=(u(\psi^m)^s)\psi^{ct}\rho=u\G^s\psi^{ct}\rho=u\G^s\rho\varphi^{ct}\\
&=(u\rho)\G^s\varphi^{ct}=(u\rho)\varphi^{ms}\varphi^{ct}=(u\rho)\varphi^{ms+ct}\\
&=u\rho\varphi.
\end{align*}
If $s<0$ then we have $m(-s)+c(-t)=-1$, with $-s>0$ and the argument  above implies instead that $u\psi^{-1}\rho
=u\rho\varphi^{-1}$. In this case, let $v=u\psi$, so $v$ also has $\psi$-characteristic $(m,\G)$ and, applying   the argument 
above to $v$ instead of 
$u$, consequently $v\psi^{-1}\rho=v\rho\varphi^{-1}$, from which it follows that  $u\psi\rho=u\rho\varphi$. 
This applies in particular
 to all elements of $X$ of type (B), with respect to $\psi$. 

Let $y'$ be an element of type (C), with respect to $\psi$;
 so there exists an integer $k$ and an
element $y\in X$ of type (B) such that $y'\psi^k=y\Om$.  
Then $y'=y\Om\psi^{-k}$, and $y\psi^{j}$ has the same $\psi$-characteristic as $y$, for all $j$: 
and so is a characteristic element for $\psi$. From the above then $y\psi^{j}\rho=(y\rho)\varphi^{j}$, for all $j$. 
Now 
\begin{align*}
y'\psi \rho&=y\Om\psi^{1-k}\rho= y\psi^{1-k}\rho\Om=y\rho\varphi^{1-k}\Om=y\rho\varphi^{-k}\varphi\Om\\
&=y\psi^{-k}\rho\varphi\Om=y\psi^{-k}\Om\rho\varphi=y'\rho\varphi.
\end{align*}
Therefore, $y\psi\rho=y\rho\varphi$, for all $y\in X$, so $\psi\sim \varphi$. 
\end{proof}

\begin{definition}\label{AJDDEFINITIONMa} Let $\psi$ be a regular infinite element of  $G_{n,r}$ and let $a$ be a positive integer. 
Define a map $\hsa:\Ms\maps \Msa$ by 
$\hsa(m,\G)=(p,\G^\a)$, where  $d=\gcd(m,a)$, $p=m/d$ and $\a=a/d$. 
\end{definition}

\begin{example}
For $\phi$ in Example \ref{pcExample1},  with $a=2$,  the map
$\widehat{\phi}^2:\Mh\maps \mathcal{M}_{\phi^2}$
is given by
\[\widehat{\phi}^2(\text{-}2,\a_1)=(\text{-}1,\a_1) \quad\text{and}\quad \widehat{\phi}^2(1,\a_2)=(1,\a_2^2).\]
\end{example}

From Lemma \ref{AJDLEMMACMa} this is a well defined map, and is surjective. In general
it is not injective. For instance if  $p,s$ and $t$ are pairwise coprime 
positive integers 
and we have $m_1=ps$, $m_2=pt$ and $a=st$, then $d_1=\gcd(m_1,a)=s$ and  
$d_2=\gcd(m_2,a)=t$. 
If, for some 
non-trivial $\L\in A^*$, we have $(m_1,\L^s)$ and $(m_2,\L^t)$ in $\Ms$ then
both these elements are mapped by $\hsa$ to $(p,\L^{st})$.

\begin{proposition}\label{divisors} Let $\psi$ and $\varphi$ be regular infinite elements of $G_{n,r}$, let $a$ and $b$ 
be positive integers and 
let the images of $\hsa$ and $\hhb$ be 
\[\Msa=\{(p_i,\G_i^{\a_i}) \mid i=1,\ldots, M\}\textrm{ and }
\Mhb=\{(q_i,\D_i^{\b_i}) \mid i=1,\ldots, N\}.\]
For $i=1,\ldots, M$, let 
\[(\hsa)^{-1}(p_i,\G_i^{\a_i})=\{(m_{i,j},\G_{i,j}) \mid 1\le j\le M_i\}\]
and, for $i=1,\ldots, N$, let 
\[(\hhb)^{-1}(q_i,\D_i^{\b_i})=\{(n_{i,j},\D_{i,j}) \mid 1\le j\le N_i\}.\]
If $\psi^a\sim \varphi^b$ then 
$M=N$ and, after reordering if necessary,  we have $p_i=q_i$ and $\G_i^{\a_i}=\D_i^{\b_i}$.
Moreover,  there exist positive integers
$\a,\b, g, d_{i,j}, e_{i,k}, s_{i,j,k}, t_{i,j,k},f_{i,j,k}$,  and $\L_{i,j,k}\in A^*$,  
 for $1\le i \le M$, $1\le j \le M_i$ and $1\le k\le N_i$, such that 
\[\a=\frac{a}{g}=d_{i,j}f_{i,j,k}t_{i,j,k} \textrm{ and }  \b=\frac{b}{g}=e_{i,k}f_{i,j,k}s_{i,j,k}, \textrm{ for all } i,j,k, \]
and 
\[\psi^\a\sim \varphi^\b,\]
where $d_{i,j}$ is a positive divisor of $m_{i,j}$, $e_{i,k}$ is a positive 
divisor of $n_{i,k}$, $\G_{i,j}=\L_{i,j,k}^{s_{i,j,k}}$ and $\D_{i,j}=\L_{i,j,k}^{t_{i,j,k}}$, 
and 
\[f_{i',j',k'} \,\left|\, \middle(\prod_{i,j,k}(t_{i,j,k}d_{i,j})\right)/t_{i',j',k'}d_{i',j'},\]
for all $i',j',k'$.  
\end{proposition}
\begin{proof}
Assume $\psi^a\sim \varphi^b$, with $a,b>0$, and that 
$\rho^{-1}\psi^a\rho=\varphi^b$. From Lemma \ref{conjugateMultipliers}, $\Msa$ and $\Mhb$ are equal, so $M=N$, and 
we may order $\Msa$ so that $(p_i,\G_i^{\a_i})=(q_i,\D_i^{\b_i})$, so $p_i=q_i$ and $\G_i^{\a_i}=\D_i^{\b_i}$.
With the notation for $(\hsa)^{-1}(p_i,\G_i^{\a_i})$ and $(\hhb)^{-1}(q_i,\D_i^{\a_i})$ given in the
statement of the proposition, let 
$d_{i,j}=\gcd(a,m_{i,j})$ and  $e_{i,k}=\gcd(b,n_{i,k})$,
so 
\[m_{i,j}/d_{i,j}=p_i=q_i= n_{i,k}/e_{i,k}\] 
and let 
\[\a_{i,j}=a/d_{i,j},\; \b_{i,k}=b/e_{i,k},\]
and 
\begin{equation}\label{eq:div1}
\G_{i,j}^{\a_{i,j}}=\G_i^{\a_i}=\D_i^{\b_i}=\D_{i,k}^{\b_{i,k}},
\end{equation}
by Definition \ref{AJDDEFINITIONMa}, for $1\le i\le M$, $1\le j\le M_i$ and $1\le k\le N_i$.

As $\G_{i,j}^{\a_{i,j}}=\D_{i,k}^{\b_{i,k}}$, by Proposition \ref{Lothaire1}, there exist
$\L_{i,j,k}\in A^*$ and positive integers $s_{i,j,k}, t_{i,j,k}$ such that $\G_{i,j}=\L_{i,j,k}^{s_{i,j,k}}$ 
and $\D_{i,j}=\L_{i,j,k}^{t_{i,j,k}}$. Taking a power of $\L_{i,j,k}$ if necessary, we may assume that 
$\gcd(s_{i,j,k}, t_{i,j,k})=1$. Then 
\begin{equation}\label{eq:div2}
\L_{i,j,k}^{s_{i,j,k}\a_{i,j}} =\G_{i,j}^{\a_{i,j}}=\D_{i,k}^{\b_{i,k}}=\L_{i,j,k}^{t_{i,j,k}\b_{i,k}},
\end{equation}
so $s_{i,j,k}\a_{i,j}=t_{i,j,k}\b_{i,k}$. As $s_{i,j,k}$ and $t_{i,j,k}$ are coprime this implies that
$\a_{i,j}/t_{i,j,k}=\b_{i,k}/s_{i,j,k}=c_{i,j,k}\in \ZZ$, and 
$\a_{i,j}=c_{i,j,k}t_{i,j,k}$ and $\b_{i,k}=c_{i,j,k}s_{i,j,k}$.

Let \[g=\gcd(\{c_{i,j,k}|1\le i \le M,1\le j\le M_i, 1\le k \le N_i\}).\]
Then there exist integers $f_{i,j,k}$ such that $c_{i,j,k}=gf_{i,j,k}$, for all $i,j,k$. 
From Lemma \ref{AJDLEMMACMa}, $\cM_{\psi^{a/g}}$ consists of elements $(m/p,\G^\a)$, where $(m,\G)\in \cM_\psi$, $p=\gcd(m,a/g)$ and 
$\a=a/gp$. Similarly,  elements of $\cM_{\varphi^{b/g}}$ are of the form 
$(n/q,\D^\b)$, where $(n,\D)\in \cM_\varphi$, $q=\gcd(n,b/g)$ and 
$\b=b/gq$. 
Now $g|c_{i,j,k}$ and $c_{i,j,k}|\a_{i,j}$ and $c_{i,j,k}|\b_{i,k}$. Therefore 
 $\gcd(m_{i,j},a/g)=\gcd(m_{i,j},a)=d_{i,j}$ and similarly $\gcd(n_{i,k},b/g)=e_{i,k}$. 
Thus $g$ is coprime to 
\[p_i=\frac{m_{i,j}}{\gcd(m_{i,j},a/g)}=\frac{n_{i,k}}{ \gcd(n_{i,k},b/g)},\] 
for all $i,j,k$. From Lemma \ref{AJDLEMMAGCD}, it follows that $\psi^{a/g}\sim \varphi^{b/g}$. 

Now 
\[a/g=\a_{i,j}d_{i,j}/g=c_{i,j,k}t_{i,j,k}d_{i,j}/g=f_{i,j,k}t_{i,j,k}d_{i,j}\] and similarly
\[b/g=f_{i,j,k}s_{i,j,k}e_{i,k},\] 
for all $i,j,k$. Also 
\[\gcd(\{f_{i,j,k}|1\le i \le M,1\le j\le M_i, 1\le k \le N_i\})=1\] 
so, for fixed $i',j',k'$, 
\[f_{i',j',k'} \,\left|\, \middle(\prod_{i,j,k}(t_{i,j,k}d_{i,j})\right)/t_{i',j',k'}d_{i',j'}.\]

\end{proof}

\begin{corollary}\label{AJDCOROLLARYPC} The power conjugacy problem for regular infinite elements of $G_{n,r}$ 
is solvable. 
\end{corollary}
\begin{proof}
Let $\psi$ and $\varphi$ be regular infinite elements of $G_{n,r}$. Suppose that $\psi^a$ is 
conjugate to $\varphi^b$, for some non-zero $a,b$. Replacing either $\psi$ or $\varphi$ or both by their inverse, 
we may assume that $a, b>0$. Then, in the notation of the proposition above, we have
$\psi^\a\sim \varphi^\b$, where $\a=f_{i,j,k}t_{i,j,k}d_{i,j}$ and $\b=f_{i,j,k}s_{i,j,k}e_{i,k}$. From the 
conclusion of the theorem it is clear that there are finitely many choices for $f_{i,j,k}$, $s_{i,j,k}$, 
$t_{i,j,k}$, $d_{i,j}$ and $e_{i,k}$. Hence there are finitely many possible $\a$ and $\b$, and we 
may effectively construct a list of all possible pairs $(\a,\b)$. Having constructed this list we may 
check whether or not $\psi^\a\sim \varphi^\b$, using Algorithm \ref{alg:reginf}. Hence 
we may decide whether or not there exist $a,b$ such that $\psi^a\sim \varphi^b$.
\end{proof}

The proof of Proposition \ref{divisors} forms the basis for the algorithm for the power conjugacy 
problem. Given regular infinite elements $\psi,\phi\in G_{n,r}$ we construct bounds $\hat a$ and $\hat b$ such that if 
some (positive) power of $\psi$ is conjugate to a (positive) power of $\phi$ then 
$\psi^c\sim \phi^d$, for $0<c\le \hat a$ and $0<d\le \hat b$. Following the proof of the 
proposition, if $\psi^a\sim \phi^b$, for some $a,b>0$, then the inverse images 
$\hat\psi_a$ and $\hat\phi_b$ partition $\cM_\psi$ and $\cM_\phi$, so we have integers
$L, M_i, N_i$ such that 
\[\cM_\psi=\cup_{i=1}^L\{(m_{i,j},\G_{i,j})  \mid  1\le j\le M_i\}\]
and
\[\cM_\phi=\cup_{i=1}^L\{(n_{i,k},\D_{i,k})  \mid  1\le  k\le N_i\}.\]

Given any $\G\in A^*$ there exists unique $\L\in A^*$ and $r\in \NN$ such that 
$\G=\L^r$ and if $\G=\L'^s$ then $s\le r$. We denote  $\L$ by $\sqrt{\G}$ and $r$ by
$m(\G)$. 
From equations \eqref{eq:div1} and \eqref{eq:div2}, it follows that 
\[\sqrt{\L_{i,j,k}}=\sqrt{\G_{i,j}}=\sqrt{\G_i}=\sqrt{\D_i}=\sqrt{\D_{i,k}}\]
and
\[s_{i,j,k}\le m(\G_{i,j})\textrm{ and } t_{i,j,k}\le m(\D_{i,k}),\]
for $1\le i\le L$, $1\le j\le M_i$ and $1\le k\le N_i$. 

From Proposition \ref{divisors} we have $\a=d_{1,1}f_{1,1,1}t_{1,1,1}$ and 
$f_{1,1,1}\le \prod_{(i,j,k)\neq (1,1,1)}d_{i,j}t_{i,j,k}$. As $d_{i,j}\le |m_{i,j}|$ and 
$t_{i,j,k}\le  m(\D_{i,k})$, this means that 
\begin{align}
\a &\le \prod_{i=1}^L \prod_{j=1}^{M_i} \prod_{k=1}^{N_i}  d_{i,j}     t_{i,j,k}                         \notag\\
   &\le \prod_{i=1}^L \prod_{j=1}^{M_i} \prod_{k=1}^{N_i} |m_{i,j}| \, m(\D_{i,k})                       \notag\\
   &\le \prod_{i=1}^L \prod_{j=1}^{M_i} \left( |m_{i,j}|^{N_i} \prod_{k=1}^{N_i}  m(\D_{i,k}) \right) \notag\\
   &\le \prod_{i=1}^L \left( \prod_{j=1}^{M_i} |m_{i,j}|   \right)^{N_i}
                      \left( \prod_{k=1}^{N_i} m(\D_{i,k}) \right)^{M_i}.\label{ag:abnd}
\end{align}
Similarly 
\begin{equation}
\b \le \prod_{i=1}^L\left[\left(\prod_{k=1}^{N_i} |n_{i,k}|\right)^{M_i}\left( \prod_{j=1}^{M_i} m(\G_{i,j}) \right)^{N_i}\right].\label{eq:bbnd}
\end{equation}
Now suppose that a solution $\psi^{a'}\sim \phi^{b'}$ gives rise to  sub-partitions of the partitions of  $\cM_\psi$ and  
$\cM_\phi$ above.
Straightforward calculation shows that in this case, the bounds on $\a$ and $\b$ obtained are again less than or equal 
to the right hand sides of \eqref{ag:abnd} and \eqref{eq:bbnd} (calculated using the original partitions).
Thus, in computing (upper) bounds $\hat a$ and $\hat b$ we may take partitions of $\cM_\psi=P_1\cup \cdots\cup P_L$ and $\cM_\phi=Q_1\cup \cdots \cup Q_L$ with
$L$ as small as possible, subject to the constraint that, for each $i$ such that 
$1\le i\le L$ we have $\sqrt{\G}=\sqrt{\D}$, for 
all $(m,\G)\in P_i$ and $(n,\D)\in Q_i$. 
If these partitions satisfy these properties, and this does not hold for any partition of fewer than $L$ subsets, 
(in other words the partitions are formed by gathering together characteristics with the same root) then the 
bounds $\hat a$ and $\hat b$ are given by 
\begin{equation}\label{eq:acomp}
\hat a=\prod_{i=1}^L\left[\left(\prod_{(m,\G)\in P_i}|m|\right)^{|Q_i|}\left(\prod_{(n,\D)\in Q_i}m(\D)\right)^{|P_i|}\right]
\end{equation}
and
\begin{equation}\label{eq:bcomp}
\hat b=\prod_{i=1}^L\left[\left(\prod_{(n,\D)\in Q_i}|n|\right)^{|P_i|}\left(\prod_{(m,\G)\in P_i}m(\G)\right)^{|Q_i|}\right].
\end{equation}

\begin{example} \label{ex:pc1}
Let $n=2$ and $r=1$ and $V_{2,1}$ be free on $\{x\}$. 
Let $\psi$ be the regular infinite element of $G_{2,1}$ of Examples \ref{snf0} and \ref{ex:snf1}. Then $\psi$  is 
in quasi-normal form with respect to the $A$-basis $X=\{x\a_1,x\a_2\}$ and  
$\cM_\psi=\{(1,\a_2),(\text{-}1,\a_1)\}$.

Let $\varphi$ be the regular infinite element of $G_{2,1}$ defined by a bijective map from 
\[Y_\phi=\{x\a_1,x\a_2\a_1^3,x\a_2\a_1^2\a_2,x\a_2\a_1\a_2,x\a_2^2\}\]
to 
\[Z_\phi=\{x\a_1^2,x\a_1\a_2\a_1,x\a_1\a_2^2\a_1,x\a_1\a_2^3,x\a_2\}\]
given as illustrated below.

\begin{center}
$\phi:$
\begin{minipage}[b]{0.5\linewidth}
\centering
 \Tree [. 1 [  [. [  2 3   ] [.4  ] ]. [.5  ] ].  ]  \quad $\longrightarrow$ \Tree [ [ [.3  ][ [.4  ] [  [.5  ] [.1  ] ] ]  ]   [.2 ] ]  \end{minipage}
\end{center}

Then $Y_\phi$ is the minimal expansion of $\{x\}$ associated to $\phi$ and 
the minimal expansion of $\{x\}$ contained in $Y_\phi\mA\cup Z_\phi\mA$ is $X$. We have 
$X\mA\setminus Y_\phi\mA=\{x\a_2,x\a_{2}\a_1,x\a_{2}\a_1^2\}$ and  $X\mA\setminus Z_\phi\mA=\{x\a_1,x\a_{1}\a_2,x\a_{1}\a_2^2\}$
The $X$-components of these elements are:
\[ 
x\a_1 \mapsto x\a_1\a_2^3\mapsto  x\a_1\a_2^6\mapsto \cdots
\]
\[ 
x\a_1\a_2 \mapsto x\a_1\a_2^4\mapsto  x\a_1\a_2^7\mapsto \cdots
\]
\[ 
x\a_1\a_2^2 \mapsto x\a_1\a_2^5\mapsto  x\a_1\a_2^8\mapsto \cdots
\]
\[
 \cdots\mapsto x\a_2\a_1^6 \mapsto  x\a_2\a_1^3 \mapsto  x\a_2
\]
\[
 \cdots\mapsto x\a_2\a_1^7 \mapsto  x\a_2\a_1^4 \mapsto  x\a_2\a_1
\]
\[
 \cdots\mapsto x\a_2\a_1^8 \mapsto  x\a_2\a_1^5 \mapsto  x\a_2\a_1^2
\]
so $\phi$ is in quasi-normal form with respect to $X$ and $\cM_\varphi=\{(1,\a_2^3),(\text{-}1,\a_1^3)\}$.
In the notation above, we have partitions $\cM_\psi=P_1\cup P_2$ and $\cM_\phi=Q_1\cup Q_2$ with $P_1=\{(1,\a_2)\}$, $P_2=\{,(\text{-}1,\a_1)\}$, $Q_1=\{(1,\a_2^3)\}$ and $Q_2=\{(\text{-}1,\a_1^3)\}$,
so we obtain bounds $\hat a=9$ and $\hat b=1$.  

 Assume there exists positive integers $a,b$ such that $\psi^a\sim \varphi^b$. We may now assume that $a\le 9$ and $b=1$. 
The map $\hsa: \Ms \to \Msa$  is given by 
\[\hsa(1,\a_2)=(1/d_1,\a_2^{a/d_1})\textrm{, }\hsa(-1,\a_1)=(-1/d_2,\a_1^{a/d_2}),\]
where $d_1=\gcd(1,a)=1$ and $d_2=\gcd(\text{-1},a)=1$.
Thus 
\[\Msa=\{(1,\a_2^{a}),(\text{-}1,\a_1^{a})\}.\]
The only possible choice for $a$ making $\Msa=\Mhb=\cM_\phi$ is $a=3$. 
 Applying Algorithm \ref{alg:reginf}   to  $\psi^3$ and $\varphi$ 
we find a conjugating element $\rho$, given by $x\a_1\rho= x\a_2$ and $x\a_2\rho=x\a_1$. 
\end{example}

\begin{remark} \label{rem:negative_powers}
In Corollary~\ref{AJDCOROLLARYPC}  the powers $a$ and $b$ were positive, giving us upper bounds $a \leq \hat a$ and $b \leq \hat b$ for the minimal powers which solve the power conjugacy problem. Now suppose that $a < 0$ and $b > 0$. We may write $\psi^a = (\psi^{-1})^{-a}$ and then $-a > 0$. If we apply Corollary~\ref{AJDCOROLLARYPC} to $(\psi^{-1}, \phi)$, we obtain a second pair of bounds $-a \leq \bar a$ and $b \leq \bar b$. Observing that $(m, \G) \in \cM_\psi$ if and only if $(-m, \G) \in \cM_{\psi^{-1}}$, we note that this replacement $\psi \mapsto \psi^{-1}$ preserves the absolute value $|m|$ of all characteristic multipliers. Thus each of the terms $|m_{i,j}|$, $|n_{i,k}|$, $|m|$ and $|n|$ in equations~(\ref{ag:abnd}--\ref{eq:bcomp}) is unchanged. We conclude that $\bar a = \hat a$ and $\bar b = \hat b$.

The same argument applies equally well to the remaining two cases $a > 0$, $b < 0$ and $a < 0$, $b < 0$. Thus, once we have obtained $\hat a$ and $\hat b$, we need only to check the ranges $1 \leq |a| \leq \hat a$ and $1 \leq |b| \leq \hat b$ to find minimal conjugating powers.
\end{remark}

\begin{example}\label{ex:pc2}
Let $\psi$ be as in Example \ref{ex:pc1} and let $\phi$ be as in Example \ref{pcExample1}. 
Then $\cM_\psi=\{(1,\a_2),(\text{-}1,\a_1)\}$ and $\cM_\phi=\{(\text{-}2,\a_1),(1,\a_2)\}$.  
In the notation above, we have partitions $\cM_\psi=P_1\cup P_2$ and $\cM_\phi=Q_1\cup Q_2$ 
with $P_1=\{(1,\a_2)\}$, 
$P_2=\{(\text{-}1,\a_1)\}$, 
$Q_1=\{(1,\a_2)\}$ and 
$Q_2=\{(\text{-}2,\a_1)\}$,
so we obtain bounds $\hat a=1$ and $\hat b=2$.

Assume there exist positive integers $a,b$ such that $\psi^a\sim\phi^b $; with $a=1$ and $b\le 2$.  
The map
$\hhb: \Mh \to \Mhb$ is given by
\[\hhb(1,\a_2)=(1/d_1,\a_2^{b/d_1}),\;
\hhb(\text{-}2,\a_1)=(\text{-}2/d_2,\a_1^{b/d_2}),\]
where $d_1=\gcd(1,b)=1$ and $d_2=\gcd(\text{-2},b)=b$. 
Thus, 
\[\Mhb=\{(1,\a_2),(\text{-}2,\a_1)\}
\quad\text{or}\quad
\{(1,\a_2^2),(\text{-}1,\a_1)\}.\]
As $\cM_\psi\neq\Mhb$, for $b=1$ and $b=2$, there is no pair of positive integers $a,b$ such that $\psi^a\sim \phi^b$. 
 The same argument applies on replacing $\phi$ or $\psi$ by $\phi^{-1}$ or $\psi^{-1}$, respectively, so no nontrivial power of $\phi$ is conjugate to a power of $\psi$.
\end{example}
In order to solve the power conjugacy problem for general regular infinite elements of $G_{n,r}$ we require an algorithm which finds all pairs $(a,b)$, within the 
bounds calculated, rather than merely deciding whether or not such a pair exists. This is the algorithm we describe here. It constructs 
a set $\cPC_{RI}$ consisting of triples $(a,b,\rho)$, such that $\rho^{-1}\psi^a\rho=\phi^b$.
\begin{algorithm}\label{alg:pcRI}
Let $\psi$ and $\phi$ be regular infinite elements of $G_{n,r}$.
 \be[\bfseries Step 1:]
\item Construct $A$-bases $X_\psi$ and $X_\phi$ with respect to which $\psi$ and $\phi$ are in quasi-normal form (Lemma \ref{lem:qnf}).
\item Construct   the sets $\cM_\psi$ and $\cM_\varphi$ (see Definition \ref{setMultipliers}). 
\item Calculate the bounds on $\hat a$ and $\hat b$, using equations \eqref{eq:acomp} and \eqref{eq:bcomp}.
\item For all pairs $a,b$ such that $1\le |a| \le \hat a$ and $1\le |b| \le \hat b$, input $\psi^a$ and $\phi^b$ to Algorithm \ref{alg:reginf}. 
If a conjugating automorphism $\rho$ is returned, add $(a, b,\rho)$ to the set  $\cPC_{RI}$. 

\item If $\cPC_{RI}=\emptyset$, output ``No'' and halt. Otherwise output $\cPC_{RI}$. 
\ee
\end{algorithm}
%
 Corollary \ref{AJDCOROLLARYPC} may be strengthened. 
\begin{corollary}\label{cor:pcris}
Given regular infinite elements $\psi, \phi\in G_{n,r}$ there is a finite subset $\cPC_{RI}$ of $\ZZ\times \ZZ\times G_{n,r}$, 
which may be effectively constructed,
such that   $\psi^a\sim \phi^b$ if and only if $a=cg$ and $b=dg$, for some $(c,d,\rho)\in \cPC_{RI}$ and $g\in \ZZ$. Moreover,
for all $(c,d,\rho)\in \cPC_{RI}$ and $g\in \ZZ$, we have  $\rho^{-1}\psi^{cg}\rho=\phi^{dg}$.
\end{corollary}
\begin{proof}
From Lemma \ref{divisors} and the description of Algorithm \ref{alg:pcRI}, $\cPC_{RI}$ is a finite set and 
it follows that if $\psi^a\sim \psi^b$, for some positive $a,b\in\ZZ$, then $(a/g,b/g,\rho)\in \cPC_{RI}$ and in this
case  $\rho^{-1}\psi^a\rho=\phi^b$. 
Replacing one or other, 
or both, of $\psi$ and $\phi$ by their inverses the same holds, without the constraint that $a,b$ be positive. On the other hand if $(c,d,\rho)$ is in 
$\cPC_{RI}$ then 
$\rho^{-1}\psi^{c}\rho=\phi^{d}$, 
so 
$\rho^{-1}\psi^{cg}\rho=\phi^{dg}$, 
for all $g\in\ZZ$.
\end{proof}
\subsection{The power conjugacy algorithm}\label{powerconjAlgorithmSec}
We combine the algorithms of Sections \ref{torPower} and \ref{regPower} to give an algorithm for the power conjugacy problem in $G_{n,r}$. 
In fact in Sections \ref{torPower} and \ref{regPower} we find a description of all solutions of the power conjugacy 
problem for periodic and regular infinite automorphisms, respectively: and the  algorithm 
in this section does the same for arbitrary elements of $G_{n,r}$.
 
If we are only interested in the existence of a solution to the power conjugacy problem then we may 
essentially ignore the periodic part of automorphisms, as long as the regular infinite part is non-trivial. To see this, 
suppose $\psi$ and $\phi$ are elements of $G_{n,r}$ and we have decompositions 
 $\psi=\psi_{P}*\psi_{RI}$, $\phi=\phi_P*\phi_{RI}$. Assume that we have found that  $V_{RI,\psi}$ is non-trivial 
and $\psi_{RI}^a$ is conjugate to $\phi_{RI}^b$, $a,b\neq 0$. In this case,
$\psi_P$ and $\phi_P$ have finite orders, $m$ and $k$ say, and so we immediately have a solution $\psi^{amk}\sim \phi^{bmk}$, $amk,bmk\neq 0$, of the
power conjugacy problem. 
 The algorithm described below allows this type of solution but also tries to find a
 solution to the power conjugacy problem corresponding to each pair $(c,d)$ such that $\psi_P^c\sim \phi_P^d$. 
Thus, in Theorem \ref{PCG}, we obtain a description of all solutions
to the power conjugacy problem, for $\psi$ and $\phi$. (That is, all pairs $(a,b)$ such that $\psi^a\sim \phi^b$. We do not 
find all possible conjugators $\rho$.)
\begin{algorithm}\label{powerconjAlgorithm}
Let $\psi$ and $\phi$ be elements of $G_{n,r}$.
 \be[\bfseries Step 1:]
\item Run Steps \ref{it:ca1}, \ref{it:ca2} and \ref{it:ca3} of Algorithm \ref{conjAlgorithm}.
\item Input $\psi_{RI}$ and $\phi_{RI}$ to Algorithm \ref{alg:pcRI}. 
\item If $X_{RI,\psi}$ is non-empty (that is, $V_{RI,\psi}$ is non-empty) and $\mathcal{PC}_{RI}$ is empty, output ``No'' and stop.
\item\label{it:pca2}  Compute the orders $k$ and $m$ of $\psi_P$ and $\phi_P$. Input $\psi_P^a$ and $\phi_P^b$ to Algorithm \ref{alg:periodic}, for all
$c,d$ such that $1\le c\le k$ and $1\le d\le m$.  Construct the set
 $\mathcal{PC}_P$ of all triples $(c,d,\rho)$ found such that 
$\rho^{-1}\psi^{c}\rho$ is conjugate to $\varphi^{d}$.  If $X_{RI,\psi}$ is non-empty, 
adjoin the triple $(0,0,\theta_0)$ to $\cPC_P$, where $\theta_0$ is the identity
map of the algebra $V_{n,s_P}$, of Step \ref{it:ca3} of Algorithm \ref{conjAlgorithm}.
\item If $\mathcal{PC}_P$ is empty, output ``No'' and stop. If $\mathcal{PC}_P$ is non-empty and $X_{RI,\psi}$ is empty output 
$\cPC_P$ and stop. 
\item If this step is reached then both $\cPC_P$ and $\cPC_{RI}$ are non-empty.  
For all $(\a,\b, \rho_{RI})$ in $\mathcal{PC}_{RI}$ and all pairs $(c,d,\rho_P)$ in $\mathcal{PC}_P$ 
consider the simultaneous congruences  
\[\a x\equiv c \mod k\textrm{ and } \b x\equiv d\mod m,\]
where $k$ and $m$ are the orders of $\psi_P$ and $\phi_P$ found in Step \ref{it:pca2}. 
For each  positive solution $x=g$ (less than lcm$(k,m)$)  
add  $(\a g,\b g,g,\rho_P\ast \rho_{RI})$ to the set $\cPC$ (which is empty at the start). 
\ee
\end{algorithm}
We verify that this algorithm solves the power conjugacy problem in the proof of the following theorem.
\begin{theorem}\label{PCG}
The power conjugacy problem for the Higman-Thompson group $G_{n,r}$ is solvable. Furthermore, given 
elements $\psi, \phi\in G_{n,r}$, let $\psi_P$ have order $k$, let $\phi_P$ have order $m$ and let $l=\textrm{lcm}(k,m)$. There is a finite subset
$\cPC\subseteq \ZZ^3\times G_{n,r}$, which may be effectively constructed, such that 
$\psi^a\sim \phi^b$ if and only if $(ag/h,bg/h,g,\rho)\in \cPC$, 
where $\rho\in G_{n,r}$ and $g,h\in\ZZ$ such that
$h\equiv g\mod l$, $h|a$ and $h|b$. In this case $\rho^{-1}\psi^{a}\rho=\phi^b$.  
\end{theorem}

\begin{proof}
Apply  Algorithm \ref{powerconjAlgorithm}  to  $\psi$ and $\phi$. 
If there exist $a,b\in\ZZ$ such that $\psi^a\sim \phi^b$ then $\psi_{P}^a\sim  \phi_{P}^b$ and $\psi_{RI}^a\sim  \phi_{RI}^b$. 
In this case let  $\psi_{P}$ and  $\phi_{P}$ have orders
$k$ and $m$, respectively and let $a_1,b_1\in\ZZ$ be such that $1\le a_1<k$ and $1\le b_1<m$ and 
$a_1\equiv a\mod k$,  $b_1\equiv b\mod m$. Then there exists $\rho_P$ such that $(a_1,b_1,\rho_P)\in \cPC_I$.
Furthermore, from  Corollary \ref{cor:pcris}, 
 there exists  $(a_2,b_2,\rho_{RI})\in \cPC_{RI}$ and $h\in \ZZ$ such that $a=a_2h$ and $b=b_2h$. Let $g$ be such that 
$1\le g< \textrm{lcm}(k,m)$, and  $g\equiv h \mod \textrm{lcm}(k,m)$ so $g\equiv h \mod k$ and $g\equiv h\mod m$. As $h$ is a solution to the 
congruences $a_2x\equiv a_1\mod k$ and $b_2x\equiv b_1 \mod m$, it follows that $g$ is also a solution to these congruences. Therefore
 $(a_2g,b_2g,g,\rho_P\ast\rho_{RI})\in \cPC$. As $a_2=a/h$ and $b_2=b/h$,  this is an element of $\ZZ^3\times G_{n,r}$ of the required form. 

Conversely, assume $(u,v,g,\rho_P\ast\rho_{RI})\in \cPC$, where $u=ag/h$ and $v=bg/h$, 
for some $a,h\in \ZZ$ satisfying the hypotheses of the theorem. 
Then there exist $(\a,\b, \rho_{RI})$ in $\mathcal{PC}_{RI}$ and  
$(c,d,\rho_P)$ in $\mathcal{PC}_P$ such that $u=\a g\equiv c\mod k$ and $v=\b g\equiv d\mod m$. 
As $g\equiv h\mod l$ this implies that $a=(u/g)h=\a h\equiv c\mod k$ and $b=(v/g)h=\b h\equiv d\mod m$. 
Therefore 
$\psi_P^a=\psi_P^c\sim \phi_P^d=\phi_P^b$, by definition of $\cPC_P$, and indeed $\rho_P^{-1}\psi_P^a\rho_P=\phi_P^b$.  
Also, $a=\a h$ and $b=\b h$ implies  $\rho_{RI}^{-1}\psi_{RI}^a\rho_{RI}=\phi_{RI}^b$, by Corollary \ref{cor:pcris}, 
so \[\psi^a=(\psi_P\ast \psi_{RI})^a=\psi_P^a\ast \psi_{RI}^a\sim \phi_P^b\ast \phi_{RI}^b=(\phi_P\ast \phi_{RI})^b=\phi^b\] and 
$\rho_P\ast\rho_{RI}$ is a conjugating element.
\end{proof}

Examples which illustrate how the algorithm works on automorphisms which are not necessarily periodic or regular infinite 
can be found at \cite{DMR}: follow the link to ``Examples'' and 
refer to the examples named ``\texttt{mixed\_pconj\_phi}'' and ``\texttt{mixed\_pconj\_psi}''. 

\bibliographystyle{alpha}

\begin{thebibliography}{00}
\bibitem
{AS74} M.~Anshel and P.~Stebe, ``The solvability of the conjugacy problem for certain HNN groups'', {\it Bull. Amer. Math. Soc.}, {\bf 80} (2) (1974) 266--270.

\bibitem
{Bar14} N.~Barker, ``Topics in Algebra:
The Higman-Thompson Group $G_{2,1}$ 
and Beauville $p$-groups'', \emph{Thesis, Newcastle University} (2014)
\bibitem
{BM}J.~M.~Belk and F.~Matucci, ``Conjugacy and dynamics in Thompson's groups'', \emph{Geom. Dedicata} \textbf{169} (1) (2014) 239--261.

\bibitem
{Bez14} N.V.~Bezverkhnii,
``Ring Diagrams with Periodic Labels and Power Conjugacy Problem in Groups with Small Cancellation Conditions C (3) -T (6)'', 
\emph{Science and Education of the Bauman MSTU}, \textbf{14} (11) (2014).

\bibitem
{BK08}
 V.N.~Bezverkhni{\v\i}, A.N.~Kuznetsova,
``Solvability of the power conjugacy problem for words in Artin groups of extra large type'',
\emph{Chebyshevskii Sb.} \textbf{9} (1)
(2008)  50--68.

\bibitem
{Bleaketal} C.~Bleak, H.~Bowman, A.~Gordon, G.~Graham, J.~Hughes, F.~Matucci and J.~Sapir, ``Centralizers in R.Thompson's group $V_{n}$'', \emph{Groups, Geometry and Dynamics} \textbf{7}, No. 4 (2013), 821--865.

\bibitem
{BMOV} O.~Bogopolski, A.~Martino, O.~Maslakova and E.~Ventura, 
``The conjugacy problem is solvable in free-by-cyclic groups'',
\emph{Bulletin of the London Mathematical Society},
\textbf{38},
(10) (2006) 787--794.


\bibitem
{B} M.~G.~Brin, ``Higher dimensional Thompson groups'', {\em Geom. Dedicata}, {\bf 108} (2004) 163--192.

\bibitem
{Br} K.~S.~Brown, ``Finiteness properties of groups'', Journal of Pure and Applied Algebra, {\bf 44} (1987) 45-75.

\bibitem
{BCR} J.~Burillo, S.~Cleary and C.~E.~R\"over, ``Obstructions for subgroups of Thompson's group $V$'', \url{arxiv.org/abs/1402.3860}

\bibitem
{CFP} J.~W.~Cannon,~W.J.~Floyd and W.~R.~Parry, ``Introductory notes on Richard Thompson's groups'', {\em Enseign. Math.}, (2) {\bf 42}(3--4) (1996) 215--256. 


\bibitem
{Cohn} P.~M.~Cohn, ``Universal Algebra''. Mathematics and its Applications, 6, D. Reidel Pub. Company, (1981).

\bibitem
{cohnAlgebra3} P.~M.~Cohn, ``Algebra, Volume 3''. J. Wiley, (1991).

\bibitem
{C77} L. P. J. Comerford, A note on power-conjugacy, {\em Houston J. Math.} 3 (1977), no. 3, 337---341.

\bibitem
{DMP} W.~Dicks, C.~Martinez-P\'{e}rez, ``Isomorphisms of Brin-Higman-Thompson groups'', \emph{Israel Journal of Mathematics}, \textbf{199} (2014), 189--218.


\bibitem
{Fes14} A.V.~Fesenko, ``Vulnerability of Cryptographic Primitives Based on the Power Conjugacy Search Problem in Quantum Computing'', 
\emph{Cybernetics and Systems Analysis},
 \textbf{50} (5)
(2014)
815--816.


\bibitem
{Higg} G.~Higman, ``Finitely presented infinite simple groups'',  {\em Notes on Pure Mathematics}, Vol.~8 (1974). 


\bibitem
{JT} B.~J\'{o}nsson and A.~Tarski, ``On two properties of free algebras'', {\em Math. Scand.}, {\bf 9} (1961) 95--101.

\bibitem
{KA09}
D.~Kahrobaei and 
M.~Anshel,
``Decision and Search in Non-Abelian
Cramer-Shoup Public Key Cryptosystem'', 
\emph{Groups-Complexity-Cryptology},
\textbf{1} (2) (2009) 217-–225.

\bibitem
{LM71} S.~Lipschutz and  C.F.~Miller, ``Groups with certain solvable and unsolvable decision problems'', {\em Comm. Pure Appl. Math.}, \textbf{24}
(1971) 7--15. 

\bibitem
{Lothaire} M.~Lothaire, ``Combinatorics on Words'', Addison-Wesley, Advanced Book Program, World Science Division, (1983).

\bibitem
{MPN} C.~Martinez-Perez, B.~Nucinkis, ``Bredon cohomological finiteness conditions for generalisations of Thompson's groups'', {\em Groups Geom. Dyn.} {\bf 7} (4) (2013) 931--959.

\bibitem
{MT} R.~McKenzie and R.~J.~Thompson, ``An elementary construction of unsolvable word problems in group theory'', Word problems: decision problems and the Burnside problem in group theory, Studies in Logic and the Foundations of Math., 71, pp. 457--478. North-Holland, Amsterdam, (1973).

\bibitem
{P} E.~Pardo, ``The isomorphism problem for Higman-Thompson groups'', \emph{Journal of Algebra}, \textbf{344} (2011), 172--183.

\bibitem
{P08} S.J.~Pride, ``On the residual finiteness and other properties of (relative) one-relator groups'', {\em Proc. Amer. Math. Soc.} {\bf136} (2) (2008) 377--386.

\bibitem
{DMR} D.~M.~Robertson, ``{\tt thompson}: a package for Python 3.3+ to work with elements of the Higman-Thompson groups $G_{n,r}$''. Source code available from \url{https://github.com/DMRobertson/thompsons\_v} and documentation available from \url{http://thompsons-v.readthedocs.org/}.

\bibitem
{Olga}O.~P.~Salazar-Diaz, ``Thompson's group V from a dynamical viewpoint'', {\em Internat. J. Algebra Comput.}, 1, 39--70, 20, (2010).

\bibitem
{RT} R.~J.~Thompson, unpublished notes.\\ \url{http://www.math.binghamton.edu/matt/thompson/index.html}
\end{thebibliography}

\end{document}